\documentclass[11pt]{amsart}

\usepackage{amssymb,amsfonts,amsmath,verbatim, enumerate}
\usepackage{mathtools}
\usepackage[dvips]{graphicx}
\usepackage{psfrag}
\usepackage{booktabs}

\newtheorem{theo}{Theorem}[section]
\newtheorem{prop}[theo]{Proposition}
\newtheorem{lemma}[theo]{Lemma} 
\newtheorem{coro}[theo]{Corollary}
\newtheorem{claim}[theo]{Claim}

\newtheorem{fact}[theo]{Fact} 
\newtheorem{setup}[theo]{Setup}
\newtheorem{defn}[theo]{Definition}

\newcommand{\mc}[1]{\mathcal{#1}}

\newcommand{\nib}[1]{\noindent {\bf #1}} 
\newcommand{\Qart}{\mathcal{Q}}
\newcommand{\Part}{\mathcal{P}}
\newcommand{\F}{\mathcal F}
\newcommand{\V}{\mathcal V}
\newcommand{\Sa}{\mathcal S}
\newcommand{\R}{\mathcal R}
\newcommand{\D}{\mathcal D} 
\newcommand{\vx}{\mathrm{vertex}} 
\newcommand{\T}{\mathcal T}
\newcommand{\K}{\mathcal K}
\newcommand{\A}{\mathcal A}
\newcommand{\Cell}{\mathrm{Cell}}
\newcommand{\B}{\mathcal B}
\newcommand{\U}{\mathcal U}
\newcommand{\Lo}{\mathcal L}
\newcommand{\Akpq}{\mathcal A^k_{p,q}}
\newcommand{\sm}{\setminus}
\newcommand{\eps}{\varepsilon}
\newcommand{\Z}{\mathbb Z}
\newcommand{\vb}{{\bf v}}
\newcommand{\yb}{{\bf y}}
\newcommand{\xb}{{\bf x}}
\newcommand{\ub}{{\bf u}}
\newcommand{\ib}{{\bf i}}

\def\noproof{{\unskip\nobreak\hfill\penalty50\hskip2em\hbox{}\nobreak\hfill%
       $\square$\parfillskip=0pt\finalhyphendemerits=0\par}\goodbreak}
\def\endproof{\noproof\bigskip}

\title{Packing $k$-partite $k$-uniform hypergraphs} 
\author{Richard Mycroft}

\begin{document}
\begin{abstract}
Let $G$ and $H$ be $k$-graphs ($k$-uniform hypergraphs); then a perfect $H$-packing in~$G$ is a collection of vertex-disjoint copies of $H$ in~$G$ which together cover every vertex of~$G$. For any fixed $H$ let $\delta(H, n)$ be the minimum $\delta$ such that any $k$-graph~$G$ on $n$ vertices with minimum codegree $\delta(G) \geq \delta$ contains a perfect $H$-packing. The problem of determining $\delta(H, n)$ has been widely studied for graphs (\emph{i.e.} $2$-graphs), but little is known for $k \geq 3$. Here we determine the asymptotic value of $\delta(H, n)$ for all complete $k$-partite $k$-graphs $H$, as well as a wide class of other $k$-partite $k$-graphs. In particular, these results provide an asymptotic solution to a question of R\"odl and Ruci\'nski on the value of $\delta(H, n)$ when $H$ is a loose cycle. We also determine asymptotically the codegree threshold needed to guarantee an $H$-packing covering all but a constant number of vertices of $G$ for any complete $k$-partite $k$-graph $H$.
\end{abstract}
\maketitle 

\section{Introduction}\label{intro}

\subsection{Basic notions}
A \emph{$k$-uniform hypergraph}, or \emph{$k$-graph}~$H$ consists of a \emph{vertex set}~$V(H)$ and an \emph{edge set}~$E(H)$, where every $e \in E(H)$ is a set of precisely~$k$ vertices of~$H$. So a $2$-graph is a simple graph. We often identify $H$ with its edge set, for example writing $e \in H$ to mean that $e$ is an edge of $H$, or $|H|$ to denote the number of edges of~$H$.

If $G$ and $H$ are $k$-graphs, then an \emph{$H$-packing} in $G$ (also known as an \emph{$H$-tiling} or \emph{$H$-matching}) is a collection of vertex-disjoint copies of $H$ in $G$. This is a generalisation of a \emph{matching} in $G$, which is the case of an $H$-packing when~$H$ is the $k$-graph with $k$ vertices and one edge. We say that a matching or $H$-packing in $G$ is \emph{perfect} if it covers every vertex of $G$. Clearly $G$ can only contain a perfect $H$-packing if $|V(H)|$ divides $|V(G)|$.

We focus mainly on the case when $H$ is a fixed $k$-graph and $|V(G)|$ is much larger than $|V(H)|$.
Our general question is then: what minimum degree conditions on $G$ are sufficient to guarantee that $G$ contains a perfect $H$-packing? 
There are various notions of minimum degree for $k$-graphs, but we shall consider here only one, namely the \emph{codegree}. Let $G$ be a $k$-graph on~$n$ vertices. For any set $S \subseteq V(G)$, the \emph{degree} $\deg_G(S)$ of $S$ is the number of edges of $G$ which contain $S$ as a subset. The \emph{minimum codegree} $\delta(G)$ of $G$ is then the minimum of $\deg_G(S)$ taken over all sets $S$ of $k-1$ vertices of $G$. 
Note that this coincides with the standard notion of degree for graphs.

For any fixed $k$-graph $H$ and any integer $n$ we define $\delta(H, n)$ to be the smallest integer $\delta$ such that any $k$-graph $G$ on $n$ vertices with minimum codegree $\delta(G) \geq \delta$ contains a perfect $H$-packing. As noted earlier, this is only defined for those $n$ which are divisible by $|V(H)|$; we shall only consider these values of $n$. A major problem in extremal graph theory is to determine the behaviour of $\delta(H, n)$ for large $n$.

\subsection{Perfect packings in graphs}
In the case when $H$ is a graph, this question has been widely studied, and the value of $\delta(H, n)$ has been determined up to an additive constant in all cases. Indeed, the celebrated Hajnal-Szemer\'edi theorem~\cite{HSz} determined that $\delta(K_r, n) = (r-1)n/r$, and Koml\'os, S\'ark\"ozy and Szemer\'edi~\cite{KSS} showed that for any graph $H$ there exists a constant $C$ such that $\delta(H, n) \leq (1 - 1/\chi(H))n + C$. This confirmed a conjecture of Alon and Yuster~\cite{AY}, who had previously established the weaker result with $o(n)$ in place of $C$, and who observed that the constant $C$ cannot be removed completely. Finally K\"uhn and Osthus~\cite{KO2} determined the value of $\delta(H, n)$ up to an additive constant for any graph $H$ using the critical chromatic number~$\chi_{\textrm{cr}}(H)$ first introduced by Koml\'os~\cite{Km}. 

Since there are multiple similarities between the results of K\"uhn and Osthus for graphs~\cite{KO2} and the results of this paper for $k$-graphs, we shall state their results in some detail. Let $H$ be a graph on $m$ vertices, and let $\chi(H)$ denote the chromatic number of $H$, which we assume here to be greater than two (the behaviour in the bipartite case is somewhat different, but is less closely related to the $k$-graph results considered in this paper). So assume that $\chi(H) = r \geq 3$, and define $\D(H) := \bigcup_c \{|X_i^c| - |X_j^c| : i, j \in [r]\}$, where the union is taken over all proper $r$-colourings $c$ of $H$, and $X_1^c, \dots, X_r^c$ denote the colour classes of $c$. We then define $\gcd(H)$ to be the greatest common factor of $\D(H)$, unless $\D(H) = \{0\}$, in which case $\gcd(H)$ is undefined. Also, we define $\sigma(H) := \min_{c, j} |X_j^c|/m$, so $\sigma(H)$ is the smallest possible size of a colour class of a proper $r$-colouring of $H$, expressed as a proportion of the number of vertices of $H$. K\"uhn and Osthus~\cite{KO2} demonstrated that there exists a constant $C$ such that
$$ \left(1 - \frac{1}{\chi^*(H)} \right)n \leq  \delta(H, n) \leq \left(1 - \frac{1}{\chi^*(H)} \right)n+C,$$
where 
$$ 
\chi^*(H) = 
\begin{cases}
\chi_{\textrm{cr}}(H) := \frac{\chi(H) - 1}{1-\sigma(H)} & \textrm{ if $\gcd(H) = 1$,} \\
\chi(H) & \textrm{ otherwise.}
\end{cases}
$$

\subsection{Perfect packings in hypergraphs: known results}
For $k$-graphs $H$ with $k \geq 3$ much less is known. Indeed, the only cases for which $\delta(H, n)$ is known even asymptotically are the cases when $H$ is a 3-graph on 4 vertices and the case of a perfect matching (\emph{i.e.} when $H = K^k_k$ consists of $k$ vertices and one edge). The first bounds for the latter case were given by Daykin and H\"aggkvist~\cite{DH}, and later R\"odl, Ruci\'nski and Szemer\'edi~\cite{RRS3} showed that $n/2 - k \leq \delta(K^k_k, n) \leq n/2$ for all sufficiently large $n$ (indeed, they actually determined the exact value of $\delta(K^k_k, n)$ for large values of $n$). Beyond this, the value of $\delta(H, n)$ is known for three other $3$-graphs, all on four vertices. Let $K^3_4-2e$, $K^3_4-e$ and $K^3_4$ denote the $3$-graphs on $4$ vertices with $2, 3$ and $4$ edges respectively. The value of $\delta(K^3_4-2e, n)$ was found to be $n/4 + o(n)$ by K\"uhn and Osthus~\cite{KO}; recently Czygrinow, DeBiasio and Nagle~\cite{CDN} found the exact value for large $n$ to be either $n/4$ or $n/4+1$ according to the parity of $n/4$. Lo and Markstr\"om~\cite{LM, LM2} showed that $\delta(K^3_4-e, n) = n/2 + o(n)$ and that $\delta(K^3_4, n) = 3n/4 + o(n)$. Simultaneously with the latter, Keevash and Mycroft~\cite{KM} showed that the exact value of $\delta(K^3_4, n)$ for large $n$ is $3n/4-1$ or $3n/4-2$, again according to the parity of $n/4$; these results confirmed a conjecture of Pikhurko~\cite{P}, who had previously shown that $\delta(K^3_4, n) \leq 0.8603 n$, and who gave the construction which establishes the lower bound on $\delta(K^3_4, n)$. The exact value of $\delta(K^3_4 - e, n)$ for large $n$ remains an open problem. The cases listed above comprise all the $k$-graphs $H$ with no isolated vertices for which the value of $\delta(H, n)$ was previously known even asymptotically (if $H$ contains an isolated vertex then the behaviour is somewhat different, as we can restate the problem as asking for non-perfect packing of a smaller $k$-graph).

Other conditions, such as different notions of degree, which guarantee a perfect $H$-packing in a large $k$-graph $G$ have also been considered; see the survey by R\"odl and Ruci\'nski~\cite{RR} for a full account of these. In particular, in recent years there has been much study of the case of a perfect matching, see e.g.~\cite{AGS09, Alon+, CK, HPS, KKM, KM, Khp1, Khp2, KOT, LMp, MR, P, TZ12, TZp}. For perfect $H$-packings other than a perfect matching, results are much more sparse. Lo and Markstr\"om~\cite{LM2} found the asymptotic values of $\delta_1(K^3_3(m), n)$ and $\delta_1(K^4_4(m), n)$, where $\delta_1(H, n)$ denotes the smallest integer $\delta$ such that any $k$-graph $G$ on $n$ vertices with $\deg_G(\{x\}) \geq \delta$ for any $x \in V(G)$ contains a perfect $H$-packing, and $K^r_r(m)$ denotes the complete $r$-partite $r$-graph (defined below) whose vertex classes each have size $m$. More recently, Han and Zhao~\cite{HZ} gave the exact value of $\delta_1(K^3_4-2e, n)$ for large $n$, whilst Lenz and Mubayi~\cite{LM3} proved that for any linear $k$-graph $F$ (meaning that any two edges of $F$ intersect in at most one vertex), any sufficiently large `quasirandom' $k$-graph with linear density contains a perfect $F$-packing. However, in general our knowledge of conditions which guarantee a perfect $H$-packing in a $k$-graph $G$ remains very limited.

\subsection{Perfect packings in hypergraphs: new results}
In this paper we determine the asymptotic value of $\delta(K, n)$ for all complete $k$-partite $k$-graphs, as well as a large class of non-complete $k$-partite $k$-graphs $K$. Let $K$ be a $k$-graph on vertex set $U$ with at least one edge (if $K$ has no edges then trivially $\delta(K, n) = 0$). Then a \emph{$k$-partite realisation} of $K$ is a partition of $U$ into \emph{vertex classes} $U_1, \dots, U_k$ so that for any $e \in K$ and $1 \leq j \leq k$ we have $|e \cap U_j| = 1$. Equivalently, we colour all vertices of $K$ with $k$ colours so that no edge contains two vertices of the same colour; the vertex classes are then the colour classes. Note in particular that we must have $|U_j| \geq 1$ for every $1 \leq j \leq k$. We say that $K$ is \emph{$k$-partite} if it admits a $k$-partite realisation. 
The \emph{complete $k$-partite $k$-graph} with vertex classes $U_1, \dots, U_k$ is the $k$-graph on $U = U_1 \cup \dots \cup U_k$ in which every set $e \subseteq U$ with $|e \cap U_j| =1$ for each $1 \leq j \leq k$ is an edge.
Observe that a complete $k$-partite $k$-graph has only one $k$-partite realisation up to permutations of the vertex classes $U_1, \dots, U_k$.

Our first theorem states that for any $k$-partite $k$-graph $K$ we have $\delta(K, n) \leq n/2 + o(n)$.

\begin{theo} \label{main1}
Let $K$ be a $k$-partite $k$-graph on $b$ vertices. Then for any $\alpha > 0$ there exists $n_0$ such that if $G$ is a $k$-graph on $n \geq n_0$ vertices for which $b$ divides $n$ and $\delta(G) \geq n/2 + \alpha n$, then $G$ contains a perfect $K$-packing.
\end{theo}

Theorem~\ref{main1} could also be proved by the so-called `absorbing method' by using similar arguments and results to those of Lo and Markstr\"om~\cite{LM2}. However, our methods also give stronger bounds for many $k$-partite $k$-graphs $K$, for this we make the following definitions. Let $K$ be a $k$-partite $k$-graph on vertex set $U$. Then we define 
$$\Sa(K) := \bigcup_\chi \{|U_1|, \dots, |U_k|\} \textrm{    and    }  \D(K) := \bigcup_\chi \{||U_i| - |U_j|| : i, j \in [k]\},$$
where in each case the union is taken over all $k$-partite realisations $\chi$ of $K$ into vertex classes $U_1, \dots, U_k$ of~$K$. The \emph{greatest common divisor} of $K$, denoted $\gcd(K)$, is then defined to be the greatest common divisor of the set $\D(K)$ (if $\D(K) = \{0\}$ then $\gcd(K)$ is undefined). So for any given $k$-partite realisation of $K$, the difference in size of any two vertex classes of this realisation must be divisible by $\gcd(K)$. 
However, it is not true that a $k$-partite $k$-graph $K$ must have some $k$-partite realisation in which the greatest common factor of the differences of vertex class sizes is $\gcd(K)$. To see this, take disjoint sets $A, B, C, D$ and $E$ of size one, one, two, two and six respectively. Form a $3$-graph $K$ on $A \cup B \cup C \cup D \cup E$ whose edges are any triple $\{x, y, z\}$ with $x \in A, y \in C$ and $z \in E$ or with $x \in B$, $y \in D$ and $z \in E$. Then, up to permutation of the vertex classes, $K$ has two distinct $3$-partite realisations, one with vertex classes $A \cup B$, $C \cup D$ and $E$ of sizes two, four and six (so the highest common factor of the differences of class sizes is two), and the other with vertex classes $A \cup D$, $B \cup C$ and $E$ of sizes three, three and six (whose differences have highest common factor three). So $\gcd(K) = 1$ in this case, but the differences between sizes of vertex classes of any single $k$-partite realisation of $K$ have a larger common factor. 

We also define the \emph{smallest class ratio} of $K$, denoted $\sigma(K)$, by 
$$ \sigma(K) := \frac{\min_{S \in \Sa(K)} S}{|V(K)|}.$$
So each vertex class of any $k$-partite realisation of $K$ has size at least $\sigma(K)|V(K)|$. The parameter $\sigma(K)$ therefore provides a measure of how `lopsided' $K$ can be. Note in particular that $\sigma(K) \leq 1/k$, with equality if and only if we have $|U_1| = |U_2| = \dots = |U_k|$ for any $k$-partite realisation of $K$ with vertex classes $U_1, \dots, U_k$.

Observe that the definitions of the parameters $\gcd(K)$ and $\sigma(K)$ of a $k$-partite $k$-graph $K$ bear a strong resemblance to those of the parameters $\gcd(H)$ and $\sigma(H)$ defined earlier for a graph $H$ with chromatic number $r=k$. We saw that K\"uhn and Osthus showed that these parameters determine $\delta(H, n)$ for any such $H$; similarly $\gcd(K)$ and $\sigma(K)$ play an extensive role in determining $\delta(K, n)$ for a $k$-partite $k$-graph $K$. Indeed, our next theorem strengthens Theorem~\ref{main1} for $k$-partite $k$-graphs $K$ with $\gcd(K) = 1$, by stating that $\delta(K, n) \leq \sigma(K)n + o(n)$ for such graphs.

\begin{theo} \label{main2}
Let $K$ be a $k$-partite $k$-graph on $b$ vertices, and suppose that $\gcd(K) = 1$. Then for any $\alpha > 0$ there exists $n_0$ such that if~$G$ is a $k$-graph on $n \geq n_0$ vertices for which $b$ divides $n$ and $\delta(G) \geq \sigma(K)n + \alpha n$, then~$G$ contains a perfect $K$-packing.
\end{theo}

Our third theorem improves on Theorem~\ref{main1} for $k$-partite $k$-graphs $K$ with $\gcd(\Sa(K)) = 1$ such that $\gcd(K) > 1$ is odd. 

\begin{theo} \label{main3}
Let $K$ be a $k$-partite $k$-graph on $b$ vertices such that $\gcd(\Sa(K)) = 1$ and $\gcd(K) > 1$,
and let $p$ be the smallest prime factor of $\gcd(K)$. Then for any $\alpha > 0$ there exists $n_0$ such that if $G$ is a $k$-graph on $n \geq n_0$ vertices for which $b$ divides $n$ and $$\delta(G) \geq \max\left\{\sigma(K) n + \alpha n, \frac{n}{p} + \alpha n\right\},$$ then $G$ contains a perfect $K$-packing.
\end{theo}

\setlength{\tabcolsep}{12pt}
\begin{table}
\begin{tabular}{c c c}
\toprule
Property of $K$ & Type & Upper bound on $\delta(K, n)$ \\ 
\midrule
$\Sa(K)= \{1\}$ & 0 &$ n/2 + o(n)$ \\
$\gcd(\Sa(K)) > 1$ & 0 &$n/2 + o(n)$ \\
$\gcd(K) = 1$ & 1 & $\sigma(K) n +o(n)$ \\
$\gcd(\Sa(K)) = 1, \gcd(K) = d > 1$ & $d$ & $\max \left(\sigma(K) n, \frac{n}{p(d)}\right) + o(n) $ \\
\bottomrule
\end{tabular}
\caption{A summary of the upper bounds on $\delta(K, n)$ provided by Theorems~\ref{main1},~\ref{main2} and~\ref{main3}, according to the \emph{type} of $K$ as defined in Section~\ref{sec:extremal}. In the final row, $p(d)$ denotes the smallest prime factor of $d$.} 
\label{boundtable}
\end{table}

For convenience, the upper bounds on $\delta(K, n)$ provided by Theorems~\ref{main1},~\ref{main2} and~\ref{main3} are summarised in Table~\ref{boundtable}. In Section~\ref{sec:extremal} we shall see that these upper bounds are best possible up to the error term for a large class of $k$-partite $k$-graphs $K$; in particular, this is true for all complete $k$-partite $k$-graphs. So for any complete $k$-partite $k$-graph $K$ the asymptotic value of $\delta(K, n)$ is determined by the parameters $\gcd(K)$ and $\sigma(K)$, in the same way that the corresponding parameters determined $\delta(G, n)$ for any $r$-chromatic graph $G$. The same is true for many incomplete $k$-partite $k$-graphs $K$; we discuss these $k$-graphs in Section~\ref{sec:conc}, as well as discussing the value of $\delta(K, n)$ for those $k$-partite $k$-graphs for which the correct asymptotic value remains unknown (of course, Theorems~\ref{main1},~\ref{main2} and~\ref{main3} still provide an upper bound on $\delta(K, n)$ in these cases). For those $k$-partite $k$-graphs $K$ for which we do determine the correct asymptotic value of $\delta(K, n)$, it is natural to ask whether the $o(n)$ error term can be removed. We conjecture that in fact each of Theorems~\ref{main1},~\ref{main2} and~\ref{main3} is still valid if the $\alpha n$ error term is replaced by a sufficiently large constant $C$ which depends only on~$K$. 

One special case of these results answers a question of R\"odl and Ruci\'nski~\cite[Problem~3.15]{RR}, who asked for the value of $\delta(C_{s}^3, n)$, where $C_s^3$ denotes the loose cycle $3$-graph of length $s$. (For general $k$, the loose cycle $k$-graph of length $s$, denoted $C^k_s$, is defined for any $s > 1$ to have $s(k-1)$ vertices $\{1, \dots, s(k-1)\}$ and $s$ edges 
$\{\{j(k-1)+1, \dots, j(k-1)+k\} \textrm{ for }0 \leq j < s\}$,
with addition taken modulo $k$.) In Section~\ref{sec:cycles} we shall see that $\gcd(C_s^k) = 1$ for any $k \geq 3$ and $s \geq 2$ except for the case $k=3$ and $s = 3$, in which case $\Sa(C_s^k) = \{2\}$ and so $\gcd(K)$ is undefined. So in all cases except for $k=s=3$ the $k$-graph $C_s^k$ satisfies the condition of Theorem~\ref{main2}, and in fact we will show that furthermore $C_s^k$ belongs to the class of $k$-graphs for which Theorem~\ref{main2} is best possible up to the error term. By modifying our arguments to handle the case $k=s=3$ separately, we obtain the following theorem.
\begin{theo} \label{cyclepack}
For integers $k \geq 3$ and $s \geq 2$ we have
$$\delta(C^k_s, n) = 
\begin{cases}
\frac{n}{2(k-1)} + o(n) &\mbox{if $s$ is even, and}\\
\frac{s+1}{2s(k-1)}n + o(n) &\mbox{otherwise.} 
\end{cases}$$
\end{theo}
Note that $C^3_2$ is identical to the $3$-graph $K^3_4 - 2e$; as described earlier, the result above was proved in this case (that is, for $k=3$ and $s=2$) by K\"uhn and Osthus~\cite{KO}, and more recently Czygrinow, DeBiasio and Nagle~\cite{CDN} gave the exact value of $\delta(C^3_2, n)$ for large $n$.

The final results of this paper, in Section~\ref{sec:conc}, concern the problem of finding a $K$-packing covering all but a constant number of vertices of a large $k$-graph $H$. By adapting the methods used for our results on perfect packings, we find that the minimum codegree requirement of Theorem~\ref{main2} (which applied only to $k$-partite $k$-graphs $K$ with $\gcd(K) = 1$) is sufficient to ensure such a $K$-packing for any $k$-partite $k$-graph $K$. More specifically, we have the following theorem. 

\begin{theo} \label{almostpack}
Let $K$ be a $k$-partite $k$-graph. Then there exists a constant $C = C(K)$ such that for any $\alpha > 0$ there exists $n_0 = n_0(K, \alpha)$ such that any $k$-graph $H$ on $n \geq n_0$ vertices with $\delta(H) \geq \sigma(K)n + \alpha n$ admits a $K$-packing covering all but at most $C$ vertices of $H$.
\end{theo}

By modifying a construction from Section~\ref{sec:extremal}, we will further see that Theorem~\ref{almostpack} is asymptotically best possible for a large class of $k$-partite $k$-graphs which includes all complete $k$-partite $k$-graphs. 

The results of this paper are significant as they provide the first cases other than that of a perfect matching for which the value of the well-studied parameter $\delta(H, n)$ is known even asymptotically for a $k$-graph $H$ on more than four vertices. Furthermore, the diverse behaviour of this parameter over different $k$-partite $k$-graphs, according to the divisibility properties of the different vertex class sizes, is interesting in itself and increases our understanding of the extensive role such divisibility conditions play in a wide variety of problems involving the embedding of a spanning subgraph in a large $k$-graph $H$ (see~\cite{KM} for further discussion of this point). The proofs in this paper also demonstrate techniques for making use of the recent hypergraph blow-up lemma of Keevash~\cite{K}, particularly the techniques used in Section~\ref{sec:delete} to delete copies of $K$ so as to meet certain divisibility conditions. 

\subsection{Layout of the paper}
In Section~\ref{sec:extremal} we give constructions which show that the lowest upper bound provided by Theorems~\ref{main1},~\ref{main2} and~\ref{main3} is asymptotically best possible for all complete $k$-partite $k$-graphs. Then, in Section~\ref{sec:reduction} we state Lemmas~\ref{main1simple} and~\ref{main2simple} which are similar to Theorems~\ref{main1},~\ref{main2} and~\ref{main3}, but which pertain only to certain complete $k$-partite $k$-graphs. We deduce Theorems~\ref{main1},~\ref{main2} and~\ref{main3} from these lemmas; having done so, we can focus solely on these complete $k$-partite $k$-graphs in proving Lemmas~\ref{main1simple} and~\ref{main2simple} (complete $k$-partite $k$-graphs are simpler to deal with as they have only one $k$-partite realisation). In Section~\ref{sec:outline} we outline how the proofs of Lemmas~\ref{main1simple} and~\ref{main2simple} will proceed. These proofs make extensive use of hypergraph regularity; in particular, we use the recent hypergraph blow-up lemma due to Keevash~\cite{K}. The necessary background for the use of these tools is given in Section 5. Section~\ref{sec:prelims} then gives a number of auxiliary lemmas which will be needed in the proofs of Lemmas~\ref{main1simple} and~\ref{main2simple}, after which we prove these lemmas in Section~\ref{sec:proofs}. In Section~\ref{sec:cycles} we turn to the problem of a loose cycle packing, proving Theorem~\ref{cyclepack}, and the final section, Section~\ref{sec:conc} consists of concluding remarks. Firstly, we consider non-complete $k$-partite $k$-graphs, identifying large classes of such $k$-graphs for which the bounds of Theorems~\ref{main1},~\ref{main2} and~\ref{main3} are asymptotically best possible. Following this we consider the question of finding a $K$-packing covering all but a constant number of vertices of a large $k$-graph $G$, proving Theorem~\ref{almostpack}. Finally, we briefly discuss the problem of finding $\delta(H, n)$ for $k$-graphs $H$ which are not $k$-partite. 

\subsection{Notation} 
Throughout this paper, when we speak of `deleting' a $k$-graph $K$ from a $k$-graph $H$, we mean that both the vertices and edges of $K$ are deleted from $H$, so what remains is the subgraph of $H$ induced by the undeleted vertices. Also, for a $k$-graph $H$ we define the \emph{adjacency graph} $\textrm{Adj}(H)$ to be the graph on $V(H)$ where there is an edge between two vertices $i$ and $j$ if and only if some edge of $H$ contains both $i$ and $j$. We say that $H$ is \emph{connected} if $\mathrm{Adj}(H)$ is connected.

We write vectors in bold font, and write, for example, $v_j$ for the $j$th coordinate of $\vb$. We write $\ub_j$ for the unit vector whose $j$th coordinate is one and whose other coordinates are all zero (the dimension of $\ub_j$ will always be clear from the context). Whenever we speak of a partition of a set, we implicitly fix an order of the parts of this partition. We write $[r]$ to denote the set of integers from $1$ to $r$. For a set $A$, we use $\binom{A}{k}$ to denote the collection of subsets of $A$ of size $k$, and similarly $\binom{A}{\leq k}$ to denote the collection of subsets of $A$ of size at most $k$. We write $x = y \pm z$ to mean that $y - z \leq x \leq y + z$, and write $o(n)$ to denote a function which tends to zero as $n \to \infty$. Also, we use $x \ll y$ to mean for any $y \geq 0$ there exists $x_0 \geq 0$ such that for any $x \leq x_0$ the following statement holds, and similar statements with more constants are defined similarly. Finally, we omit floors and ceilings throughout this paper whenever they do not affect the argument. 

\section{Extremal examples} \label{sec:extremal} 

In this section we shall give constructions which demonstrate that the upper bound on $\delta(K, n)$ provided by Theorems~\ref{main1},~\ref{main2} and~\ref{main3} is asymptotically best possible for all complete $k$-partite $k$-graphs, and many others besides.

For ease of discussion, we divide all $k$-partite $k$-graphs into \emph{types}. Indeed, let $K$ be a $k$-partite $k$-graph with at least one edge. Then we say that $K$ is \emph{type 0} if $\gcd(\Sa(K)) > 1$ or if $K$ consists of $k$ vertices and one edge (in which case a $K$-packing is a matching). Note that the latter condition is equivalent to saying that $\Sa(K) = \{1\}$. If $K$ is not type 0, then for any $d \geq 1$ we say that $K$ is \emph{type $d$} if $\gcd(\Sa(K)) =1$ and $\gcd(K) = d$. Observe that every $k$-partite $k$-graph $K$ with at least one edge falls into precisely one of these types, since $\gcd(K)$ is defined for any $K$ which is not type 0. Also note that the definitions of $\Sa(K)$ and $\gcd(K)$ immediately imply that $\gcd(\Sa(K))$ divides $\gcd(K)$, so we cannot have $\gcd(\Sa(K)) > 1$ and $\gcd(K) = 1$.
Table~\ref{boundtable} displays the asymptotic upper bounds provided for $k$-partite $k$-graphs of each type by Theorems~\ref{main1},~\ref{main2} and~\ref{main3}. 
The results of this section will show that for a large class of $k$-partite $k$-graphs, which includes all complete $k$-partite $k$-graphs, these bounds are best possible up to the $o(n)$ error terms.

Our first construction is well-known and gives a condition (P1) on a $k$-partite $k$-graph~$K$ which is sufficient to ensure that the bound given by Theorem~\ref{main1} is asymptotically tight. The left hand part of Figure~\ref{fig:div2} gives an illustration of this construction.

\begin{prop} \label{extrem1}
Let $p>1$, and let $K$ be a $k$-partite $k$-graph on vertex set $U$ such that 
\begin{equation}
\mbox{each set $A \subseteq U$ for which $|e \cap A|$ is even for every $e \in K$, has size divisible by $p$.}\tag{P1}
\end{equation}
 Then for any $n$ there exists a $k$-graph $G$ on $n$ vertices with $\delta(G) \geq n/2-k$ such that $G$ does not contain a perfect $K$-packing.
\end{prop}

\begin{figure}[t] 
\centering
\psfrag{4}{$V_1$}
\psfrag{2}{$V_2$} 
\psfrag{3}{$V_3$}
\psfrag{A}{$V_1$}
\psfrag{B}{$V_2$}
\includegraphics[width=4cm]{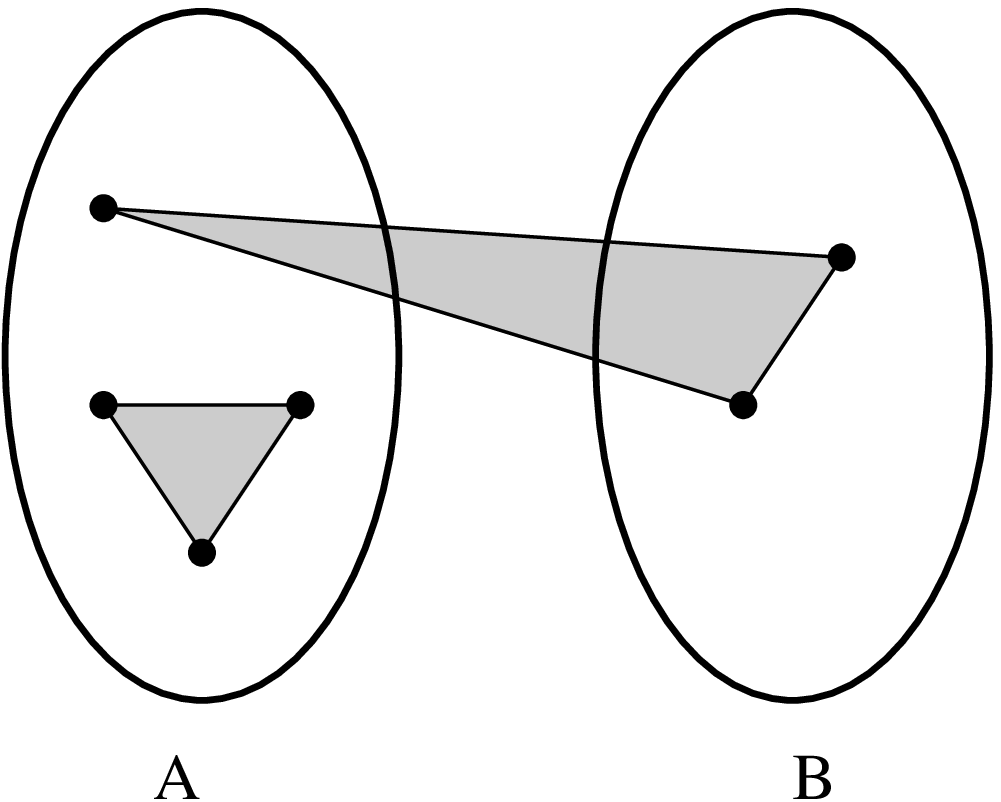} \hspace{2cm}
\includegraphics[width=6cm]{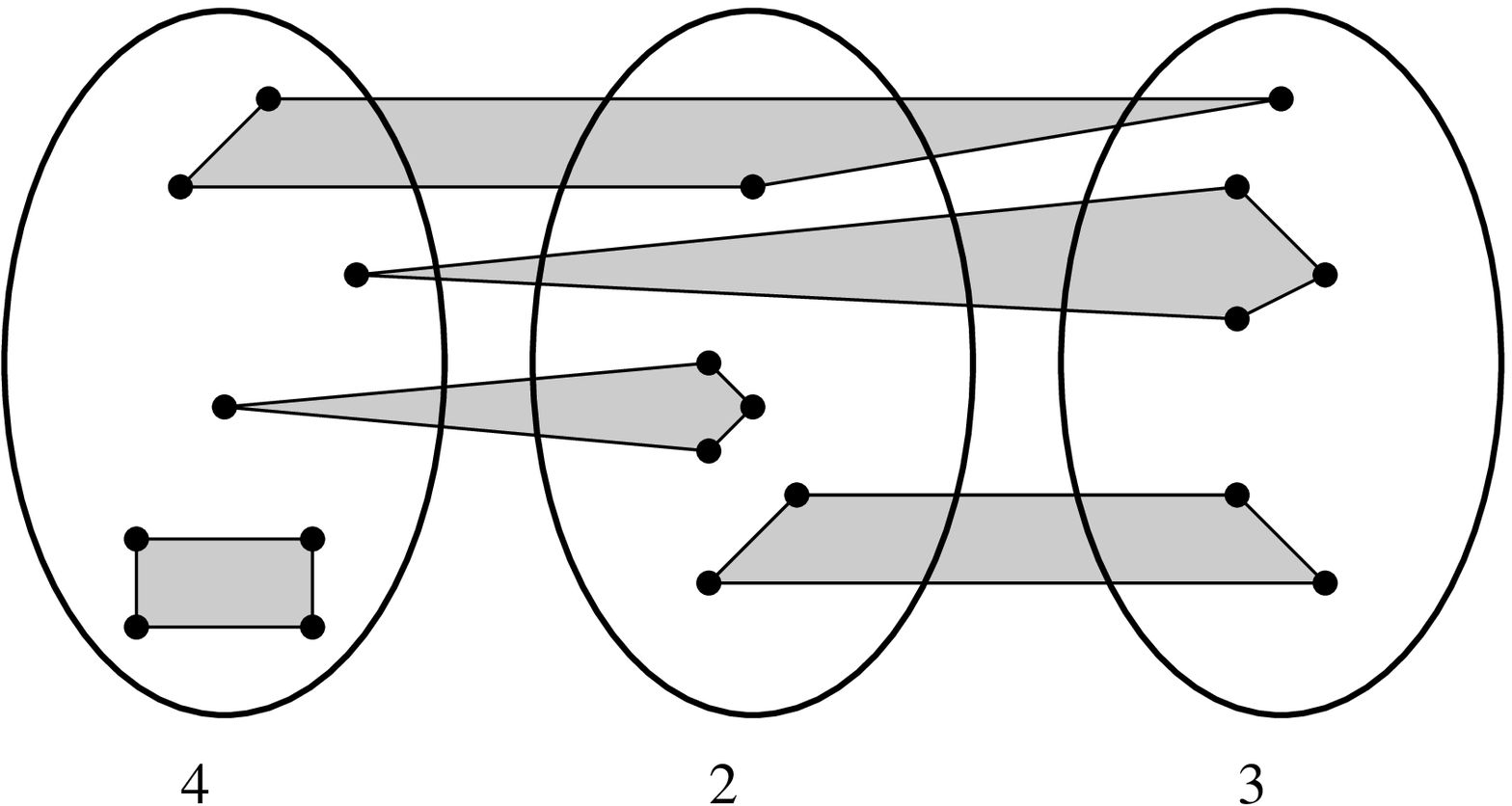}
\caption{The construction of Proposition~\ref{extrem3}, shown for $p=2, k=3$ on the left and for $p=3, k=4$ on the right. The left hand construction is also used for Proposition~\ref{extrem1} for $k=3$. In each case the $k$-graph of the construction has all $k$-tuples of the forms shown as edges (so, for example, the edges of the $3$-graph on the left are all $3$-tuples with either 1 or 3 vertices in~$V_1$).}
\label{fig:div2} 
\end{figure} 

\begin{proof}
Let $V_1$ and $V_2$ be disjoint sets of vertices with $|V_1 \cup V_2| = n$ such that $|V_1|, |V_2| \geq n/2 - 1$ and $p$ does not divide $|V_2|$. Let $G$ be the $k$-graph on vertex set $V_1 \cup V_2$ whose edges are all $k$-tuples $e \in \binom{V_1 \cup V_2}{k}$  such that $|V_2 \cap e|$ is even. Then $\delta(G) \geq \min\{|V_1|,|V_2|\} - (k-1) \geq n/2 - k$. Moreover, for any copy $K'$ of $K$ in $G$, every edge $e \in K'$ is an edge of $G$, and so $|V_2 \cap e|$ is even. By our assumption on $K$ this implies that $p$ divides $|V_2 \cap V(K')|$, and so the number of vertices of $V_2$ covered by any $K$-packing in $G$ is divisible by $p$. Since $p$ does not divide $|V_2|$, we conclude that $G$ does not contain a perfect $K$-packing. 
\end{proof}

Suppose that $K$ is a complete $k$-partite $k$-graph of type 0 with vertex classes $U_1, \dots, U_k$. If $K$ consists of just one edge then (P1) is trivially satisfied for $p = 2$; otherwise there exists some $p > 1$ such that $p$ divides $|U_j|$ for every $j \in [k]$. 
Let $A \subseteq U := \bigcup_{i \in [k]} U_i$ be such that $|e \cap A|$ is even for every $e \in K$. For any $i \in [k]$ and any $x,y \in U_i$ we may choose vertices $u_j \in U_j$ for each $j \neq i$, and since $K$ is complete both $\{x\} \cup \{u_j : j \neq i\}$ and $\{y\} \cup \{u_j : j \neq i\}$ are edges of $k$. Both these sets therefore have an even number of vertices in $|A|$, so we must have $x, y \in A$ or $x, y \notin A$. It follows that for each $i \in [k]$ we have either $U_i \subseteq A$ or $U_i \cap A = \emptyset$, and therefore that $p$ divides $|A|$. We conclude that any complete $k$-partite $k$-graph $K$ of type $0$ satisfies (P1), and so Theorem~\ref{main1} is asymptotically best possible for any such $K$. 

Our next construction generalises the construction used in Proposition~\ref{extrem1}. For this we need the following definitions. For any integer $p \geq 2$, let $\V_p$ be the $(p-1)$-dimensional sublattice of $\Z_p^p$ generated by the vectors $\vb_1, \dots, \vb_{p-1}$, where for each $1 \leq j \leq p-1$ we define
$$\vb_{j} = \ub_j + (j-1) \ub_p =  (\overbrace{0, \dots,0}^{j-1}, 1, \overbrace{0, \dots, 0}^{p-j-1}, {j-1}).
$$
The key property of the lattice $\V_p$ is that, working in $\Z^p_p$,
\begin{equation}
\parbox{13cm}{for any $\xb \in \Z_p^p$ there is precisely one $j \in [p]$ such that $\xb + \ub_j \in \V_p$.}\tag{$\dagger$}
\end{equation}
That is, for any vector $\xb \in \Z_p^p$ there is precisely one coordinate of $\xb$ which, if incremented by one (modulo $p$), yields a vector $\yb \in \V_p$. To see this, let $\xb \in \Z_p^p$. Then it follows immediately from the definition of $\V_p$ that there is $\yb' \in \V_p$ such that $\xb$ and $\yb'$ differ only in their last coordinates $x_p$ and $y_p'$ (or do not differ at all). Let $d \in [p]$ be such that $x_p - y_p' \equiv d-1$ modulo $p$, and define 
$$\yb :=
\begin{cases}
\yb' + \vb_d & \mbox{if $d \leq p-1$}\\
\yb' & \mbox{if $d=p$.} 
\end{cases} $$ 
Then in any case we have $\yb \in \V_p$. Moreover, if $d = p$ then $x_p - y_p = x_p - y_p' \equiv -1$ modulo $p$, so $\yb = \xb + \ub_p$. On the other hand, if $d \neq p$ then $y_p \equiv y'_p + (\vb_d)_p \equiv y'_p + d-1 \equiv x_p$ modulo $p$, so $\yb$ and $\xb$ differ only in coordinate $d$, with $y_d \equiv x_d + 1$ modulo $p$, and therefore $\yb = \xb + \ub_d$. This proves that for any $\xb \in \Z^p_p$ there is at least one $j \in [p]$ such that $\xb + \ub_j \in \V_p$. If for some $\xb \in \Z^p_p$ there were $j \neq j'$ such that $\xb + \ub_j \in \V_p$ and $\xb + \ub_{j'} \in \V_p$, then we would obtain $(\xb + \ub_j) - (\xb + \ub_{j'}) = \ub_j - \ub_{j'} \in \V_p$. However, it is easily checked that $\ub_j - \ub_{j'} \notin \V_p$ for any $j \neq j'$, proving $(\dagger)$. 

If $\Part$ is a partition of a set $X$ into parts $X_1, \dots, X_p$, for any $S \subseteq X$ we define the \emph{index vector of $S$ with respect to $\Part$}, denoted $\ib_\Part(S)$, to be the vector in $\Z_p^p$ whose $j$-th coordinate is $|S \cap X_j|$ modulo~$p$; this is well-defined since we consider $\Part$ to include an order on its parts. We sometimes omit the subscript $\Part$ and write simply $\ib(S)$ if $\Part$ is clear from the context.

\begin{prop} \label{extrem3}
Let $p \geq 2$ and let $K$ be a $k$-partite $k$-graph on vertex set $U$ such that 
\begin{equation}
\parbox{13cm}{for any partition $\Part$ of $U$ into $p$ parts such that $\ib(e) \in \V_p$ for every $e \in K$ we must also have $\ib(U) \in \V_p$.} \tag{P2}
\end{equation}
Then for any $n$ there exists a $k$-graph $G$ on $n$ vertices with $\delta(G) \geq n/p-k$ such that $G$ does not contain a perfect $K$-packing.
\end{prop}

\begin{proof}
Let $V$ be a set of $n$ vertices, and choose a partition $\Part$ of $V$ into parts $V_1, \dots, V_p$ such that $\ib(V) \notin \V_p$ and $|V_j| \geq n/p - 1$ for each $j \in [p]$. Let $G$ be the $k$-graph on vertex set $V$ such that a $k$-tuple $e \in \binom{V}{k}$ is an edge of $G$ precisely if $\ib(e) \in \V_p$ (see Figure~\ref{fig:div2} for two illustrations of this construction). Then for any $(k-1)$-tuple $e' \in \binom{V}{k-1}$ we can choose $j \in [p]$ by $(\dagger)$ such that $\ib(e') + \ub_j \in \V_p$, and then adding any of the $|V_j \sm e'| \geq n/p-k$ vertices $v \in V_j \sm e'$ to $e'$ gives a $k$-tuple $e := e' \cup \{v\}$ such that $\ib(e) = \ib(e') + \ub_j \in \V_p$, that is, an edge $e \in G$. So $\delta(G) \geq n/p - k$.  Now, for any copy $K'$ of $K$ in $G$, every edge $e \in K'$ is an edge of $G$ and so has the property that $\ib(e) \in \V_p$. By (P2) it follows that $\ib(V(K')) \in \V_p$. So if $\F$ is a $K$-packing in $G$, then $\ib(V(\F)) = \sum_{K' \in \F} \ib(V(K')) \in \V_p$; since $\ib(V) \notin \V_p$ this implies that $\F$ is not perfect.
\end{proof}

Note in particular that for $p=2$ the $k$-graph $G$ constructed in the proof of Proposition~\ref{extrem3} is the same as that in Proposition~\ref{extrem1}. 
A similar argument as above shows that any complete $k$-partite $k$-graph $K$ of type $d \geq 2$ satisfies (P2) for any $p$ which divides $d$. Indeed, let $U_1, \dots, U_k$ be the vertex classes of $K$, and suppose that sets $V_1, \dots, V_p$ partition $V(K)$ such that (taking index vectors with respect to this partition) $\ib(e) \in \V_p$ for every $e \in K$. Fix any $j \in [k]$ and any $u, v \subseteq U_j$; then we may choose edges $e, e' \in K$ such that $u \in e$, $v \in e'$ and $e \sm \{u\} = e' \sm \{v\} =: e^*$. Since $\ib(e^*) \in \Z^p_p$, by $(\dagger)$ there is precisely one $i \in [k]$ such that adding a vertex of $V_i$ to $e^*$ gives an edge whose index vector lies in $\V_p$. Since $\ib(e), \ib(e') \in \V_p$, we must have $u, v \in V_i$ for this $i$, and so we conclude that every vertex class $U_j$ of $K$ must be a subset of some $V_i$. Since $K$ has type $d$ we have that $p$ divides $\gcd(K) = d$, and so each vertex class $U_j$ has equal size $b_1$ modulo $p$. So we must have $\ib(V(K)) = b_1 \ib(e) \in \V_p$, proving that $K$ indeed satisfies property (P2). So, up to the $o(n)$ error term, the bound of $n/p + o(n)$ given in Theorem~\ref{main3} is best possible for any complete $k$-partite $k$-graph of type $d \geq 2$.

The final construction we use is well-known. For this, we write $\tau(K)$ to denote the proportion of vertices of $K$ contained in a smallest vertex cover of $K$. That is, $\tau(K) = |S|/b$, where $b = |V(K)|$ and $S \subseteq V(K)$ is a set of minimum size with the property that every edge of $K$ contains a vertex of $S$ (we express $\tau(K)$ as a proportion for comparison with $\sigma(K)$). 
 
\begin{figure}[t] 
\centering
\psfrag{A}{$A$}
\psfrag{B}{$B$}
\includegraphics[width=5cm]{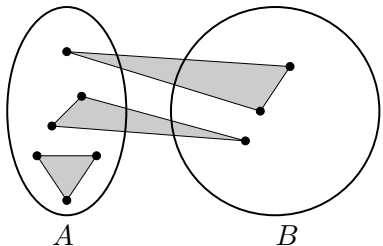}
\caption{The construction of Proposition~\ref{extrem2}, illustrated for $k=3$. The edges of this 3-graph are all 3-tuples which intersect $A$.}
\label{fig:space} 
\end{figure} 

\begin{prop} \label{extrem2}
For any $k$-graph $K$ and any $n$ there exists a $k$-graph $G$ on $n$ vertices with $\delta(G) = \lceil\tau(K) n\rceil - 1$ such that $G$ does not contain a perfect $K$-packing.
\end{prop}

\begin{proof} 
Write $b = |V(K)|$ and $\tau =\tau(K)$. Let $A$ and $B$ be disjoint sets of vertices such that $|A| = \lceil\tau n\rceil - 1$ and $|A \cup B| = n$, and let $G$ be the $k$-graph on vertex set $A \cup B$ whose edges are all $k$-tuples $e \in \binom{A \cup B}{k}$ such that $|e \cap A| \geq 1$ (this construction is illustrated in Figure~\ref{fig:space}). Then $\delta(G) = |A| = \lceil\tau n\rceil - 1$. Let $K'$ be a copy of $K$ in $G$. Then any edge $e \in K'$ must contain a vertex of $A$. So $A \cap V(K')$ is a vertex cover of $K'$, so $|A \cap V(K')| \geq \tau b$. Any $K$-packing in $G$ therefore has size at most $|A|/\tau b < n/b$, so is not perfect.
\end{proof}
 
This means that for any $k$-graph $K$ with $\tau(K) = \sigma(K)$ we have $\delta(K, n) \geq \sigma(K) n$. Together with Theorems~\ref{main2} and~\ref{main3} this determines $\delta(K, n)$ asymptotically for any such $k$-partite $k$-graph $K$ which has type 1, or which has type $d \geq 2$ and also satisfies property (P2) with $p$ being the smallest prime factor of $d$. Note that in particular we have $\tau(K) = \sigma(K)$ for any complete $k$-partite $k$-graph $K$. To see this, let $K$ be a complete $k$-partite $k$-graph on $b$ vertices with vertex classes $U_1, \dots, U_k$. Since $K$ is complete a subset $S \subseteq V(K)$ is a vertex cover if and only if $U_j \subseteq S$ for some $j \in [k]$. So $\tau(K) = \min_{j \in [k]} |U_j|/b = \sigma(K)$, as required.

	 In conclusion, these examples show that the bound given by Theorems~\ref{main1},~\ref{main2} and~\ref{main3} is best possible up to the $o(n)$ error term for all complete $k$-partite $k$-graphs. We consider non-complete $k$-partite $k$-graphs further in Section~\ref{sec:conc}. 

\section{Reduction to complete $k$-partite $k$-graphs} \label{sec:reduction} 

For simplicity, we would like to restrict our attention to complete $k$-partite $k$-graphs alone, as these have only one $k$-partite realisation (up to permutations of the vertex classes) and so are easier to work with. Clearly if $U_1, \dots, U_k$ are the vertex classes of a $k$-partite realisation of a $k$-graph $K$, then $K$ is a spanning subgraph of the complete $k$-partite $k$-graph $K'$ on the same vertex classes, and so if $G$ contains a $K'$-packing then $G$ contains a $K$-packing also. However, we cannot deduce from this that proving Theorems~\ref{main1},~\ref{main2} and~\ref{main3} for complete $k$-partite $k$-graphs would imply that these theorems hold for all $k$-partite $k$-graphs, as it may well be the case that $\gcd(K) \neq \gcd(K')$. Instead, in this section we shall state two lemmas (Lemmas~\ref{main1simple} and~\ref{main2simple}) which essentially say that Theorems~\ref{main1},~\ref{main2} and~\ref{main3} hold for certain complete $k$-partite $k$-graphs. We shall then deduce Theorems~\ref{main1},~\ref{main2} and~\ref{main3} from these two lemmas, showing that it is indeed sufficient to consider only these complete $k$-partite $k$-graphs for the rest of the paper (in which we prove Lemmas~\ref{main1simple} and~\ref{main2simple}). 

\begin{lemma} \label{main1simple}
Let $K$ be the complete $k$-partite $k$-graph whose vertex classes have sizes $b_1, \dots, b_k$, where these sizes are not all equal, and suppose that $\gcd(K)$ and $b_1$ are coprime. Then for any $\alpha > 0$ there exists $n_0 = n_0(K, \alpha)$ such that the following statement holds. Let $H$ be a $k$-graph on $n \geq n_0$ vertices such that
\begin{enumerate}
\item[(a)] $b := b_1 + \dots + b_k$ divides $n$,
\item[(b)] $\delta(H) \geq \sigma(K)n + \alpha n$, and
\item[(c)] if $\gcd(K) > 1$, then $\delta(H) \geq n/p + \alpha n$, where $p$ is the smallest prime factor of $\gcd(K)$.
\end{enumerate}
Then $H$ contains a perfect $K$-packing.
\end{lemma}

\begin{lemma} \label{main2simple}
Let $K$ be the complete $k$-partite $k$-graph whose vertex classes each have size $b_1$. Then for any $\alpha > 0$ there exists $n_0 = n_0(K, \alpha)$ such that if $n \geq n_0$ is divisible by $b_1k$ and $H$ is a $k$-graph on $n$ vertices with $\delta(H) \geq n/2 + \alpha n$ then $H$ contains a perfect $K$-packing.
\end{lemma}

For the rest of this section we seek to deduce Theorems~\ref{main1},~\ref{main2} and~\ref{main3} from Lemmas~\ref{main1simple} and~\ref{main2simple}. For this we shall need the following fact of elementary number theory.

\begin{fact}\label{gcdcoeffs}
For any positive integers $r_1, \dots, r_k$ there exist integers $a_1, \dots, a_k$ such that $a_1 r_1 + \dots + a_k r_k = \gcd(\{r_1, \dots, r_k\})$.
\end{fact}

Let $\F$ be a collection of $k$-graphs. Then we say that a $k$-graph $G$ contains an \emph{$\F$-packing} if $G$ can be packed with members of $\F$. More precisely, an $\F$-packing in $G$ is a collection of pairwise vertex-disjoint subgraphs $F_1, \dots, F_r$ of $G$ so that each $F_j$ is in $\F$ (that is, $F_j$ is isomorphic to a member of $\F$). We say that an $\F$-packing of $G$ is $\emph{perfect}$ if it covers every vertex of $G$. This naturally generalises the notion of an $H$-packing for a $k$-graph $H$, as an $H$-packing of $G$ and an $\{H\}$-packing of $G$ are identical. The following elementary proposition implies that to demonstrate that $G$ contains a perfect $H$-packing it is sufficient to show that $G$ contains a perfect $\F$-packing for some family $\F$ such that every $F \in \F$ contains a perfect $H$-packing (we omit the simple proof).

\begin{prop}\label{fpacktohpack}
Suppose that $G$ and $H$ are $k$-graphs and that $\F$ is a collection of $k$-graphs such that 
\begin{enumerate}[(i)]
\item $G$ contains a perfect $\F$-packing, and
\item every $F \in \F$ contains a perfect $H$-packing. 
\end{enumerate}
Then $G$ contains a perfect $H$-packing.
\end{prop}

To deduce Theorems~\ref{main1},~\ref{main2} and~\ref{main3} from Lemmas~\ref{main1simple} and~\ref{main2simple}, we make use of the following complete $k$-partite $k$-graphs, which will also play important roles later on in the paper. 

\begin{defn}\label{defBULo}
Fix an integer $k \geq 3$.
\begin{enumerate}[(i)]
\item For any integer $m$, we define the \emph{balanced $k$-partite $k$-graph} $\B(m)$ to be the complete $k$-partite $k$-graph on vertex classes $W_1, \dots, W_k$, where $|W_1| = |W_2| = \dots = |W_k| = m$. So $\B(m)$ has $km$ vertices.
\item Likewise, for integers $m$ and $d$ with $d < m$, we define the \emph{$d$-unbalanced $k$-partite $k$-graph} $\U(m, d)$ to be the complete $k$-partite $k$-graph on vertex classes $W_1, \dots, W_k$, where 
$$ |W_1| = m - d, |W_2| = m+d, \mbox{ and } |W_3| = \dots = |W_k| = m.$$
So $\U(m, d)$ has $km$ vertices also.
\item Finally, let $m$ be an integer and $0 < \sigma < 1$. Then we define the \emph{$\sigma$-lopsided $k$-partite $k$-graph $\Lo(m, \sigma)$} to be the complete $k$-partite $k$-graph on vertex classes $W_1, \dots, W_k$, where 
$$|W_1| = \sigma m, \mbox{ and } |W_2| = \dots = |W_k| = \frac{(1-\sigma)m}{k-1}, $$
provided that these vertex class sizes are each integers (otherwise $\Lo(m, \sigma)$ is undefined). So $\Lo(m, \sigma)$ has $m$ vertices.
\end{enumerate}
Note that the definitions of $\B(m)$, $\U(m, d)$ and $\Lo(m, \sigma)$ each depend on $k$; this dependence is suppressed in our notation as $k$ will always be clear from the context.

Given a $k$-partite $k$-graph $K$, some special cases of the above definition will be of particular importance; we therefore define $\B(K) := \B(b)$ and $\Lo(K) := \Lo((k-1)!b,\sigma(K))$, where $b$ denotes the number of vertices of $K$. Finally, if $\gcd(K)$ is defined, then we define $\U_s(K) := \U(sb, \gcd(K))$ for those integers $s$ for which the $k$-graph $\U_s(K)$ so defined admits a perfect $K$-packing; the next proposition tells us that this is the case for any sufficiently large~$s$.
\end{defn}

We note for future reference that $\B(K)$, $\U_s(K)$ and $\Lo(K)$ have $kb$, $kbs$ and $(k-1)!b$ vertices respectively, that $\sigma(\Lo(K)) = \sigma(K)$, and that if $\gcd(K) $ divides $s$ then $\gcd(\U_s(K)) = \gcd(K)$. Crucially, each of these $k$-partite $k$-graphs admits a perfect $K$-packing, as shown by the following proposition. 

\begin{prop}\label{hpackings}
Let $K$ be a $k$-partite $k$-graph on $b$ vertices, and let $\sigma := \sigma(K)$. Then the $k$-graphs $\B(K) = \B(b)$ and $\Lo(K) = \Lo((k-1)!b, \sigma)$ each contain a perfect $K$-packing. Furthermore, if $d := \gcd(K)$ is defined then there exists $s_0 = s_0(K)$ such that for any $s \geq s_0$ the $k$-graph $\U_s(K) = \U(sb, d)$ contains a perfect $K$-packing.
\end{prop}

\begin{proof}
Let $U_1, \dots, U_k$ be the vertex classes of a $k$-partite realisation of $K$. We form a $k$-partite $k$-graph $K^*$ with vertex classes $W_1, \dots, W_k$ as follows. Initially take $W_1, \dots, W_k$ to be empty sets, and then add $k$ vertex-disjoint copies of $K$ to $K^*$, so that the vertices of $U_i$ in the $j$th copy of $K$ are added to $W_{i+j}$ (with addition taken modulo $k$). That is, each vertex class of $K^*$ receives the vertices of one copy of $U_1$, one copy of $U_2$, and so forth. So each vertex class of $K^*$ has size $b$. We conclude that $K^*$ is a spanning subgraph of $\B(b)$. By construction $K^*$ contains a perfect $K$-packing, so $\B(b)$ contains a perfect $K$-packing also.

A similar argument holds for $\Lo((k-1)!b, \sigma)$. Indeed, by the definition of $\sigma(K)$ we may assume that $|U_1| = \sigma b$. Then we form a $k$-partite $k$-graph $K^*$ consisting of $(k-1)!$ vertex-disjoint copies of $K$: for each permutation $\rho$ of $[k]$ with $\rho(1) = 1$ we add a copy of $K$ to $K^*$ in which the vertices of $U_j$ are included in $W_{\rho(j)}$ for each $j \in [k]$. So $K^*$ has $(k-1)!b$ vertices in total; the first vertex class of $K^*$ has size $(k-1)!|U_1| = (k-1)!b\sigma$, whilst each other vertex class of $K^*$ has equal size 
$$\frac{(k-1)!b - (k-1)!b\sigma}{k-1} = (k-2)! b(1-\sigma).$$
So $K^*$ is a spanning subgraph of $\Lo((k-1)!b, \sigma)$. As before, since $K^*$ contains a perfect $K$-packing by construction, $\Lo((k-1)!b, \sigma)$ contains a perfect $K$-packing also.

Finally we come to $\U(sb, d)$. For this we must consider all possible $k$-partite realisations $\chi$ of $K$; let $\aleph$ be the set formed by all such $\chi$. We write $U_1^\chi, \dots, U_k^\chi$ for the vertex classes of the realisation $\chi$. Note that we consider all possible realisations, not simply all possible realisations up to permutations of the vertex classes. In particular, this means that the number of realisations $N := |\aleph|$ is divisible by $k!$. Note also that $N \leq k^b$, and that by symmetry we have $\sum_{\chi \in \aleph} |U_j^\chi| = bN/k$ for each $j \in [k]$. 
In addition, recall that 
$$d := \gcd(K) := \gcd \left(\{|U_1^\chi| - |U_2^\chi| : \chi \in \aleph\}\right).$$ 
So by Fact~\ref{gcdcoeffs} we may choose integers $a_\chi$ for each $k$-partite realisation $\chi$ of $K$ such that 
$$\sum_{\chi \in \aleph} a_\chi (|U_1^\chi| - |U_2^\chi|) = d.$$ 
Let $a := \max_{\chi \in \aleph} a_\chi$. We now form a $k$-partite $k$-graph $K^*$ similarly as before, with vertex classes $W_1, \dots, W_k$ which we initially take to be empty sets. Then, for each realisation $\chi$ of $K$, add $a - a_\chi$ vertex-disjoint copies of $K$ to $K^*$, with the vertices of $U_j^\chi$ added to $W_j$ for each $j \in [k]$, and also add $a_\chi$ vertex-disjoint copies of $K$ to $K^*$, with the vertices of $U_1^\chi$ added to $W_2$, the vertices of $U_2^\chi$ added to $W_1$, and the vertices of $U_j^\chi$ added to $W_j$ for each $j \geq 3$. Then the total number of vertices added to $W_1$ is 
$$\sum_{\chi \in \aleph} \left((a - a_\chi) |U_1^\chi| + a_\chi |U_2^\chi|\right) = \sum_{\chi \in \aleph} a|U_1^\chi| - \sum_{\chi \in \aleph} a_\chi(|U_1^\chi| - |U_2^\chi|) = \frac{aNb}{k} - d.$$ 
In the same way the number of vertices added to $W_2$ is 
$$\sum_{\chi \in \aleph} (a - a_\chi) |U_2^\chi| + a_\chi|U_1^\chi| = \frac{aNb}{k} + d,$$ 
and the number of vertices added to $W_j$ for each $j \geq 3$ is  
$\sum_{\chi \in \aleph} a|U^\chi_j| = aNb/k.$ 
So we may take $s_0 = aN/k$. Then $K^*$ contains a perfect $K$-packing by construction, and is a spanning subgraph of $\U(s_0b, d)$, from which we conclude that $\U(s_0b, d)$ contains a perfect $K$-packing. Finally, for any $s \geq s_0$ observe that $\U(sb, d)$ admits a $\{\B(b), \U(s_0b, d)\}$-packing consisting of $s-s_0$ copies of $\B(b)$ and one copy of $\U(s_0b, d)$; since we have already seen that both $\B(b)$ and $\U(s_0b, d)$ contain perfect $K$-packings it follows that $\U(sb, d)$ contains a perfect $K$-packing by Proposition~\ref{fpacktohpack}.
\end{proof}

We now have the definitions we need to derive Theorems~\ref{main1},~\ref{main2} and~\ref{main3} from Lemmas~\ref{main1simple} and~\ref{main2simple}. We shall also need the following weak version of a theorem of Erd\H{o}s~\cite{E}, which states that the Tur\'an density of any $k$-partite $k$-graph is zero.

\begin{theo}[\cite{E}] \label{turandensityzero}
For any $k$-partite $k$-graph $K$ and any $\alpha> 0$ there exists $n_0$ such that any $k$-graph $G$ on $n \geq n_0$ vertices with at least $\alpha \binom{n}{k}$ edges contains a copy of $K$.
\end{theo} 

\medskip \noindent {\bf Proof of Theorem~\ref{main1}.}
We may assume that $1/n \ll 1/k, 1/b, \alpha$. By repeated application of Theorem~\ref{turandensityzero} we may delete at most $k$ vertex-disjoint copies of $K$ from $G$ to obtain a subgraph $H$ such that $kb$ divides $n' := |V(H)|$. Since we deleted at most $kb \leq \alpha n/2$ vertices in forming $H$ we have $\delta(H) \geq n/2 + \alpha n/2 \geq n'/2 + \alpha n'/2$. So $H$ contains a perfect $\B(K)$-packing by Lemma~\ref{main2simple} (applied with $\B(K), n'$ and $\alpha/2$ in place of $K, n$ and $\alpha$ respectively).  Together with the deleted copies of $K$ this gives a perfect $\{\B(K), K\}$-packing of $G$, and $G$ therefore contains a perfect $K$-packing by Propositions~\ref{fpacktohpack} and~\ref{hpackings}. 
\endproof
 
\medskip \noindent {\bf Proof of Theorems~\ref{main2} and~\ref{main3}.}
We may assume that $1/n \ll 1/k, 1/b, \alpha$. Introduce new constants $m$ and $s$ with $1/n \ll 1/m \ll 1/s \ll 1/k, 1/b, \alpha$ such that both $m$ and $s$ are divisible by $\gcd(K)$. Then we may assume that $s$ is large enough for $\U_s(K)$ to be defined and so to contain a perfect $K$-packing by Proposition~\ref{hpackings}.

Our first step is to form a complete $k$-partite $k$-graph $K_1$ with vertex class sizes $b_1, \dots, b_k$ such that $K_1$ admits a perfect $K$-packing and such that $b_1$ and $\gcd(K)$ are coprime. For Theorem~\ref{main2} we have $\gcd(K) = 1$ so we may simply take any $k$-partite realisation of $K$ and add edges to form a complete $k$-partite $k$-graph $K_1$. So assume that $\gcd(K) > 1$ and $\gcd(\Sa(K)) = 1$, as in Theorem~\ref{main3}, and let $s_1, \dots, s_t$ be the elements of $\Sa(K)$. By Fact~\ref{gcdcoeffs} we may choose integers $a_1, \dots, a_t \geq 0$ such that $a_1s_1 +\dots + a_ts_t$ is coprime to $\gcd(K)$, and since $1/s \ll 1/k, 1/b$ we may do this so that $a_1 + \dots + a_t \leq s$. By definition of $\Sa(K)$, for each $i \in [t]$ we may choose a $k$-partite realisation $\chi_i$ of $K$ whose $k$ vertex classes have sizes $b_1^i, b_2^i, \dots, b_k^i$ with $b_1^i = s_i$. Let $K_1$ be the complete $k$-partite $k$-graph whose $j$th vertex class has size $b_j := \sum_{i \in [t]} a_i b_j^i$. Then by construction $K_1$ has a perfect $K$-packing consisting of $a_1$ copies of $K$ with realisation $\chi_1$, $a_2$ copies of $K$ with realisation $\chi_2$, and so forth. Moreover, $b_1 = \sum_{i \in [t]} a_i b_1^i = \sum_{i \in [t]} a_i s_i$, so $b_1$ is coprime to $\gcd(K)$. So $K_1$ has the desired properties; note also that $K_1$ has at most $sb$ vertices.
 
We next form a complete $k$-partite $k$-graph $K_2$ with vertex class sizes $b_1', \dots, b_k'$ such that $K_2$ admits a perfect $\{\U_s(K), K_1\}$-packing, $\gcd(K_2) = \gcd(K)$ and $b_1'$ and $\gcd(K)$ are coprime. 
For this observe that, since $K_1$ has a perfect $K$-packing, the definition of $\gcd(K)$ implies that $b_i - b_j$ is divisible by $\gcd(K)$ for any $i, j \in [k]$. Without loss of generality we may assume that $b_2 \geq b_3$. Fix $d := (b_2 - b_3)/\gcd(K)$, and 
define 
\begin{align*}
b_1' &:= b_1 + (d+1)(bs + \gcd(K)),\\
b_2' &:= b_2 + (d+1)(bs - \gcd(K)), \mbox{ and}\\ 
b_i' &:= b_i + (d+1)bs \mbox{ for $3 \leq i \leq k$}.
\end{align*} 
So in particular $b_i' \leq 3b^3s$ for any $i \in [k]$. 
Let $K_2$ be the complete $k$-partite $k$-graph with vertex class sizes $b'_1, \dots, b'_k$. Then $K_2$ admits a perfect $\{\U_s(K), K_1\}$-packing consisting of $d+1$ copies of $\U_s(K)$ and one copy of $K_1$. Also, since  $b_1$ and $\gcd(K)$ were coprime, and $\gcd(K)$ divides $s$, we find that $b'_1$ and $\gcd(K)$ are coprime. Finally, observe that $b_3' - b_2' = b_3 - b_2 + (d+1)\gcd(K) = \gcd(K)$. So $\gcd(K_2) \leq \gcd(K)$, and from the definition of $b'_i$ for $i \in [k]$ and the fact that $\gcd(K)$ divides $s$ we see that $\gcd(K)$ divides $b'_i - b'_j$ for any $i, j \in [k]$, from which we conclude that $\gcd(K_2) = \gcd(K)$. So $K_2$ has the desired properties. 

Define $b''_i$ for $i \in [k]$ by 
\begin{align*}
b''_1 &= (k-1)!b\sigma(K)m + b'_1, \mbox{ and} \\
b''_i &= (k-2)!b(1-\sigma(K))m + b'_i \mbox{ for $2 \leq i \leq k$,}
\end{align*}
and let $K_3$ be the complete $k$-partite $k$-graph with vertex class sizes $b''_1, \dots, b''_k$. Then $K_3$ admits a perfect $\{\Lo(K), K_2\}$-packing consisting of one copy of $K_2$ and $m$ copies of $\Lo(K)$; by Propositions~\ref{fpacktohpack} and~\ref{hpackings} and the fact that $K_2$ admits a perfect $\{\U_s(K), K_1\}$-packing it follows that $K_3$ admits a perfect $K$-packing. Also, since $b'_1$ and $\gcd(K)$ were coprime, and $\gcd(K)$ divides $m$, we find that $b''_1$ and $\gcd(K)$ are coprime. Furthermore, $b''_3 - b''_2 = b'_3-b'_2 = \gcd(K)$, and since $\gcd(K)$ divides $m$ we deduce that $\gcd(K_3) = \gcd(K)$ also. Finally, $|V(K_3)| = (k-1)!bm + |V(K_2)| \geq (k-1)!bm$, and so
$$\sigma(K_3) \leq \frac{b_1''}{|V(K_3)|} \leq \frac{(k-1)!b\sigma(K)m + b'_1}{(k-1)!bm} = \sigma(K) + \frac{3b^3s}{(k-1)!bm} \leq \sigma(K) + \frac{\alpha}{3}.$$

By Theorem~\ref{turandensityzero} we may arbitrarily choose and delete from $G$ at most $|V(K_3)|/b = |V(\Lo(K))|m/b + |V(K_2)|/b \leq (k-1)!m + 3kb^2s \leq \alpha n/3b$ copies of $K$ so that the set $V' \subseteq V(G)$ of undeleted vertices is such that $|V(K_3)|$ divides $|V'|$. Also $H := G[V']$ has $\delta(H) \geq \delta(G) - \alpha n/3 \geq \sigma(K)n + 2\alpha n/3 \geq \sigma(K_3)n + \alpha n/3$. Similarly, if $\gcd(K) > 1$ then $\delta(H) \geq n/p + 2\alpha n/3$, where $p$ is the smallest prime factor of $\gcd(K) = \gcd(K_3)$. So we may apply Lemma~\ref{main1simple} with $K_3, \alpha/3$ and $|V'|$ in place of $K, \alpha$ and $n$ respectively to obtain a perfect $K_3$-packing in $H$. Together with the deleted copies of $K$ this gives a perfect $\{K_3, K\}$-packing of $G$, and $G$ therefore contains a perfect $K$-packing by Proposition~\ref{fpacktohpack}. 
\endproof

\section{Outline of the proofs} \label{sec:outline}

The proofs of Lemmas~\ref{main1simple} and~\ref{main2simple} use strong hypergraph regularity and the recent hypergraph blow-up lemma due to Keevash. The broad outline of how these are used will be familiar to those acquainted with the use of the blow-up lemma in graphs, but this method remains relatively novel for hypergraphs (for which there are many additional technicalities and subtleties). In this section we give a rough outline of how these proofs proceed. 

\subsection{Proof outline for Lemma~\ref{main1simple}}

The proof proceeds through the following steps.

\medskip \noindent \emph{Apply the Regular Approximation Lemma:}
The first step is to apply the Regular Approximation Lemma (Theorem~\ref{eq-partition}) to~$H$. This returns both a partition of $V(H)$ into `clusters' $U_1, \dots, U_m$, and a $k$-graph $G$ on $V(H)$, with the following properties. Firstly, $G$ is close to $H$, meaning that almost all edges of $G$ are edges of $H$, and vice versa. Secondly, this partition is regular for $G$, meaning (loosely speaking) that for the purposes of embedding small subgraphs in $G$, the $k$-partite subgraph of $G$ induced by any $k$-tuple of clusters behaves like a random $k$-graph of similar density. We write $Z = G \triangle H$, so $Z$ is a sparse graph which contains all the `bad' edges of $G$ which are not edges of~$H$; we will often choose copies of $K$ in $G \sm Z$, and by definition of $Z$ these copies are also in~$H$.

Having obtained $G$, $Z$ and the partition of $V(H)$ into $m$ clusters, we define a `reduced $k$-graph' $\R$ on $[m]$ (Definition~\ref{redgraphdef}). This has $m$ vertices, one corresponding to each cluster, and the edges of $\R$ are those $k$-tuples $S$ for which the corresponding clusters induce a dense $k$-partite subgraph of $G$ and a sparse $k$-partite subgraph of $Z$. Defined in this way, $\R$ `almost inherits' the minimum codegree condition of $H$, meaning that almost all $(k-1)$-tuples of vertices in $\R$ have almost the same degree (proportionately) as $(k-1)$-tuples in $H$ (Lemma~\ref{redgraphmindeg}).

\begin{figure}[t] 
\centering
\psfrag{A}{$V_1$} 
\psfrag{B}{$V_2$} 
\psfrag{C}{$V_3$} 
\psfrag{D}{$V_4$} 
\psfrag{E}{$1$}
\psfrag{F}{$2$}
\psfrag{G}{$3$} 
\psfrag{H}{$4$}
\psfrag{I}{$e$}
\psfrag{J}{$f$}
\includegraphics[width=6cm]{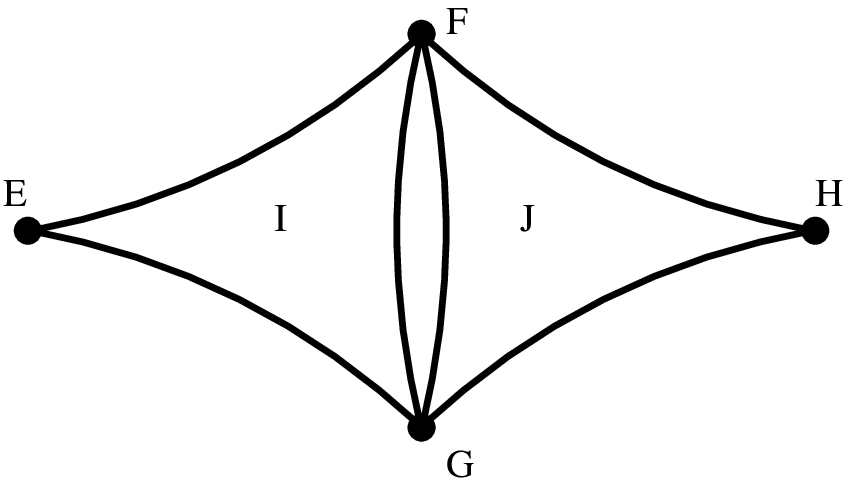} \hspace{2cm}
\psfrag{E}{$x_1$}
\psfrag{F}{$x_2$}
\psfrag{G}{$x_3$} 
\psfrag{H}{$x_4$}
\psfrag{I}{}
\psfrag{J}{}
\includegraphics[width=6cm]{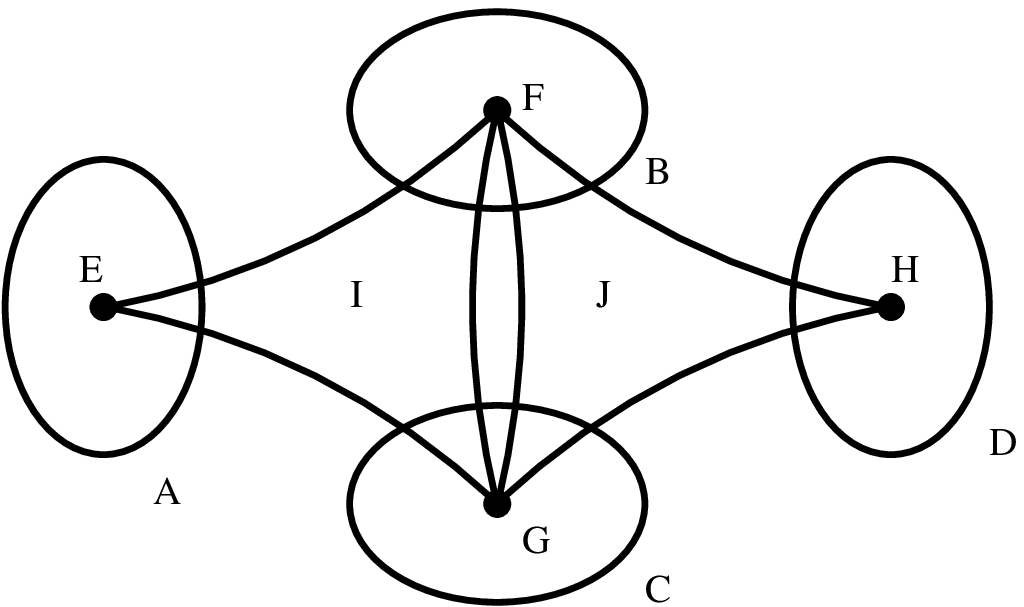}
\caption{The left hand diagram illustrates $\Phi$ in the case $k=3$, whilst the diagram on the right shows a copy of $\Phi$ located within four vertex classes $V_1, V_2, V_3$ and $V_4$. If there are many copies of $\Phi$ of the form shown, then we say that the triple $(1, S, 4)$ is $\Phi$-dense, where $S = \{2, 3\}$.} 
\label{fig:phi} 
\end{figure} 

One $k$-graph which will play an important role in our proof is the $k$-graph on $k+1$ vertices with two edges, which we denote by $\Phi$ (see Figure~\ref{fig:phi}). We call the $k-1$ vertices in the intersection of these edges the \emph{central vertices} of $\Phi$, and the remaining two vertices of $\Phi$ are the \emph{end vertices}. For vertices $i, j$ of $\R$ and a $(k-1)$-tuple $S$ of vertices of $\R$, we then say that the triple $(i, S, j)$ is \emph{$\Phi$-dense} if there are many copies of $\Phi$ in $G$ whose end vertices lie in the clusters $U_i$ and $U_j$ and whose central vertices lie in the clusters $U_{\ell}$ for $\ell \in S$, and the triple is \emph{$Z$-sparse} if there are few edges of $Z$ in either of the $k$-tuples of clusters corresponding to edges of $\Phi$. We are particularly interested in $\Phi$-dense and $Z$-sparse triples in $\R$, since we will be able to choose copies of $K$ within these triples with some flexibility over the number of vertices which are embedded in $U_i$ and $U_j$ (indeed, $k-1$ of the vertex classes of $K$ will be contained in the clusters $U_\ell$ for $\ell \in S$, whilst we will be able to choose how many vertices from the remaining vertex class are contained in each of $U_i$ and $U_j$). Also, to keep track of these useful triples we define a graph $\Sa$ on $[m]$, where each vertex corresponds to a cluster, and an edge $ij$ indicates the existence some $S$ for which the triple $(i, S, j)$ is $\Phi$-dense and $Z$-sparse. In Lemma~\ref{redgraphmindeg} we show that the condition $\delta(H) \geq n/p + o(n)$ yields a minimum codegree condition $\delta(\Sa) > m/p$ on $\Sa$. It follows that $\Sa$ has fewer than $p$ connected components if $\gcd(K) > 1$, a fact which plays a crucial role later in the proof. 

\medskip \noindent \emph{Refine the regularity partition into `lopsided groups':}
Our next step is to find an almost-perfect packing of $\R$ with a specific $k$-graph $\Akpq$ (where $p$ and $q$ are chosen depending on $\sigma(K)$). Lemma~\ref{akpqpacking} shows that the degree condition on $\R$ which is `inherited' from the condition $\delta(H) \geq \sigma(K) + \alpha n$ is sufficient to guarantee such a packing. For each copy $\A$ of $\Akpq$ in this packing we then use Lemma~\ref{akpqsplit} to partition the clusters covered by $\A$ into $kt$ `subclusters' $V^i_j$, which are labelled so that the subclusters $V^i_1, \dots, V^i_k$ are taken from clusters corresponding to an edge of $\A$ (\emph{i.e.} an edge of $\R$). This guarantees that the $k$-partite subgraphs $G^i$ and $Z^i$ induced by these $k$ subclusters are dense and sparse respectively. Furthermore, whilst it is unavoidable that these $k$ subclusters may have different sizes, the definition of $\Akpq$ will allow us to ensure that each $G^i$ has $\sigma(G^i) > \delta(H) - o(n) > \sigma(K) + o(n)$. That is, each $G^i$ is `less lopsided' than $K$. This is the point in the argument where this part of the minimum codegree assumption is used; a weaker condition would not suffice to guarantee the existence of subclusters with $\sigma(G^i) > \sigma(K)$.

Having obtained the $kt$ subclusters $V^i_j$, we define a $k$-graph $\R'$ and a graph $\Sa'$ on vertex set $[t] \times [k]$ to correspond to $\R$ and $\Sa$. Indeed, the vertex $(i, j)$ of $\R'$ and $\Sa'$ corresponds to the subcluster $(i, j)$, and $k$ subclusters form an edge of $\R'$ if the $k$ clusters from which these subclusters were taken form an edge of $\R$. Similarly, two subclusters form an edge of $\Sa'$ if the two clusters from which these subclusters were taken form an edge of $\Sa$. This allows us to retain information about the regularity partition when working with the subclusters. It follows from the definition of $\Sa'$ that the connected components of $\Sa'$ correspond to those of $\Sa$, so $\Sa'$ also has fewer than $p$ connected components if $\gcd(K) > 1$.

\medskip \noindent \emph{Obtain robustly-universal complexes:}
Next, for each $i$ we delete a small number of vertices from each subcluster $V^i_j$ so that the $k$-partite $k$-graph $G^i \sm Z^i$ restricted to the remaining vertices is `robustly universal'. This means that, even after the removal of a few more vertices, we can find any $k$-partite $k$-graph of bounded maximum degree in $G^i \sm Z^i \subseteq H$ which we can find in the complete $k$-partite $k$-graph on the same vertex set.  These deletions are achieved by Theorem~\ref{robust-universal}, a result of Keevash~\cite{K} which conceals the use of the hypergraph blow-up lemma. We also `put aside' a randomly-chosen set $X$ consisting of a small number of vertices from each subcluster; these vertices are immune from deletion over the next two steps, and ensure that every vertex which is not deleted lies in many edges which are not deleted, which is a requirement for the application of robust universality.

\medskip \noindent \emph{Delete a $K$-packing covering bad vertices:}
At this point we deal with the small number of `bad vertices', meaning those vertices in clusters which were not covered by our $\Akpq$-packing, as well as those vertices which were deleted to make the $k$-graphs $G^i \sm Z^i$ robustly universal. For this, Lemma~\ref{incorporateexcep} shows that for any vertex $v$ of $H$ there is a copy of $K$ in $H$ which contains $v$; this is a straightforward corollary of Theorem~\ref{turandensityzero} (the well-known result of Erd\H{o}s that $k$-partite $k$-graphs have Tur\'an density zero). Using this, we greedily choose and delete copies of $K$ in $H$ which cover all the bad vertices but which only cover a small number of vertices from each subcluster. Following these deletions, all remaining vertices of $H$ lie in $G^i \sm Z^i$ for some $i$. 

\medskip \noindent \emph{Delete a $K$-packing to ensure divisibility of cluster sizes:}
We now delete further copies of $K$ in $H$ so that, following these deletions, the number of vertices remaining in each subcluster is divisible by $bk \gcd(K)$ (recall that $b$ is the order of $K$). Lemma~\ref{gcdbalance} states that we can do this; loosely speaking, this is achieved by deleting a series of $K$-packings in $H$ to achieve successively stronger divisibility conditions on the subcluster sizes (a more detailed outline of the proof of this lemma is given in Section~\ref{sec:delete}). If $\gcd(K) > 1$, then it is crucial for this that, as stated above, $\Sa'$ has fewer than $p$ components. For example, the $k$-graph $G$ constructed for Proposition~\ref{extrem3} would yield a graph $\Sa'$ with $p$ components corresponding to the parts $V_1, \dots, V_p$, and the point of the construction is that it is not possible to delete a $K$-packing in $G$ so that every part has size divisible by $p$. Recall that $\Sa'$ has at most $p$ components since $\Sa$ had minimum codegree $\delta(\Sa) > m/p$, and that this in turn was inherited from the minimum codegree condition $\delta(H) \geq n/p + \alpha n$ of $H$. This is the point in the argument where this part of the minimum codegree assumption is used, and a weaker condition would not suffice. However, if $\gcd(K) = 1$ then we do not need this part of the minimum codegree assumption.

\medskip \noindent \emph{Blow-up a perfect $K$-packing in the remaining $k$-graph:}
Certainly a $K$-packing has bounded vertex degree, so our robustly universal $k$-graphs $G^i \sm Z^i$ each contain a perfect $K$-packing if and only if the complete $k$-partite $k$-graph on the same vertex set does also. To this end, Corollary~\ref{completepacking} shows that for those $k$-partite $k$-graphs $K$ which meet the conditions of Lemma~\ref{main1simple}, two properties are sufficient to ensure such a packing: firstly that $G^i \sm Z^i$ should be `less lopsided' than $K$, and secondly that each vertex class of $G^i \sm Z^i$ should have size divisible by $bk \gcd(K)$. Our partition into subclusters was chosen so that the first condition holds, whilst the final round of deletions described above ensures that the second  condition holds also. We can therefore find a perfect $K$-packing in $G^i \sm Z^i \subseteq H$ for each $i$; these $K$-packings, together with the deleted copies of $K$, form a perfect $K$-packing in $H$. This completes the outline of the proof of Lemma~\ref{main1simple}. 

\subsection{Proof outline for Lemma~\ref{main2simple}}
In this lemma $K$ is instead a complete $k$-partite $k$-graph $K$ whose vertex classes each have the same size $b_1$. The proof of this lemma proceeds through the same steps as the proof of Lemma~\ref{main1simple}, though there are two principal differences. Firstly, rather than finding an $\Akpq$-packing in $\R$, we can now find simply a matching $M_\R$ in $\R$ which covers almost all of the vertices of $\R$. In consequence, there is no need to divide the clusters into subclusters, or to define $\R'$ and $\Sa'$; we simply continue working with the clusters and the $k$-graph $\R$ and graph $\Sa$. The second principal difference is that a complete $k$-partite $k$-graph $G$ contains a perfect $K$-packing if and only if every vertex class of $G$ has equal size and this common size is divisible by $b_1$. So to find a perfect $K$-packing in our robustly universal $k$-graphs $G^i \sm Z^i$ in the final step of the proof, it is not sufficient to delete copies of $K$ in the penultimate step such that every cluster of $G^i \sm Z^i$ has size divisible by $bk\gcd(K)$; we must now ensure also that these clusters have the same size for any $i$. As a consequence we must be more precise in our definition of $\Sa$. Indeed, we now define a directed graph $\Sa^+$ whose vertices correspond to clusters, and whereas before an edge $ij \in S$ indicated the existence of some $(k-1)$-tuple $S$ for which $(i, S, j)$ is $\Phi$-dense and $Z$-sparse, we now only have an edge $i \to j$ of $\Sa^+$ if this is true for $S = e(i) \sm \{i\}$, where $e$ is the edge of $M_\R$ which contains $i$. Then, similarly as before, the minimum degree condition $\delta(H) \geq n/2 + \alpha n$ implies that $\Sa^+$ has minimum outdegree $\delta^+(\Sa^+) > m/2$.

It would be possible to proceed by considering the directed graph $\Sa^+$, but there are a number of additional problems which would arise in this case. Instead, we make use of the notion of `irreducibility' of $k$-graphs containing a perfect matching, which was introduced by Keevash and Mycroft~\cite{KM}, and is presented in Section~\ref{sec:irreduc}. Using the stronger minimum degree condition of Lemma~\ref{main2simple}, we can insist that the reduced $k$-graph $\R$ is irreducible on the matching $M_\R$. Then, in Section~\ref{sec:deleteeq} we show that, under this assumption, we need only consider the undirected base graph $\Sa$ of $\Sa^+$ for the purposes of deleting $K$-packings to adjust cluster sizes. Using this, we prove Lemma~\ref{samesizebalance}, which shows that it is indeed possible to delete a $K$-packing in $H$ so that, following these deletions, the $k$ clusters corresponding to any edge of $M_\R$ have equal size. We then use Lemma~\ref{samesizebalance} in place of Lemma~\ref{gcdbalance} in the proof of Lemma~\ref{main2simple}; all other steps of the proof proceed roughly as before. Note, however, that our use of irreducibility requires that almost all $(k-1)$-tuples $S$ of vertices of $\R$ have $\deg_\R(S) > m/k$ (this condition is inherited from the minimum codegree of $H$), so it would not be possible to use this approach in the proof of Lemma~\ref{main1simple}.

\section{Regularity and the Blow-up Lemma} \label{sec:regularity}

Much of the notation introduced in this section was first introduced by R\"odl and Skokan~\cite{RSk} and by R\"odl and Schacht~\cite{RS, RS2}.

\subsection{Hypergraphs, complexes and partitions}

A \emph{hypergraph} $H$ consists of a vertex set $V(H)$ and an edge set $E(H)$, where every edge of $H$ is a set of vertices of $H$. So a $k$-graph (as defined in
Section~\ref{intro}) is a hypergraph in which all the edges have size $k$. As with $k$-graphs we frequently identify a hypergraph $H$ with the set of its edges. So, for example, $e \in H$ means that $e$ is an edge of $H$, and $|H|$ is the number of edges in~$H$. Likewise, if $G$ and $H$ are hypergraphs on a common vertex set $V$ then the hypergraph $G \sm H$ is the hypergraph on $V$ formed by removing from $G$ any edge which also lies in $H$. For any hypergraph $H$ and any $U \subseteq V(H)$, the \emph{restriction of $H$ to $U$}, denoted $H[U]$ is the hypergraph on vertex set $U$ whose edges are those edges of $H$ which are subsets of $U$. Also, recall that if $H$ is a hypergraph with vertex set $V$, the \emph{degree} $\deg_H(S)$ of a set $S \subseteq V(H)$ is defined to be the number of edges of $H$ which contain $S$ as a subset. 
The \emph{maximum vertex degree} of $H$, denoted $\Delta_\vx(H)$, is then defined to be the maximum of $\deg_H(\{v\})$ taken 
over all vertices $v \in V(H)$; so every vertex of $H$ is contained in at most $\Delta_\vx(H)$ edges of $H$. For any set of vertices $X$, we write $K(X)$ for the complete hypergraph on vertex set $X$, that is, the edges of $K(X)$ are all subsets of $X$.

Now let $X$ be a set of vertices, and let $\Qart$ be a partition of $X$ into $r$ parts $X_1, \dots, X_r$. We say that a subset $S \subseteq X$ is \emph{$\Qart$-partite} if $|S \cap X_i| \leq 1$ for any $i \in [r]$. Similarly, we say that a hypergraph $H$ on~$X$ is \emph{$\Qart$-partite} if every edge of $H$ is $\Qart$-partite, and we refer to the parts $X_i$ of $\Qart$ as the \emph{vertex classes} of $H$. We say that $H$ is \emph{$r$-partite} if it is $\Qart$-partite for some partition $\Qart$ of $X$ into $r$ parts. For any $\Qart$-partite set $S \subseteq X$ we define the \emph{index} of $S$ to be $i(S) := \{i \in [r] : |S \cap X_i| = 1\}$. So $S$ intersects precisely those $X_i$ for which $i \in i(S)$. Likewise, for any $A \subseteq [r]$ we write $X_A := \bigcup_{i \in A} X_i$, and define $H_A$ to be the $|A|$-graph with vertex set $X_A$ whose edges are all edges of $H$ of index $A$ (note that $H_A$ is naturally $|A|$-partite with vertex classes $X_i$ for $i \in A$). In particular, $K(X)_A$ is the complete $|A|$-partite $|A|$-graph with vertex classes $X_i$ for $i \in A$.

A \emph{$k$-complex} $J$ is a hypergraph in which every edge has size at most $k$ and which has the property that if $e_1 \in J$ and $e_2 \subseteq e_1$ then $e_2 \in J$ (so the edges of $J$ form a simplicial complex). We refer to edges of size $i$ as \emph{$i$-edges}, and write $J_i$ for the $i$-graph on $V(J)$ formed by the $i$-edges of $J$.
Informally, it may be helpful to think of a $k$-complex $J$ as consisting of `layers' $J_i$ for $0 \leq i \leq k$. So any edge $e$ in the `$i$th layer' $J_i$ of $J$ lies `above' $i$ edges of $J$ in the `$(i-1)$th layer', namely those subsets of $e$ of size $i-1$. The `top layer' of a $k$-complex $J$ will play a particularly important role; due to this we often write $J_=$ in place of $J_k$ to emphasise that this is the `top layer'. So $J_=$ is a $k$-graph on $V(J)$. For any $k$-graph $H$ we can naturally generate a $k$-complex $H^{\leq}$ on $V(H)$, whose edges are all subsets of edges of $H$. Observe in particular that $(H^{\leq})_= = H$, and also that if $H$ is $\Qart$-partite for some partition $\Qart$ of $V(H)$ then $H^{\leq}$ is $\Qart$-partite also.

Now suppose again that $\Qart$ partitions a set of vertices $X$ into $r$ parts $X_1, \dots, X_r$, and that $J$ is a hypergraph on $X$. For any $A \subseteq [r]$, the \emph{absolute density of $J$ at $A$}, denoted $d(J_A)$, is the proportion of edges of $K(X)_A$ which are also edges of $J_A$.  So
$$d(J_A) := \frac{|J_A|}{|K(X)_A|} =  \frac{|J_A|}{\prod_{i\in A} |X_i|}.$$
If $J$ is a $k$-complex then we also have the notion of \emph{relative density}. Indeed, the \emph{relative density of $J$ at $A$} is the proportion of those edges which could feasibly be in $J_A$ (in the sense that they are supported by `lower levels' of $J$) which are actually edges of $J_A$. More precisely, we write $J^*_A$ for the set of all edges $e \in K(X)_A$ such that every proper subset $e' \subset e$ is an edge of $J$. So $J^*_A$ is the set of edges which could feasibly be in $J_A$ (given the `lower levels' of $J$), and we define the relative density of $J$ at index $A$ to be
$$d_A(J) := \frac{|J_A|}{|J^*_A|}.$$
(If the set $J^*_A$ is empty then we instead define $d_A(J)$ to be zero.)

\subsection{Partition complexes}

Loosely speaking, the Regular Approximation Lemma states that 
any $k$-graph is close to another $k$-graph which can be split into pieces, each of which forms the `top level' of a regular $k$-complex. For a graph $G$, this split involves simply a partition of the vertex set into a number of `clusters', whereupon the edges between any pair of clusters form a regular bipartite graph. However, for a $k$-graph $H$ (for $k \geq 3$) we must not only partition the vertices of $H$, but also the pairs of vertices of $H$, the triples of vertices of $H$, and so forth, up to $(k-1)$-tuples of vertices of $H$. To keep track of these partitions we need the notion of a \emph{partition complex}, which we now introduce.

Let $X$ be a set of vertices, and let $\Qart$ partition $X$ into parts $X_1, \dots, X_r$. Recall that for any $A \subseteq [r]$, $K(X)_A$ consists of all $|A|$-tuples of vertices of $X$ with index $A$. A \emph{partition $k$-system~$\Part$ on $X$} consists of a partition $\Part_A$ of the edges of $K(X)_A$ for each $A \subseteq [r]$ with $|A| \leq k$. We refer to the partition classes of $\Part_A$ as \emph{cells}. So every edge of $K(X)_A$ is contained in precisely one cell of $\Part_A$. 
We say that $\Part$ is \emph{$a$-bounded} if for each $A$ the partition $\Part_A$ has at most $a$ cells.
Also, for any $j \in [k]$ we write $$\Part^{(j)} = \bigcup_{A \in \binom{[r]}{j}} \Part_A,$$ so $\Part^{(j)}$ is a partition of the set of all $\Qart$-partite $j$-tuples of vertices of $X$. Note in particular that $\Part^{(1)}$ is a partition of the vertex set $X$ which refines $\Qart$.
We refer to the cells of $\Part^{(1)}$ as \emph{clusters} of $\Part$, so each cluster is a subset of some $X_i$, and every vertex of $X$ lies in some cluster of $\Part$. We say that $\Part$ is \emph{vertex-equitable} if every cluster of $\Part$ has equal size. Also, for any $\Qart$-partite set $S \subseteq X$ with $|S| \leq k$, we write $\Cell(S)$ to denote the cell of $\Part$ which contains $S$. 

We say that $\Part$ is a \emph{partition $k$-complex on $X$} if it is a partition $k$-system on $X$ with the additional property that for any edges $S, S' \in K(X)_A$ with $\Cell(S) = \Cell(S')$ and any subset $B \subseteq A$ we have $\Cell(S \cap X_B)=\Cell(S' \cap X_B)$. That is, if two sets lie in the same cell of $\Part$, then their subsets of any given index also lie in the same cell of $\Part$. To illustrate this definition, consider the following example of a partition $3$-complex, where we slightly abuse notation in subscripts by writing, for example, $\Part_{12}$ rather than $\Part_{\{1, 2\}}$. Take $X = X_1 \cup X_2 \cup X_3$, and let the vertex classes $X_1$, $X_2$ and $X_3$ also be the clusters of $\Part$ (but bear in mind it is also possible for each vertex class to be partitioned into several clusters). Then the partition $\Part^{(1)}$ is simply the partition of $X$ into the clusters $X_1, X_2$ and $X_3$. Next, $\Part_{12}$ is a partition of the set of all pairs $\{x_1, x_2\}$ with $x_1 \in X_1$ and $x_2 \in X_2$. That is, the cells $C^i_{12}$ of $\Part_{1 2}$ are edge-disjoint bipartite graphs with vertex classes $X_1$ and $X_2$, whose union is the complete bipartite graph on $X_1$ and $X_2$. Similarly, the cells $C^j_{13}$ of $\Part_{13}$ are bipartite graphs with vertex classes $X_1$ and $X_3$, and the cells $C^\ell_{23}$ of $\Part_{23}$ are bipartite graphs with vertex classes $X_2$ and $X_3$. Now, for any choice of cells $C^i_{12}$, $C^j_{13}$ and $C^\ell_{23}$ from $\Part_{12}$, $\Part_{13}$ and $\Part_{23}$ respectively, the union of these cells is a tripartite graph; let $\Delta_{ij\ell}$ be the set of triangles in this tripartite graph. Observe that these sets $\Delta_{ij\ell}$ partition the set of all triples $\{x_1, x_2, x_3\}$ with $x_1 \in X_1$, $x_2 \in X_2$ and $x_3 \in X_3$; indeed any triple $\{x_1, x_2, x_3\}$ appears only in the $\Delta_{ij\ell}$ such that $\{x_1, x_2\} \in C^i_{12}, \{x_1, x_3\} \in C^j_{13}$ and $\{x_2, x_3\} \in C^\ell_{23}$. Finally, $\Part_{123}$ is also a partition of the set of all triples $\{x_1, x_2, x_3\}$ with $x_1 \in X_1$, $x_2 \in X_2$ and $x_3 \in X_3$; the requirement that $\Part$ is a partition $k$-complex requires that $\Part_{123}$ is a refinement of the partition into sets $\Delta_{ij\ell}$.

Suppose that $\Part$ is a partition $k$-complex on $X$. For any $\Qart$-partite set $e \in \binom{X}{k}$, define $\Part(e) := \bigcup_{e' \subseteq e} \Cell(e')$. Then the fact that $\Part$ is a partition $k$-complex implies that $\Part(e)$ is a $k$-partite $k$-complex with vertex classes $X_j$ for $j \in i(e)$. Loosely speaking, the Regular Approximation Lemma will provide us with a partition $k$-complex so that all these $k$-complexes $\Part(e)$ are regular complexes (as defined in the next section). Now suppose instead that $\Part$ is a partition $(k-1)$-complex on $X$, and recall that $X$ is partitioned into $r$ parts $X_1, \dots, X_r$. Then for any $A \in \binom{[r]}{k}$ the cells of $\Part$ naturally generate a partition $\hat{\Part}_A$ of the edges of $K(X)_A$. Indeed, we say that edges $S$ and $S'$ in $K(X)_A$ are \emph{weakly equivalent} if $\Cell(S_B) = \Cell(S'_B)$ for any $B \subsetneq A$. This defines an equivalence relation on $K(X)_A$; we take the equivalence classes of this relation to be the parts of $\hat{\Part}_A$. We can then extend $\Part$ to a partition $k$-complex $\hat\Part$ on $X$ by adding the partitions  $\hat{\Part}_A$ for $A \in \binom{[r]}{k}$ to $\Part$. That is, for any $A \subseteq [r]$ with $|A| < k$ the cells of $\hat{\Part}_A$ are the cells of $\Part$ of index $A$, and for any $A \in \binom{[r]}{k}$ the cells of $\hat{\Part}_A$ are the equivalence classes of the weak equivalence relation on $K(X)_A$. We refer to $\hat\Part$ as the \emph{partition generated from $\Part$ by weak equivalence}. In particular, if $\Part$ is $a$-bounded, then $\Part_A$ has at most $a$ cells for each $A \in \binom{[r]}{k-1}$, so $\hat\Part$ is $a^k$-bounded. In a similar manner, for any $\Qart$-partite $k$-graph $G$ on $X$ we can generate a partition $k$-complex $G[\hat{\Part}]$ on $X$ from $\hat{\Part}$ by refining the partitions $\hat{\Part}_A$ for each $A \in \binom{[r]}{k}$. Indeed, for each such $A$ and each cell $C$ of $\hat{\Part}_A$ we have two cells of $G[\hat{\Part}]_A$, namely $G \cap C$ and $C \sm G$, whilst for any $A \in \binom{[r]}{\leq k-1}$, the cells of $G[\hat{\Part}]_A$ are the same as those of $\hat{\Part}_A$.

\subsection{Hypergraph regularity}

We now have all of the notation that we need to explain the notion of a regular complex and state the Regular Approximation Lemma we shall use. The concept of regularity with which we shall work was first introduced in the $k$-uniform case by R\"odl and Skokan~\cite{RSk}, but we shall consider it in the form used by R\"odl and Schacht~\cite{RS, RS2}.

Roughly speaking, an $r$-partite $k$-complex $J$ is $\eps$-regular if whenever we restrict $J$ to those edges supported by a large subcomplex of $J \sm J_=$ (that is, $J$ minus its `top layer'), the resulting $k$-complex has similar densities to $J$. To demonstrate this, we shall first consider graphs (\emph{i.e.} 2-graphs). If $G$ is a bipartite graph with vertex classes $V_1$ and $V_2$, then the standard definition of $\eps$-regularity of $G$ is that for any $V_1' \subseteq V_1$ and $V_2' \subseteq V_2$ with $|V_1'| > \eps |V_1|$ and $|V_2'| > \eps |V_2|$ we have $d(G[V_1' \cup V_2']) = d(G) \pm \eps$. However, the definition of regularity which we generalise to hypergraphs is subtly different. Indeed, we say that that $G$ is $\eps$-regular if for any $V_1' \subseteq V_1$ and $V_2' \subseteq V_2$ with $|V_1'||V_2'| > \eps |V_1||V_2|$ we have $d(G[V_1' \cup V_2']) = d(G) \pm \eps$ (note that this is equivalent to the previous definition in the sense that $\eps$-regularity in the former implies $\eps$-regularity in the latter, whilst $\eps^2$-regularity in the latter implies $\eps$-regularity in the former). Now consider the $2$-partite $2$-complex $J$ with edge set $\{\emptyset\} \cup \{\{v\} : v \in V(G)\} \cup G$, so the `layers' of $J$ are $\{\emptyset\}$, $\{\{v\} : v \in V(G)\}$ and $G$. Then saying that $G$ is $\eps$-regular (under the latter definition) is equivalent to saying that $J$ is $\eps$-regular under the following definition: $J$ is $\eps$-regular if, for any subcomplex $L \subseteq J$ with $|L^*_{\{1, 2\}}| \geq \eps |J^*_{\{1, 2\}}|$, we have $|J_{\{1, 2\}} \cap L_{\{1, 2\}}^*| / |L_{\{1, 2\}}^*| = d_{\{1, 2\}}(J) \pm \eps$. Indeed, using the correspondence $V_j' = \{v \in V_j : \{v\} \in L_{\{j\}}\}$ for $j \in \{1, 2\}$ we find that the two definitions are equivalent, since then $|L^*_{\{1, 2\}}| = |V_1'||V_2'|$ and $|J^*_{\{1, 2\}}| = |V_1||V_2|$.

In general, let $\Qart$ partition a set $X$ into $r$ parts $X_1, \dots, X_r$, and let $J$ be a $\Qart$-partite $k$-complex. Then we generalise the definition above as follows: for any $A \in \binom{[r]}{\leq k}$ we say that $J$ is \emph{$\eps$-regular at $A$} if for 
any subcomplex $L \subseteq J$ with $|L^*_A| \geq \eps |J^*_A|$ we have
$$
\frac{|J_A \cap L_A^*|}{|L_A^*|} = d_A(J) \pm \eps.
$$ 
We say $J$ is \emph{$\eps$-regular} if $J$ is $\eps$-regular at $A$ for every $A \in \binom{[r]}{\leq k}$. Now suppose that $\Part$ is a partition $k$-complex on $X$. Recall that for any $\Qart$-partite set $e \in \binom{X}{k}$ the partition $k$-complex $\Part$ naturally yields a $k$-partite $k$-complex $\Part(e)$ with vertex classes $X_j$ for $j \in i(e)$; we say that $\Part$ is \emph{$\eps$-regular} if $\Part(e)$ is $\eps$-regular for any $\Qart$-partite set $e \in \binom{X}{k}$.

Let $G$ and $H$ be $r$-partite $k$-graphs with common vertex classes $X_1, \dots, X_r$, and let $X := X_1 \cup \dots \cup X_r$. Then we say that $G$ and $H$ are \emph{$\xi$-close} if $|G_A \triangle H_A| < \xi |K_A(X)|$ for every $A \in \binom{[r]}{k}$. The Regular Approximation Lemma states that for any $r$-partite $k$-graph $H$ there is an $r$-partite $k$-graph $G$ on $V(H)$ (with the same vertex partition as $H$) and a partition $(k-1)$-complex $\Part$ on $V(H)$ such that $G$ is $\xi$-close to $H$ and the partition $k$-complex $G[\hat{\Part}]$ is $\eps$-regular. This will suffice for our purposes as we shall avoid using any edge of $G \sm H$ whilst working with $G$, so any edge we do use will be an edge of $H$. There are other  regularity lemmas for $k$-graphs which give information on $H$ itself (see~\cite{G1,RSk}) but the regular complexes yielded by these are not sufficiently dense to apply the blow-up lemma (see~\cite[Section 3]{K} for further discussion of this point). The next theorem is the Regular Approximation Lemma; this is a slight restatement of a result of R\"odl and Schacht (Theorem 14 of~\cite{RS}).

\begin{theo}[Regular Approximation Lemma, \cite{RS}]\label{eq-partition} 
Suppose that integers $n,a,r,k$ and reals $\eps, \xi$ satisfy $1/n \ll\eps \ll 1/a \ll \xi, 1/r, 1/k$ and that $a!r$ divides~$n$. Let $\Qart$ partition a set $X$ of $n$ vertices into $r$ parts of equal size, and let $H$ be a $\Qart$-partite $k$-graph on $X$. Then there is an $a$-bounded $\eps$-regular vertex-equitable partition $(k-1)$-complex $\Part$ on $X$ and a $\Qart$-partite $k$-graph $G$ with vertex set $X$ such that $G$ is $\xi$-close to $H$ and $G[\hat{\Part}]$ is $\eps$-regular.
\end{theo}

One useful property of regularity if that if $G$ is regular and dense, then the restriction of $G$ to any not-too-small subsets of its vertex classes is also regular and dense. The following lemma (a weakened version of Theorem~6.18 in~\cite{K}) states this more precisely.

\begin{lemma} [Regular restriction, \cite{K}] \label{regularrestriction}
Suppose that $1/n \ll \eps \ll d, c, 1/k$. Let $J$ be an $\eps$-regular $k$-partite $k$-complex with vertex classes $X_1, \dots, X_k$ such that $d(J_{[k]}) \geq c$, $d_{[k]}(J) \geq d$ and $|X_j| \geq n$ for each $j \in [k]$. Also, for each $j \in [k]$, let $X_j' \subseteq X_j$ have $|X_j'| \geq \eps^{1/2k} |X_j|$, and let $J' := J[X_1' \cup \dots \cup X_k']$. Then $J'$ is $\sqrt{\eps}$-regular, $d(J'_{[k]}) \geq c/2$ and $d_{[k]}(J') \geq d/2$.
\end{lemma}

\subsection{Robustly universal complexes.}

Another vital tool in the proofs of Lemmas~\ref{main1simple} and~\ref{main2simple} is the recent hypergraph blow-up 
lemma of Keevash~\cite{K}. This states that if an $r$-partite $k$-complex $J$ is `super-regular' (a stronger property than regularity), then $J$ contains a copy of any $r$-partite $k$-complex $L$ with the same vertex classes and small maximum vertex degree. Another result in~\cite{K} shows that any regular and dense $r$-partite $k$-complex $J$ can be made super-regular by the deletion of a few vertices from each vertex class. However, the notion of hypergraph super-regularity is very technical, so we shall avoid these technicalities through the related notion of `robust universality', also from~\cite{K}. Roughly speaking, we say that an $r$-partite $k$-complex $J'$ is robustly universal if even after the deletion of many vertices of $J'$, the resulting complex $J$ has the property that one can find in $J$ a copy of any $r$-partite $k$-complex $L$ with the same vertex classes as $J$ and small maximum vertex degree. The next definition states this property formally; for this we make the following definitions. Let $R$ be a $k$-complex on vertex set $[r]$, and let $\Qart$ partition a set $X$ into parts $X_1, \dots, X_r$. Then a $\Qart$-partite $k$-complex $J$ on $X$ is \emph{$R$-indexed} if every edge $e \in J$ has $i(e) \in R$ (recall that $i(e)$ denotes the index of $e$). Also, for any $S \in R_=$, any $j \in S$ and $v \in X_j$, we write $J_S(v)$ for the $(k-1)$-partite $(k-1)$-complex with vertex set $\bigcup_{j \in S \sm \{i\}} X_j$ and whose edges are those $(k-1)$-tuples $e'$ of vertices such that $e' \cup \{v\} \in J_S$.

\begin{defn}[Robustly universal complexes, \cite{K}]
Suppose that $R$ is a $k$-complex on vertex set $[r]$, and that $J'$ is an $r$-partite $k$-complex with vertex classes $V'_1, \dots V'_r$ such that $|J'_{\{i\}}| = |V'_i|$ for each $i \in [r]$. Then we say that 
\begin{itemize}
\item $J'$ is \emph{$D$-universal on $R$} if for any $R$-indexed $r$-partite $k$-complex $L$ with vertex classes $U_1, \dots, U_r$ such that $|U_j| \leq |V_j'|$ for all $j \in [r]$ and $\Delta_\vx(L) \leq D$, there is a copy of $L$ in $J'$ in which the vertices of $U_j$ correspond to the vertices of $V_j'$. 
\item $J'$ is \emph{$\eta$-robustly $D$-universal on $R$} if for any sets $V_j \subseteq V'_j$ such that $|V_j| \ge \eta|V'_j|$ for any $j \in [r]$ and $|J_S(v)| \ge \eta|J'_S(v)|$ for any $S \in R_=$ and $v \in V_S$, where $J=J'[\bigcup_{j \in [r]} V_j]$, the $r$-partite $k$-complex $J$ is $D$-universal on $R$. 
\end{itemize}
In the case where $R$ has $k$ vertices and is formed by the downwards closure of a single edge, we omit `on $R$' and write simply `$D$-universal' or `$\eta$-robustly $D$-universal'.
\end{defn}
\medskip

Note that the definition of robust universality given here is weaker than that from~\cite{K} in two ways. Firstly, the definition there allows $R$ to be a so-called `multicomplex', allowing us to distinguish between edges of $J$ with the same index. Secondly, the definition in~\cite{K} also permits us to choose for a small number of vertices $v \in V(L)$ a small `target set' into which $v$ is to be embedded; an additional parameter $c_0$ governs how small these `target sets' can be. However, we do not need either of these strengthenings.

Clearly robust universality is a very strong property, and so the main difficulty in the use of robust universality lies in obtaining robustly universal complexes in the first place. For this purpose we have the following theorem, which is a weakened version of Theorem~6.32 in~\cite{K} (to correspond to our weakened definition of robustly universal complexes). It states that if $J$ is a regular $k$-complex which is dense on edges of $R$, and $Z$ is a $k$-graph which has few edges in common with $J_=$, then we may delete a small number of vertices from each vertex class of $J$ so that the subcomplex of $J \sm Z$ induced by the remaining vertices is robustly universal on $R$. Our use of the blow-up lemma is therefore concealed in this theorem, which we have slightly restated from the form in~\cite{K} in that the statements in (i) apply to $J' \sm Z$ rather than to $J'$. The proof of this theorem in~\cite{K} in fact gives this altered result; alternatively, it can be derived by first deleting vertices of $J$ which lie in atypically few edges of $J_=$ or in atypically many edges of $Z$, and then applying the form of the theorem stated in~\cite{K} (although the deletion step here is redundant, since these atypical vertices are deleted in the proof of this theorem in~\cite{K}).

\begin{theo}[\cite{K}]\label{robust-universal} 
Suppose that $$1/n \ll 1/r' \ll \eps \ll d^* \ll d_a \ll \nu \ll d, \eta, 1/k, 1/D, 1/C, 1/D_R,$$ and that $r \leq r'$.
Let $R$ be a $k$-complex on $[r]$ with $\Delta_\vx(R) \leq D_R$, and let $J$ be an $r$-partite $k$-complex with vertex classes $V_1, \dots, V_r$, such that $n \leq |J_{\{j\}}| = |V_j| \leq Cn$ for every $j \in [r]$. Also let $Z$ be a $k$-graph on $V(J)$, and suppose that 
\begin{enumerate}[(a)]
\item $J$ is $\eps$-regular,
\item $d_{S}(J) \ge d$ and $d(J_S) \ge d_a$ for any $S \in R_=$, 
\item $|Z \cap J_S| \leq \nu |J_S|$ for any $S \in R_=$. 
\end{enumerate}
Then we can can delete at most $2\nu^{1/3} |V_j|$ vertices from each set $V_j$ to obtain subsets $V'_j$ so that, writing $V' = V_1' \cup  \dots \cup V_r'$ and $J' = J[V']$, we have
\begin{enumerate}[(i)]
\item $d((J' \sm Z)_S) > d^*$ and $|(J' \sm Z)_S(v)| > d^*|(J' \sm Z)_S|/|V'_j|$ for every $S \in R_=$, $j \in S$ and $v \in V'_j$, and
\item $J' \sm Z$ is $\eta$-robustly $D$-universal on $R$. 
\end{enumerate}
\end{theo}

Having obtained a robustly universal $k$-partite $k$-complex $J$, we will then delete further vertices of $J$, and we wish these deletions to preserve the property that $J$ is universal. The following proposition allows us to do this without difficulty; we shall only delete vertices which do not lie in the sets $X_j$.

\begin{prop} \label{randomsplitkeepsmatching2} Suppose that $1/n \ll d^* \ll \eta \ll \beta, 1/D, 1/k$. Let $J$ be a $k$-partite $k$-complex with vertex classes $W_1, \dots, W_k$ which is $\eta$-robustly $D$-universal and which satisfies $d(J_{[k]}) > d^*$ and $|J_{[k]}(v)| > d^*|J_=|/|W_j|$ for every $j \in [k]$ and $v \in W_j$. Suppose also that $\beta n \leq s_j \leq |W_j| \leq n$ for each $j \in [k]$ and some integers $s_j$. For each $j \in [k]$ choose a subset $X_j \subseteq W_j$ of size $s_j$ uniformly at random and independently of all other choices. Then with probability $1-o(1)$ we have the property that for any sets $Y_j$ with $X_j \subseteq Y_j \subseteq W_j$ for each $j \in [k]$, the induced $k$-complex $J[\bigcup_{j \in [k]} Y_j]$ is $D$-universal.
\end{prop} 

\begin{proof}
Observe that for any such sets $Y_j$ we have $|Y_j| \geq |X_j| = s_j \geq \beta n \geq \eta |W_j|$ for any $j \in [k]$, and that $|J_=[Y](v)| \geq |J_=[X \cup \{v\}](v)|$ for every $v \in Y$, where we define $Y := \bigcup_{j \in [k]} Y_j$ and $X = \bigcup_{j \in [k]} X_j$. So by definition of a $\eta$-robustly $D$-universal complex it suffices to show that with probability $1-o(1)$ we have the property that $|J_=[X \cup \{v\}](v)| \geq \eta|J_=(v)|$ for every $v \in W := \bigcup_{j \in [k]} W_j$. In fact, Lemma 4.4 of~\cite{KKMO}, which was proved by a straightforward application of Azuma's inequality, states that for any $v \in W$ this inequality holds with probability at least $1 - 1/n^2$, so taking a union bound over all vertices of $W$ proves the result.
\end{proof}

\subsection{The reduced $k$-graph} \label{sec:redgraph}
In this section we introduce the idea of the \emph{reduced $k$-graph}, for which we make use of the $k$-graph $\Phi$ defined in Section~\ref{sec:outline}. Recall that $\Phi$ has vertex set $[k+1]$, and has two edges, $\{1, \dots, k\}$ and $\{2, \dots, k+1\}$. Also recall that the vertices $1$ and $k+1$ are the end vertices of $\Phi$, and the vertices $2, \dots, k$ are the central vertices of $\Phi$. Our definition of the reduced $k$-graph $R$ will enable us, given a copy of $\Phi$ in $R$, to find a $(k+1)$-partite $k$-complex which is universal on this copy of $\Phi$. The next proposition shows that within this $k$-complex we can find many copies of $\Phi(m)$, the \emph{$m$-fold blowup} of $\Phi$, which is the $(k+1)$-partite $k$-graph with vertex classes $L_1, \dots, L_{k+1}$ of size $m$ and whose edges are all $k$-tuples of vertices whose index is an edge of $\Phi$. Copies of $\Phi(m)$ are particularly useful since we have flexibility over how a $k$-partite $k$-graph $K$ can be embedded within $\Phi(m)$: we can embed $k-1$ of the vertex classes of $K$ in the central vertex classes of $\Phi(m)$ (that is, those vertex classes corresponding to central vertices of $\Phi)$, and then the vertices of the remaining vertex class of $K$ can be distributed as we choose among the two end vertex classes of $\Phi(m)$ (that is, those vertex classes corresponding to end vertices of $\Phi$).

\begin{prop} \label{varykcomp}
Let $J$ be a $(k+1)$-partite $k$-complex with vertex classes $X_1, \dots, X_{k+1}$ which is $D$-universal on $\Phi$, where $D \geq 2^{(k+1)m}$, and suppose that $|X_i| \geq n$ for each $i \in [k+1]$. Then there are at least $\lfloor n/m \rfloor$ vertex-disjoint copies of $\Phi(m)$ in $J_=$ whose end vertex classes lie in $X_1$ and $X_{k+1}$, and whose central vertex classes lie in $X_2, \dots, X_k$.
\end{prop}

\begin{proof}
Let $L$ be the $(k+1)$-partite $k$-complex formed by the downwards closure of $\lfloor n/m \rfloor$ vertex-disjoint copies of $\Phi(m)$; equivalently, $L$ consists of $\lfloor n/m \rfloor$ vertex-disjoint copies of $\Phi(m)^\leq$. Since each copy of $\Phi(m)$ has $(k+1)m$ vertices we have $\Delta_\vx(L) \leq 2^{(k+1)m}$. Together with the fact that $J$ is $D$-universal on $\Phi$, it follows that $J$ contains a copy of $L$ in which the end vertex classes of each copy of $\Phi(m)^\leq$ lie in $X_1$ and $X_{k+1}$ and whose central vertex classes lie in $X_2, \dots, X_k$. Then $L_= \subseteq J_=$ consists of the desired copies of $\Phi(m)$.
\end{proof}

For notational simplicity, for the rest of this section we work within the following setup. 

\begin{setup} \label{redsetup}
Fix integers $n, a, r, D$ and $k$ and constants $\eps, d^*, \xi, \nu, \mu, c, \eta, \theta$ and $\gamma$ with 
$$1/n \ll \eps \ll d^* \ll 1/a \ll 1/r, \xi \ll \nu \ll \mu \ll c, \eta \ll \theta \ll \gamma,  1/D, 1/k.$$
Let $X$ be a set of $n$ vertices, and let $\Qart$ be a partition of $X$ into $r$ parts $T_1, \dots, T_r$ of equal size. Let $\Part$ be an $a$-bounded $\eps$-regular vertex-equitable partition $(k-1)$-complex on $X$ such that the partition $\Part^{(1)}$ of $X$ into clusters $X_1, \dots, X_m$ refines $\Qart$. Assume that the number of clusters $m$ satisfies $r \leq m \leq ar$, and let $n_1 = n/m$ be the common size of each cluster. Finally let $G$ and $Z$ be $\Qart$-partite $k$-graphs on $X$, such that the partition $k$-complex $G[\hat{\Part}]$ is $\eps$-regular.
\end{setup} 

We can now give our definition of the reduced $k$-graph $R$. Similarly as in previous applications of hypergraph regularity, $R$ has vertices corresponding to the clusters of $\Part$, and edges corresponding to $k$-tuples of clusters which support many edges of $G$ and few edges of $Z$. However, we also add a third condition, which we will use for Lemma~\ref{redgraphmindeg} to show that any edge of $R$ can be extended to a copy of $\Phi$ in $R$ whose corresponding vertex classes support many copies of $\Phi$ in $G$.

\begin{defn}[Reduced $k$-graph, $\Phi$-dense, $Z$-sparse] \label{redgraphdef}
Under Setup~\ref{redsetup}, the \emph{reduced $k$-graph of $G$ and $Z$} (with parameters $c$ and $\nu$)
is the $k$-graph $R$ on vertex set $[m]$ in which vertex $i$ corresponds to the cluster $X_i$, and where $e \in \binom{[m]}{k}$ is an edge of $R$ if 
\begin{enumerate}[(i)]
\item $|G[\bigcup_{i \in e} X_i]| \geq c n_1^k$,
\item $|Z[\bigcup_{i \in e} X_i]| \leq \nu n_1^k$, and
\item for any $e' \in \binom{e}{k-1}$ there are at most $\nu^2 mn_1^k$ edges of $Z$ which intersect all of the clusters $X_i$ with $i \in e'$.
\end{enumerate}
Furthermore, for any $i, j \in [m]$ and $S \in \binom{[m]}{k-1}$, we say that the triple $(i, S, j)$ is \emph{$\Phi$-dense} if there are at least $c^2n_1^{k+1}$ copies of $\Phi$ in $G$ which have an end vertex in each of $X_i$ and $X_j$ and a central vertex in $X_\ell$ for each $\ell \in S$, and we say that $(i, S, j)$ is $Z$-sparse if each of $Z_{S \cup \{i\}}$ and $Z_{S \cup \{j\}}$ contains at most $\nu n_1^k$ edges.
\end{defn}
 
Note that under Setup~\ref{redsetup}, the partition $\Qart$ of $X$ naturally induces a partition of $[m] = V(R)$ into $r$ parts of equal size; we denote this partition by $\Qart_R$. So $i$ and $j$ are in the same part of $\Qart_R$ if and only if the clusters $X_i$ and $X_j$ are subsets of the same part of $\Qart$.
The next lemma shows that within any $\Phi$-dense and $Z$-sparse triple we can obtain a $k$-complex which is $D$-universal on $\Phi$, to which we can gainfully apply Proposition~\ref{varykcomp}.

\begin{lemma} \label{getphirobuni}
Adopt Setup~\ref{redsetup}, and suppose that sets $A, B \in \binom{[m]}{k}$ satisfy $|A \cap B| = k-1$. Let $i$ and $j$ be the elements of $A \sm B$ and $B \sm A$ respectively and suppose that the triple $(i, A \cap B, j)$ is $\Phi$-dense and $Z$-sparse. Suppose also that we have subsets $Y_\ell \subseteq X_\ell$ with $|Y_\ell| \geq \eta n_1$ for each $\ell \in A \cup B$. Then there exist subsets $W_\ell \subseteq Y_\ell$ with $|W_\ell| \geq (1-\mu)|Y_\ell|$ for each $\ell \in A \cup B$ and a $(k+1)$-partite $k$-complex $J$ whose vertex classes are $W_\ell$ for $\ell \in A \cup B$ such that $J_= \subseteq G \sm Z$ and $J$ is $D$-universal on $\Phi$ (where we here consider $A$ and $B$ to be the edges of $\Phi$, so $i$ and $j$ are the ends of $\Phi$).
\end{lemma}

\begin{proof}
Introduce new constants $d_a, \nu'$ and $\nu''$ with $d^* \ll d_a \ll 1/a$ and $\nu \ll \nu' \ll \nu'' \ll \mu$. 
Recall that for each $\Qart_R$-partite set $S \in \binom{[m]}{k}$, $\hat{\Part}$ partitions the $n_1^k$ edges of $K(X)_S$ into at most $a^k$ cells. We call such a cell $C$ a \emph{good} cell if it satisfies 
\begin{enumerate}[(a)]
\item $|C| \geq c^2 n_1^k/5a^k$, 
\item $|C \cap G| \geq c^2|C|/5$, and
\item $|C \cap Z| \leq \nu^{1/2} |C|$;
\end{enumerate}
otherwise, $C$ is a \emph{bad} cell. Consider the copies of $\Phi$ in $G$ whose edges $e$ and $f$ have indices $i(e) = A$ and $i(f) = B$. Since $(i, A \cap B, j)$ is $\Phi$-dense, there are at least $c^2 n_1^{k+1}$ such copies of $\Phi$ in $G$. We will show that at least one of these copies of $\Phi$ must have the property that both of its edges are contained in good cells of $\hat{\Part}$. For this, first note that since $\hat{\Part}$ partitions $K(X)_A$ into at most $a^k$ cells, at most $c^2 n_1^k/5$ edges of $K(X)_A$ lie in cells $C$ which fail (a). Likewise, since $(i, A \cap B, j)$ is $Z$-sparse we have $|Z_A| \leq \nu n_1^k$, and so at most $\nu^{1/2} n_1^k$ edges of $K(X)_{A}$ lie in cells $C$ which fail (c). Finally, the number of edges of $G_A$ which lie in cells $C$ which fail (b) is
$$ \sum_{C} |C \cap G| <  \sum_C c^2|C|/5 \leq c^2n_1^k/5, $$
where the sum is taken over all cells $C$ of index $A$ which fail (b). We deduce that at most $(c^2/5 + \nu^{1/2} + c^2/5)n_1^k \cdot n_1 < c^2n_1^{k+1}/2$ of the copies of $\Phi$ we counted have the edge of index $A$ in a cell which fails (a), (b) or (c). The same argument shows that fewer than $c^2n_1^{k+1}/2$ of the copies of $\Phi$ we counted have the edge of index $B$ in a cell which fails (a), (b) or (c). 

We may therefore fix a copy of $\Phi$ in $G$, whose edges $e$ and $f$ have indices $i(e) = A$ and $i(f) = B$ respectively, such that $\Cell(e)$ and $\Cell(f)$ each satisfy (a), (b) and (c). Recall that $G[\hat{\Part}](e)$ is defined to be the $k$-partite $k$-complex with vertex classes $X_i$ for $i \in A$ and whose edge set is 
$$(G_A \cap \Cell(e)) \cup \bigcup_{e' \subsetneq e} \Cell(e'),$$ 
and that $G[\hat{\Part}](f)$ is defined similarly. We define a $(k+1)$-partite $k$-complex $J^1$ with vertex classes $X_i$ for $i \in A \cup B$ to have edge set
$$J^1 := G[\hat{\Part}](e) \cup G[\hat{\Part}](f).$$
So the `top level' of $J^1$ consists of all edges of $G$ in the same cell as either $e$ or $f$, whilst the lower levels of $J^1$ are comprised of the cells of $\Part$ which lie `below' these cells. The crucial observation is that since $e$ and $f$ are the edges of a copy of $\Phi$ in $G$, for any $e' \subseteq e$ and $f' \subseteq f$ with $i(e') = i(f')$ we have $e' = f'$, and so $J^1$ includes only one cell of this index. That is, $J^1[X_A] = G[\hat{\Part}](e)$, and $J^1[X_B] = G[\hat{\Part}](f)$.
Since $G[\hat{\Part}]$ is $\eps$-regular, $G[\hat{\Part}](e)$ and $G[\hat{\Part}](f)$ are $\eps$-regular, and so it follows from the previous observation that $J^1$ is $\eps$-regular also. Furthermore, we have 
$$d_{A}(J^1) = \frac{|J^1_A|}{|(J_A^1)^*|} = \frac{|G \cap \Cell(e)|}{|\Cell(e)|} \geq \frac{c^2}{5},$$
and similarly $d_B(J^1) \geq c^2/5$. Also
$$d(J^1_A) = \frac{|J^1_A|}{|K(X)_A|} = \frac{|G \cap \Cell(e)|}{|\Cell(e)|} \cdot \frac{|\Cell(e)|}{n_1^k} \geq \frac{c^2}{5} \cdot \frac{c^2}{5a^k} > 2d_a,$$
and similarly $d(J^1_B) > 2d_a$. 
Finally, observe that 
$$|Z \cap J^1_A| \leq |Z \cap \Cell(e)| \leq \nu^{1/2}|\Cell(e)| \leq \nu^{1/2}\frac{|G \cap \Cell(e)|}{c^2/5} \leq \nu' |J^1_A|,$$
and similarly $|Z \cap J^1_B| \leq \nu' |J^1_B|$.

Let $Y := \bigcup_{\ell \in A \cup B} Y_i$, and define $J^2 := J^1[Y]$. So $J^2$ is a $(k+1)$-partite $k$-complex with vertex classes $Y_\ell$ for $\ell \in A \cup B$. By Lemma~\ref{regularrestriction} applied to $J^2[Y_A]$ and $J^2[Y_B]$ in turn, we find that $J^2$ is $\sqrt{\eps}$-regular, that $d_A(J^2), d_B(J^2) \geq c^2/10$, and that $d(J^2_A) \geq d(J^1_A)/2 \geq d_a$ and $d(J^2_B) \geq d(J^1_B)/2 \geq d_a$. In particular, the fact that $d(J^2_A) \geq d(J^1_A)/2$, together with our assumption that $|Y_\ell| \geq \eta |X_\ell|$ for each $\ell \in A$, implies that $|J^2_A| \geq \eta^k |J^1_A|/2$. So
$$|Z \cap J^2_A| \leq |Z \cap J^1_A| \leq \nu' |J^1_A| \leq \frac{2\nu' |J^2_A|}{\eta^k} \leq \nu'' |J^2_A|,$$
and similarly $|Z \cap J^2_B| \leq \nu'' |J^2_B|$. So we may apply Theorem~\ref{robust-universal} with $J^2$, $\Phi$ and the sets $Y_\ell$ in place of $J$, $R$ and the sets $V_\ell$ respectively, and with $\eta n_1, 1/\eta, \nu''$ and $c^2/10$ in place of $n, C, \nu$ and $d$ respectively. This yields subsets $W_\ell \subseteq Y_\ell$ with $|W_\ell| \geq (1-2(\nu'')^{1/3})|Y_\ell| \geq (1-\mu)|Y_\ell|$ for each $\ell \in A \cup B$ such that, writing $J := J^2[\bigcup_{\ell \in A \cup B} W_\ell] \sm Z$, we have that $J$ is $\eta$-robustly $D$-universal on $\Phi$ (so in particular $J$ is $D$-universal on $\Phi$), and that $J_= \subseteq G \sm Z$. 
\end{proof} 

Note that the application of Theorem~\ref{robust-universal} at the end of the proof also yields the facts that $d(J_A) > d^*$ and $|J_A(v)| > d^*|J_A|/|W_j|$ for any $j \in A$ and $v \in W_j$. We do not need these facts when applying Lemma~\ref{getphirobuni}, but we do need the analogous results when applying the next lemma, whose proof is similar to but simpler than that of Lemma~\ref{getphirobuni}, so we omit it (a comparable result was also proved for a slightly different definition of reduced $k$-graph in~\cite[Section 5.1.5]{KKMO}, by a similar argument).
 
\begin{lemma} \label{getrobuni} 
Adopt Setup~\ref{redsetup}, let $R$ be the reduced $k$-graph of $G$ and $Z$, and let $A$ be an edge of $R$. 
Then 
for any subsets $Y_i \subseteq X_i$ with $|Y_i| \geq \eta |X_i|$ for each $i \in A$, there exist subsets $W_i \subseteq Y_i$ with $|W_i| \geq (1-\mu)|Y_i|$ for each $i \in A$ and a $k$-partite $k$-complex $J$ with vertex classes $W_i$ for $i \in A$ such that $J$ is $\eta$-robustly $D$-universal, $J_= \subseteq G \sm Z$, $d(J_A) \geq d^*$ and $|J_A(v)| > d^* |J_A|/|W_j|$ for every $j \in A$ and $v \in W_j$.
\end{lemma}

Our final lemma shows if all $\Qart$-partite $(k-1)$-tuples have large degree in $G \cup Z$, then this degree condition is `almost' inherited by the reduced $k$-graph $R$, in that almost all $(k-1)$-tuples of $R$ satisfy a comparable condition. Furthermore, we also find that any edge of $R$ can be extended to many $\Phi$-dense and $Z$-sparse triples. To prove this latter result we make use of the unusual condition (iii) in Definition~\ref{redgraphdef}; this is the purpose of that condition. 

\begin{lemma} \label{redgraphmindeg}
Adopt Setup~\ref{redsetup}, and suppose that every $\Qart$-partite $(k-1)$-tuple $e$ of vertices of $G$ has $\deg_{G \cup Z}(e) \geq \gamma n$, and also that $|Z| \leq \xi n^k$. Then 
\begin{enumerate}[(i)]
\item there are at most $\theta m^{k-1}$ many $(k-1)$-tuples $S' \in \binom{[m]}{k-1}$ with $\deg_R(S') < (\gamma - \theta) m$. 
\item Furthermore, for any edge $S \in R$ and any $i \in S$ there are at least $(\gamma - \theta) m$ choices for $j \in [m] \sm S$ such that the triple $(i, S \sm \{i\}, j)$ is $\Phi$-dense and $Z$-sparse.
\end{enumerate}
\end{lemma} 

\begin{proof}
Let $\mc{S'}$ consist of all sets $S' \in \binom{[m]}{k-1}$ such that 
\begin{enumerate}[(a)]
\item $S'$ is $\Qart_R$-partite,
\item at most $\nu^{2} mn_1^{k}$ edges of $Z$ intersect all of the clusters $X_\ell$ with $\ell \in S'$, and 
\item for any $S'' \in \binom{S'}{k-2}$ at most $\nu^3 m^2n_1^{k}$ edges of $Z$ intersect all of the clusters $X_\ell$ with $\ell \in S''$.
\end{enumerate}
We will show that every $S' \in\mc{S'}$ has $\deg_R(S') \geq (\gamma - \theta)m$. To see this, fix some $S' \in \mc{S'}$, and let 
$$\mc{S} := \{S' \cup \{i\} : i \in [m] \sm S'\}.$$ 
Since $S'$ is $\Qart_R$-partite, any $(k-1)$-tuple $e'$ which consists of one vertex of $X_\ell$ for each $\ell \in S'$ has $\deg_{G \cup Z}(e') \geq \gamma n$ by assumption. Since $G \cup Z$ is $\Qart$-partite, any edge $e \in G \cup Z$ with $e' \subseteq e$ must have $e \in (G \cup Z)_S$ for some $S \in \mc{S}$, and so we conclude that there are at least $n_1^{k-1} \gamma n = \gamma n_1^k m$ edges $e \in G \cup Z$ whose index $i(e)$ is a member of $\mc{S}$. By (b), at most $\nu^2 mn_1^k$ of these edges lie in $Z$, and furthermore
at most $c n_1^k m$ of these edges lie in $G_S$ for some $S \in \mc{S}$ with $|G_S| < c n_1^k$. This leaves at least $(\gamma - c - \nu^2)n_1^k m$ edges which lie in $G_S$ for some $S \in \mc{S}$ with $|G_S| \geq c n_1^k$. Since $|G_S| \leq n_1^k$ for any $S \in \mc{S}$, we conclude that there are at least $(\gamma - c - \nu^2)m$ sets $S \in \mc{S}$ such that $|G_S| \geq cn_1^k$. Now observe that there can be at most $\nu m$ sets $S \in \mc{S}$ such that $|Z_S| > \nu n_1^k$. Indeed, if there were more, then taking the union of these $Z_S$ we would obtain more than $\nu^2 mn_1^k$ edges of $Z$ which meet $X_\ell$ for each $\ell \in S'$, contradicting (b). Similarly, there can be at most $2k\nu m$ sets $S \in \mc{S}$ for which some subset $T' \in \binom{S}{k-1}$ has the property that least $\nu^2 n_1^{k}m$ edges of $Z$ meet $X_\ell$ for every $\ell \in T'$. Indeed, if there were more, then some $S'' \in \binom{S'}{k-2}$ would be a subset of at least $2\nu m$ of the subsets $T'$, implying that more than $\nu^3 m^2n_1^{k}$ edges of $Z$ meet $X_\ell$ for every $\ell \in S''$, contradicting (c). We conclude that there are at least $(\gamma - c - \nu^2 - \nu - 2k\nu)m \geq (\gamma - \theta)m$ sets $S \in \mc{S}$ such that $|G_S| \geq cn_1^k$, $|Z_S| \leq \nu n_1^k$, and no subset $T' \in \binom{S}{k-1}$ has the property that at least $\nu^2 n_1^{k}m$ edges of $Z$ meet $X_\ell$ for every $\ell \in T'$; any $S$ with these three properties is an edge of $R$. So we do indeed have
$\deg_R(S') \geq (\gamma - \theta)m$.

It remains to prove that there are at most $\theta m^{k-1}$ sets $S' \in \binom{[m]}{k-1}$ such that $S' \notin \mc{S}'$, that is, which fail either (a), (b) or (c). For this, first note that at most $m^{k-1}/r$ sets $S' \in \binom{[m]}{k-1}$ are not $\Qart_R$-partite. 
Writing $N$ for the number of sets $S' \in \binom{[m]}{k-1}$ such that more than $\nu^{2} mn_1^{k}$ edges of $Z$ meet $X_\ell$ for each $\ell \in S'$, the fact that $|Z| \leq \xi n^k$ implies that $N \nu^2 mn_1^k \leq k\xi n^k$, so $N \leq k \xi m^{k-1}/\nu^2$. Finally, write $N'$ for the number of sets $S' \in \binom{[m]}{k-1}$ such that some subset $S'' \in \binom{S'}{k-2}$ has the property that there are more than $\nu^3 m^2n_1^{k}$ edges of $Z$ which meet $X_\ell$ for every $\ell \in S''$. The number of sets $S''$ with this property is then at least $N'/m$, so we obtain $(N'/m)\nu^3 m^2n_1^{k} \leq k^2 \xi n^k$, that is, $N' \leq k^2 \xi m^{k-1}/\nu^3$. We conclude that, as claimed, the number of sets $S' \in \binom{[m]}{k-1}$ such that $S' \notin \mc{S}'$ is at most
$$m^{k-1}/r + N + N' \leq m^{k-1}/r + k \xi m^{k-1}/\nu^2 + k^2 \xi m^{k-1}/\nu^3 \leq \theta m^{k-1}.$$

For the `furthermore' part, fix any $S \in R$ and $i \in S$, and write $S' := S \sm \{i\}$. Since $S \in R$ we know that there are at most $\nu^2 mn_1^k$ edges of $Z$ which meet $X_\ell$ for every $\ell \in S'$. So at most $\nu n_1^{k-1}$ edges $e' \in K(X)_{S'}$ have $\deg_Z(e') \geq \nu n$. Now, for any edge $e \in G_S$ we have a $(k-1)$-tuple $e' := e \sm X_i \in K(X)_{S'}$; since $G$ is $\Qart$-partite our minimum codegree assumption implies that $\deg_{G \cup Z}(e') \geq \gamma n$. Each $(k-1)$-tuple $e'$ is formed in this way from at most $n_1$ edges of $G_S$, so we conclude that there are at least $|G_S| - \nu n_1^k$ edges $e \in G_S$ for which $\deg_G(e') \geq (\gamma - \nu) n$. Since at most $kn_1 \leq \nu n$ vertices lie the sets $X_\ell$ for $\ell \in S$, there are at least $(|G_S| - \nu n_1^k) (\gamma - 2\nu)n$ copies of $\Phi$ in $G$ whose edges have indices $S$ and $S' \cup \{j\}$ for some $j \in [m] \sm S$. Since for any $j \notin S$ at most $|G_S| n_1$ of these copies have a vertex in $X_j$, we conclude that the triple $(i, S', j)$ is $\Phi$-dense for at least 
$$\frac{(|G_S| - \nu n_1^k) (\gamma - 2\nu) n - c^2n_1^{k+1}m}{|G_S| n_1} \geq \frac{|G_S|(\gamma - 2\nu)m - (\gamma \nu + c^2)n_1^km}{|G_S|} \geq \left(\gamma - \frac{\theta}{2}\right)m$$ 
choices of $j \in [m] \sm S$, where we used the fact that $|G_S| \geq c n_1^k$ since $S \in R$. 
So to complete the proof it suffices to show that the triple $(i, S', j)$ is $Z$-sparse for all but at most $\theta m/2$ choices of $j \in [m] \sm S$. For this, recall that $|Z_S| \leq \nu n_1^k$ since $S$ is an edge of $R$, so if $(i, S', j)$ is not $Z$-sparse then $|Z_{S' \cup \{j\}}| > \nu n_1^k$. Furthermore, since $S \in R$ there are at most $\nu^2 mn_1^k$ edges of $Z$ which meet $X_\ell$ for every $\ell \in S'$. So, writing $N''$ for the number of choices of $j$ for which the triple is not $Z$-sparse, we have $N'' \nu n_1^k \leq \nu^2 mn_1^k$, and so $N'' \leq \nu m \leq \theta m/2$, as required.
\end{proof}

\subsection{Degree sequences and irreducibility} \label{sec:irreduc}
Let $J$ be a $k$-complex. Then the \emph{degree sequence} of $J$ is the sequence $\delta(J) = (\delta_0(J), \delta_1(J), \dots, \delta_{k-1}(J))$, where for any $i \in [k]$ we define$$\delta_{i-1}(J) := \min_{e \in J_{i-1}} \deg_{J_i}(e).$$
So every edge $e \in J_{i-1}$ is a subset of at least $\delta_{i-1}(J)$ edges of $J_i$, or in other words there are at least $\delta_{i-1}(J)$ vertices $v \in V(J)$ such that $e \cup \{v\} \in J$. Inequalities between degree sequences should always be interpreted pointwise. Note also that we only defined the minimum codegree $\delta(H)$ for $k$-graphs $H$, and the degree sequence $\delta(J)$ for $k$-complexes $J$, so there should be no confusion.

The following lemma states that if almost all $(k-1)$-tuples of vertices of a $k$-graph $H$ have high degree, then we can find a $k$-complex $J$ which covers almost all of the vertices of $H$, such that $J$ has a useful degree sequence and the `top level' $J_=$ of $J$ is a subgraph of $H$. The $k$-partite form of this lemma was given by Keevash, Knox and Mycroft~\cite[Lemma~7.3]{KKM} with a straightforward proof. The proof of the form given below is identical except for the simplification of not having to handle multiple vertex classes, so we omit it (this form is also implicit in~\cite{KM}).

\begin{lemma} \label{obtainseq}
Suppose that $1/m \ll \theta \ll \beta, 1/k$, and let $H$ be a $k$-graph on a vertex set $V$ of size $m$ in which at most $\theta m^{k-1}$ sets $S \in \binom{V}{k-1}$ have $\deg_{H}(S) \leq D$. Then there exists a $k$-complex $J$ with $V(J) \subseteq V$ such that $J_= \subseteq H$, $m' := |V(J)| \geq (1-\sqrt{\theta})m$ and $\delta(J) \geq (m', (1-\beta)m', \dots, (1-\beta)m', D-\beta m')$.
\end{lemma}

Now let $H$ be a $k$-graph on $n$ vertices which admits a perfect matching $M$. Using the terminology of Keevash and Mycroft~\cite{KM} we say that $H$ is \emph{$(C, L)$-irreducible} on $M$ if for any $u, v \in V(H)$ with $u \neq v$ there exist multisets $S$ and $T$ of edges of $H$ and $M$ respectively, so that $|S|, |T| \leq L$ and, counting with multiplicity, for some $c \leq C$ the vertex $u$ appears in precisely $c$ more edges of $S$ than of $T$, the vertex $v$ appears in precisely $c$ more edges of $T$ than of $S$, and every other vertex of $H$ appears equally often in $S$ as in $T$. The next lemma, a special case of a result of Keevash and Mycroft~\cite[Lemma~5.6]{KM}, gives a sufficient degree sequence condition on a $k$-complex $J$ for $J_=$ to be irreducible on a perfect matching in $J_=$.

\begin{lemma} \label{seqirreduc}
Suppose that $1/m \ll 1/C,  1/L \ll \alpha, 1/k$, and let $J$ be a $k$-complex on $m$ vertices with $\delta(J) \geq (m, (k-1)m/k + \alpha m, (k-2)/k + \alpha m,  \dots, m/k + \alpha m)$ such that $J_=$ admits a perfect matching $M$. Then $J_=$ is $(C, L)$-irreducible on $M$.
\end{lemma}

\section{Ingredients of the proof}\label{sec:prelims}

\subsection{Partitioning clusters into lopsided groups.}\label{sec:partitions}
 
As described in Section~\ref{sec:outline}, we will find an almost-perfect packing of the reduced $k$-graph $\R$ with a particular $k$-partite $k$-graph $\Akpq$. The $k$-graph $\Akpq$ which we use is defined as follows. The vertex set $V(\Akpq)$ is the union of disjoint sets $A_1, \dots, A_{q-p}$ and $B$, where $|A_j| = k-1$ for each $j \in [q-p]$ and $|B| = p(k-1)$. Then any $k$-tuple of the form $\{x\} \cup A_j$ with $j \in [q-p]$ and $x \in B$ is an edge of $\Akpq$ (see Figure~\ref{fig:akpq} for an illustration). In particular we have $|V(\Akpq)|= q(k-1)$. 

\begin{figure}[t] 
\centering
\psfrag{1}{$A_1$}
\psfrag{2}{$A_2$} 
\psfrag{3}{$A_3$}
\psfrag{4}{$B$}
\includegraphics[width=6cm]{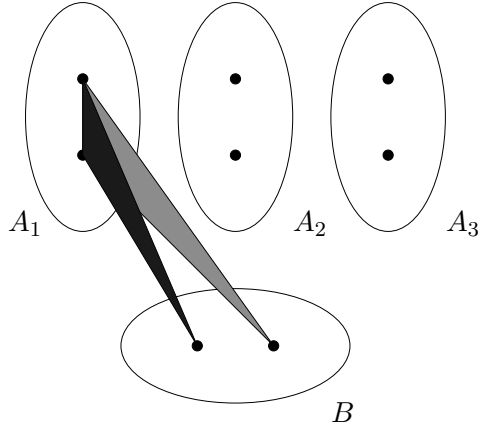} 
\caption{The $k$-graph $\Akpq$ in the case $k=3, p=1, q=4$; the edges between $B$ and $A_1$ are shown, and there are similar edges between $B$ and $A_2$ and between $B$ and $A_3$, giving six edges in total.}
\label{fig:akpq}
\end{figure} 

The next lemma shows that if $G$ is a $k$-graph on $m$ vertices in which almost all sets of $k-1$ vertices have degree slightly greater than $pm/q$, then $G$ contains an almost-perfect $\Akpq$-packing, that is, one which covers almost all vertices of $G$. We will apply this result with the reduced $k$-graph $\R$ in place of $G$ and with $p$ and $q$ chosen so that $p/q \approx \sigma(K)$. The degree condition needed will then follow from Lemma~\ref{redgraphmindeg} and our assumption in Lemma~\ref{main1simple} that $\delta(H) \geq \sigma(K)n + \alpha n$.
This lemma was previously proved for $k=3, p=1, q=4$ by K\"uhn and Osthus~\cite{KO} and then for $p=1, q=2k-2$ by Keevash, K\"uhn, Mycroft and Osthus~\cite{KKMO}; the proof given here is essentially identical, but is included for completeness.

\begin{lemma} \label{akpqpacking}
Suppose that $1/m \ll \theta \ll \psi\ll 1/q, 1/p, 1/k$ and that $G$ is a $k$-graph on vertex set $[m]$ such
that $\deg_G(S) > (\frac{p}{q} + \theta)m$ for all but at most
$\theta m^{k-1}$ sets $S \in \binom{[m]}{k-1}$. Then~$G$ admits an $\mc{A}^k_{p,q}$-packing~$\F$ such that $|V(\F)| \geq (1-\psi)m$ and $G[V(\F)]$ is connected (where $V(\F)$ denotes the set of vertices covered by $\F$).
\end{lemma}

\begin{proof} 
Let $\F$ be a maximal $\Akpq$-packing in 
$G$, and let $X := V(G) \sm V(\F)$. We will show that $|X| \leq \psi m/2$; to do this, we suppose for a contradiction that $|X| > \psi m/2$. For any $(k-1)$-tuple $S$ of 
vertices of $G$, we write $N(S)$ to denote the set $\{v \in V(G) : S \cup \{v\} \in G\}$ of neighbours of $S$, so $|N(S)| = \deg_G(S)$. We also write $\deg(S)$, $N_X(S)$ and $\deg_X(S)$ for $\deg_G(S)$, $N(S) \cap X$ and $|N_X(S)|$ respectively. Note that since $\theta \ll \psi$ we can greedily form a collection of at least $2\theta m$ disjoint $(k-1)$-tuples $S \in 
\binom{X}{k-1}$ which each satisfy $\deg(S) \geq pm/q + \theta m$.  

Suppose first that for some $r \geq \theta m$ there exist disjoint sets $S_1, \dots, S_r \in \binom{X}{k-1}$ such that $\deg_X(S_i) \geq \theta m/2$ for any $i \in [r]$. 
In this case, we count the pairs $(i, B)$ such that $i \in [r]$ and $B \subseteq N_X(S_i)$ has size $p(k-1)$. By our choice of the sets $S_1, \dots, S_r$, the number of such pairs is at 
least 
$$r \binom{\theta m/2}{p(k-1)} \geq \theta m \binom{\theta m/2}{p(k-1)} \geq (q-p) \binom{m}{p(k-1)} \geq (q-p) \binom{|X|}{p(k-1)}.$$ 
So there must be some set $B \in \binom{X}{p(k-1)}$ which lies in at least $q-p$ such pairs; the corresponding $q-p$ sets $S_i$ together with this set $B$ form a copy of $\Akpq$ contained in $G[X]$, contradicting the maximality of $\F$.

Since there are at least $2\theta m$ disjoint $(k-1)$-tuples $S \in \binom{X}{k-1}$ which each satisfy $\deg(S) \geq pm/q + \theta m$, it follows that we may choose a family of $r \geq \theta m$ subsets $S_1, \dots, S_r \in \binom{X}{k-1}$ such that each~$S_i$ satisfies $\deg(S_i) \geq pm/q + \theta m$ and $\deg_X(S_i) < \theta m/2$. Having fixed this family, we say that a copy $\A \in \F$ is \emph{good} for $S_j$ if $|V(\A) \cap N(S_j)| > p(k-1)$. Note that each set $S_j$ has at least $pm/q + \theta m/2$ neighbours in $V(\F)$, and at most 
$$|\F| p(k-1) \leq \frac{p(k-1) m}{|V(\Akpq)|} = pm/q$$ 
of these neighbours lie in members of $\F$ which are not good for $S_j$. So the number of copies $\A \in \F$ which are good for $S_j$ is at least $\theta m/2|V(\Akpq)| = \theta m/2q(k-1)$.

We now count the number of pairs $(j, \T)$ where $j \in [r]$ and $\T \subseteq \F$ consists of $p(k-1)$ copies $\A \in \F$, each of which is good for $S_j$. By the above calculation, this number is at least 
$$ r \binom{\theta m/2q(k-1)}{p(k-1)} \geq \theta m \binom{\theta m/2q(k-1)}{p(k-1)} \geq \sqrt{m}\binom{m}{p(k-1)} \geq \sqrt{m}\binom{|\F|}{p(k-1)}.$$
We can therefore choose a collection $\T$ of $p(k-1)$ copies $\A \in \F$ and a subset $R \subseteq [r]$ of size $|R| \geq \sqrt{m}$ such that $\A$ is good for $S_j$ for any $j \in R$ and $\A \in \T$. This means that for each $j \in R$ and each $\A \in \T$ we may choose a subset $L^\A_j \subseteq N(S_j) \cap V(\A)$ of size $p(k-1)+1$. Having done so, the fact that $|R| \geq \sqrt{m}$ implies that we may choose a subset $R' \subseteq R$ of size $(p(k-1)+1)(q-p)$ so that for any fixed $\A \in \T$, $L^\A_j$ is the same set for every $j \in R'$. We write $L^\A$ for this common value of $L^\A_j$. 
 
Arbitrarily partition $R'$ into $p(k-1)+1$ sets $R'_1, \dots, R'_{p(k-1)+1}$ of size $(q-p)$, and label the vertices of each $L^\A$ as $\{v^\A_1, v^\A_2, \dots, v^\A_{p(k-1)+1}\}$. Then for each $s \in [p(k-1)+1]$, the sets $S_j$ for $j \in R'_s$ and the set $\{v^\A_s : \A \in \T\}$ together form a copy of $\Akpq$. This produces $p(k-1)+1$ vertex-disjoint copies of $\Akpq$ which are contained in $X \cup V(\T)$, so we may enlarge $\F$ by replacing the members of $\T$ with these copies, giving another contradiction. 

This proves that $|X| \leq \psi m/2$, so $\F$ covers at least $(1-\psi/2)m$ vertices of $G$. Note that $G[\A]$ is connected for any $\A \in \F$. Let $\F' \subseteq \F$ be of maximum size such that $G[V(\F')]$ is connected, and suppose for a contradiction that $|V(\F')| < (1-\psi)m$, so $|V(\F) \sm V(\F')| > \psi m/2$. We first observe that some vertex of $V(\F)$ must lie in some $(k-1)$-tuple $S \in \binom{V(\F)}{k-1}$ with $\deg_G(S) \geq pm/q + \theta m$, and so has at least $pm/2q$ neighbours in $V(\F)$, so by maximality of $\F'$ we have $|V(\F')| \geq pm/2q$. Therefore, the number of sets $S \in \binom{V(\F)}{k-1}$ which contain a vertex $x \in V(\F) \sm V(\F')$ and a vertex $y \in V(\F')$ is at least 
$$ \frac{1}{(k-1)!} \cdot \frac{\psi m}{2} \cdot \frac{pm}{2q} \cdot ((1-\psi/2) m)^{k-3} > \theta m^{k-1}.$$
It follows that some such $S$ has degree at least $pm/q > \psi m/2$, and so can be extended to an edge of $G[V(\F)]$. But then the member of $\F$ containing $x$ can be added to $\F'$ to give a larger subpacking $\F'' \subseteq \F$ such that $G[V(\F'')]$ is connected, a contradiction. This proves that $|V(\F')| \geq (1-\psi)m$, so $\F'$ is the desired $\Akpq$-packing. 
\end{proof}

Having obtained an almost-perfect $\Akpq$-packing in the reduced $k$-graph $\R$, we will proceed to partition the clusters corresponding to copies of $\Akpq$, and then to rearrange the parts obtained into groups of $k$ subclusters which support regular and dense complexes. This partition is effected in the following way.

\begin{lemma} \label{akpqsplit}
Suppose that $pk \leq q$, and that for each vertex $u \in V(\Akpq)$ we have a set $V_u$ of $n$ vertices such that the sets $V_u$ are pairwise-disjoint. Let $V = \bigcup_{u \in V(\Akpq)} V_u$, and suppose also that $(q-p)p(k-1)$ divides $n$. Then we may partition $V$ into sets $X_j^i$ with $j \in [k]$ and $i \in [(q-p)p(k-1)]$ such that
\begin{enumerate}[(i)]
\item $|X_1^i| = \frac{p}{q}\sum_{j \in [k]} |X_j^i|$ for each $i$,
\item $n/(q-p) = |X_1^i| \leq |X_2^i| = |X_3^i| = \dots = |X_k^i|$ for each $i$,
\item for each $i$ and $j$ there exists $f(i,j) \in V(\Akpq)$ so that $X_j^i \subseteq V_{f(i,j)}$, and
\item for each fixed $i$ the set $\{f(i,j) : j \in [k]\}$ is an edge of $\Akpq$.
\end{enumerate}
\end{lemma}

\begin{figure}[t] 
\centering
\psfrag{A}{$A_1$} 
\psfrag{B}{$A_2$} 
\psfrag{C}{$A_3$} 
\psfrag{D}{$B$} 
\psfrag{E}{$X^{1, 1}_3$}
\psfrag{F}{$X^{2, 1}_3$}
\psfrag{G}{$X^{1, 1}_2$} 
\psfrag{H}{$X^{2, 1}_2$} 
\psfrag{I}{$X^{1, 2}_3$}
\psfrag{J}{$X^{2, 2}_3$}
\psfrag{K}{$X^{1, 2}_2$}
\psfrag{L}{$X^{2, 2}_2$} 
\psfrag{M}{$X^{1, 3}_3$} 
\psfrag{N}{$X^{2, 3}_3$}
\psfrag{O}{$X^{1, 3}_2$} 
\psfrag{P}{$X^{2, 3}_2$}
\psfrag{Q}{$X^{1, 1}_1$} 
\psfrag{R}{$X^{1, 2}_1$} 
\psfrag{S}{$X^{1, 3}_1$} 
\psfrag{T}{$X^{2, 1}_1$} 
\psfrag{U}{$X^{2, 2}_1$} 
\psfrag{V}{$X^{2, 3}_1$}
\includegraphics[width=12cm]{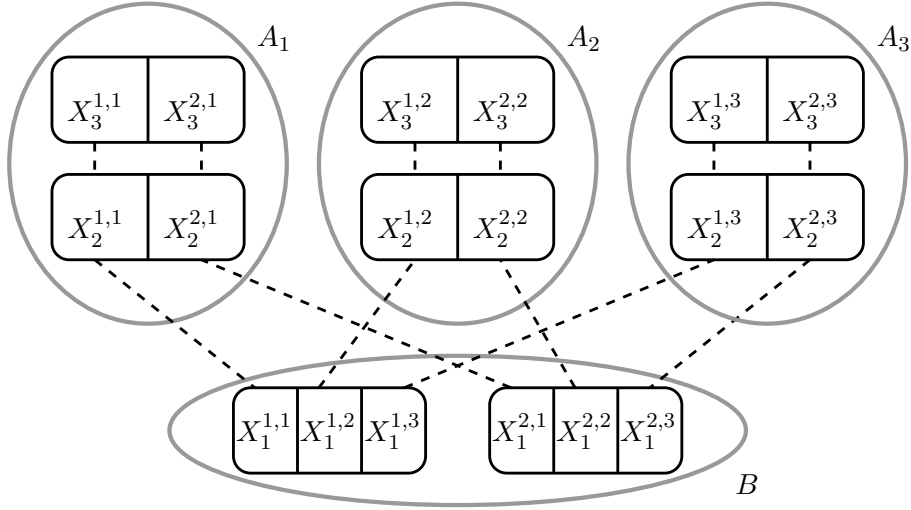}
\caption{An illustration of the division of clusters implemented in Lemma~\ref{akpqsplit} for the case $k=3, p=1, q=4$. The dashed lines join subclusters which form part of the same group.} 
\label{fig:split} 
\end{figure}

\begin{proof}
Let $A_1, \dots, A_{q-p}$ and $B$ be as in the definition of the $k$-graph $\Akpq$. So these sets are pairwise-disjoint and their union is $V(\Akpq)$; also, $|B| = p(k-1)$ and $|A_a| = k-1$ for each~$a \in [q-p]$. Arbitrarily order the vertices of each of these sets, and for $i \in [k-1]$ and $a \in [q-p]$ write $u(i, A_a)$ for the $i$th vertex of~$A_a$, and similarly for $j \in [p(k-1)]$ write $v(j, B)$ for the $j$th vertex of $B$. Next, for every $j \in [q-p]$ and every $u \in A_j$, partition the set $V_u$ into $p(k-1)$ parts $V_u^1, \dots, V_u^{p(k-1)}$ of equal size. Similarly, for each $v \in B$ partition $V_v$ into $q-p$ parts $V_v^1, \dots, V_v^{q-p}$ of equal size. Then for each $a \in [q-p]$ and $b \in [p(k-1)]$ define 
$$\mbox{$X^{a, b}_1 = V_{v(b, B)}^a$ and $X^{a, b}_j = V_{u(j-1, A_a)}^b$ for $2 \leq j \leq k$.}$$
Relabelling these sets (that is, replacing the superscript $(a,b)$ by an integer in $[(q-p)p(k-1)]$) gives the desired sets.

Property (iii) is immediate from the construction, and since any set of the form $A_a \cup \{v\}$ with $v \in B$ is an edge of $\Akpq$, (iv) is satisfied also. Finally, observe that for each $i$, $|X_1^i| = n/(q-p)$ and $|X_2^i| = \dots = |X_k^i| = n/p(k-1)$. So our assumption that $pk \leq q$ implies that $|X^i_1| \leq |X^i_2|$, proving (ii), and $$\frac{p}{q} \cdot \sum_{j \in [k]} |X_j^i| = \frac{p}{q}\left(\frac{n}{q-p} + (k-1) \cdot \frac{n}{p(k-1)}\right) = \frac{n}{q-p} = |X^i_1|,$$ 
so (i) holds also.
\end{proof}

In the proof of Lemma~\ref{main2simple} we will find a matching $M_\R$ in the reduced $k$-graph $\R$, rather than an $\Akpq$-packing. Observe for this that $\A^k_{1, k}$ contains a perfect matching, so it suffices to find an $\A^k_{1, k}$-packing in $\R$, which we can do by applying Lemma~\ref{akpqpacking} with $p=1$ and $q=k$. However, in the proof of Lemma~\ref{main2simple} we also require an additional assumption, namely that the restriction of $\R$ to any large submatching $M'_\R \subseteq M_\R$ of this matching is irreducible on $M'_\R$. The following corollary states that we can do this.

\begin{coro} \label{almostmatching}
Suppose that $1/m \ll \theta \ll \psi \ll 1/C, 1/L \ll \alpha, 1/k$, and that $G$ is a $k$-graph on vertex set $[m]$ such
that $\deg_G(S) > (1/k + \alpha)m$ for all but at most
$\theta m^{k-1}$ sets $S \in \binom{[m]}{k-1}$. Then~$G$ admits a matching $M$ with $|V(M)| \geq (1-\psi)m$ such that $G[V(M')]$ is $(C, L)$-irreducible on $M'$ for any $M' \subseteq M$ with $|M'| \geq (1-\alpha/2)|M|$.
\end{coro}

\begin{proof}
Introduce a new constant $\beta$ with $\theta \ll \beta \ll \psi$. By Lemma~\ref{obtainseq} there exists a $k$-complex $J$ with $V(J) \subseteq [m]$ such that $J_= \subseteq G$, $m_1 := |V(J)| \geq (1-\sqrt{\theta})m$, and $\delta(J) \geq (m_1, (1-\beta)m_1, \dots, (1-\beta)m_1, (1/k + \alpha-\beta) m_1)$. Then at least $(1-\beta)^{k-2}m_1^{k-1}/(k-1)! \geq (1-k\beta)\binom{m_1}{k-1}$ many $(k-1)$-tuples $S \in \binom{V(J)}{k}$ are edges of $J_{k-1}$, and so have $\deg_{J_=}(S) \geq (1/k+\alpha-\beta)m_1 \geq (1/k+k\beta)m_1$. So we can apply Lemma~\ref{akpqpacking} to $J_=$ with $p=1$ and $q=k$, and with $k\beta$ and $\psi/2$ in place of $\theta$ and $\psi$ respectively. Since $\A^k_{1, k}$ admits a perfect matching, this yields a matching $M$ in $J_= \subseteq G$ with $|V(M)| \geq (1-\psi/2)m_1 \geq (1-\psi)m$. Now fix any $M' \subseteq M$ of size $|M'| \geq (1-\alpha/2)|M|$, and define $m' := |V(M')|$, so $m_1 - m' \leq 2\alpha m_1/3$. It follows that $\delta(J[V(M')]) \geq (m', (1-\alpha)m_1, \dots, (1-\alpha)m_1,(1/k + \alpha/4) m_1)$, and so $J_=[V(M')]$ is $(C, L)$-irreducible on $M'$ by Lemma~\ref{seqirreduc} (with $\alpha/4$ in place of $\alpha$). Since $J_= \subseteq G$ it follows that $G[V(M')]$ is $(C, L)$-irreducible on $M'$.
\end{proof}

\subsection{Incorporating exceptional vertices}

We will need to be able to remove a small number of `bad' vertices of~$H$. Our strategy here will be to find a copy of~$K$ which contains the vertex to be removed, and to delete that copy of~$K$ from~$H$. This copy of~$K$ will ultimately form part of the perfect $K$-packing of~$H$ which we construct. The next lemma allows us to do this by demonstrating that any vertex of a $k$-graph $H$ with high codegree must lie in some copy of $K$ in $H$.

\begin{lemma} \label{incorporateexcep}
Suppose that $1/n \ll \alpha, 1/b$. Let $K$ be a $k$-partite $k$-graph on $b$ vertices, and let $H$ be a $k$-graph on $n$ vertices with $\delta(H) \geq \alpha n$. Then for any vertex $u \in V(H)$ there is a copy of $K$ in $H$ which contains $u$.
\end{lemma}

\begin{proof}
Partition the vertices of $H$ into parts $V_1$ and $V_2$ by assigning $u$ to $V_1$ and randomly assigning each other vertex of $H$ to $V_1$ with probability 1/2 and $V_2$ otherwise, where these assignments are independent for each vertex. Let $H' \subseteq H$ be the $k$-graph on vertex set $V(H)$ whose edge set is $$\{e \in H : |e \cap V_1| = 1 \textrm{ and } \{u\} \cup (e \cap V_2) \in H\}.$$ 
So an edge of $H$ is an edge of $H'$ if it has precisely $k-1$ vertices in $V_2$ and these $k-1$ vertices together with $u$ also form an edge of $H$. It suffices to show that for some outcome of our random selection the $k$-graph $H'$ has at least $2^{-k}\alpha^2 \binom{n}{k}$ edges. Indeed, by Theorem~\ref{turandensityzero}, $H'$ must then contain a copy of $\B(K)$ (recall from Definition~\ref{defBULo} that this is the complete $k$-partite $k$-graph with $k$ vertex classes each of size $b$). Together with $u$, this gives a subgraph of $H$ which contains as a subgraph a copy of $K$ containing $u$. 

Now, if we choose vertices $x_1, \dots, x_{k-1}$ in turn to form an edge $\{u, x_1, \dots, x_{k-1}\}$ of $H$, then we have $n-j$ choices for $x_j$ for $1 \leq j \leq k-2$ and at least $\delta(H) \geq \alpha n$ choices for $x_{k-1}$. Since this process will count each edge $(k-1)!$ times, we find that $u$ lies in at least $\alpha \binom{n}{k-1}$ edges of $H$. For any such edge $\{u, x_1, \dots, x_{k-1}\}$ there are at least $\delta(H) \geq \alpha n$ choices of $y$ such that $\{y, x_1, \dots, x_{k-1}\}$ is an edge $e \in H$. Each such edge may be formed by up to $k$ different choices of $x_1, \dots, x_{k-1}$, so we find that there are at least $\alpha^2 \binom{n}{k}$ edges $e \in H$ for which there is some $y \in e$ such that $\{u\} \cup e \sm \{y\}$ is an edge of $H$. For each such edge, the probability that $y$ is assigned to $V_1$ and all vertices of $e \sm \{y\}$ are assigned to $V_2$ is at least $2^{-k}$. So the expected number of edges $e$ with this form whose vertices are assigned in this way is at least $2^{-k} \alpha^2 \binom{n}{k}$, and every such edge is an edge of~$H'$. There must therefore be some outcome of our random partition of $V(H)$ for which $H'$ has at least this many edges, as required. 
\end{proof}

\subsection{Ensuring divisibility of subcluster sizes} \label{sec:delete}

As described in Section~\ref{sec:outline}, a key step in the proof of Lemma~\ref{main1simple} is to delete a $K$-packing in $H$ such that, following these deletions, the size of each subcluster is divisible by $bk \gcd(K)$ (recall that $b$ is the order of the $k$-graph $H$). This allows us to complete the proof by finding a perfect $K$-packing in each of our robustly universal $k$-partite $k$-graphs $G^i \sm Z^i$. In this section we prove Lemma~\ref{gcdbalance}, which states that we can indeed do this. We begin with the following lemma. 

\begin{lemma} \label{gcd1balance}
Let $G$ be a $t$-partite $k$-graph with vertex classes $X_1, \dots, X_t$. Fix any integer $d$, and let $d'$ be a factor of $d$ such that $d'$ divides $|X_j|$ for any $j \in [t]$. Also fix a $k$-graph $K$ on $b$ vertices, and suppose that $\Sa$ is a graph on vertex set~$[t]$ such that 
\begin{enumerate}[(i)]
\item for any connected component $C$ of $\Sa$, $\sum_{j \in V(C)} |X_j|$ is divisible by $d$, and
\item for any edge $uv \in \Sa$ there are at least $bdt^2$ vertex-disjoint copies $K'$ of $K$ in $G$ such that for each $j \in [t]$ we have
\begin{equation*}
|V(K') \cap X_j| \equiv 
\begin{cases}
-d' \mod d & \textrm{if } j = u\\
d' \mod d  & \textrm{if } j = v\\
0 \mod d & \textrm{otherwise}
\end{cases}
\end{equation*}
\end{enumerate}
Then $G$ contains an $K$-packing $M$ of size at most $dt^2$ so that $d$ divides $|X_j \sm V(M)|$ for any $j \in [t]$.
\end{lemma}  

\begin{proof}
We prove the lemma by repeatedly choosing an $K$-packing in $G$, deleting its vertices from $G$, and adding its members to $M$ (which is initially taken to be empty). This ensures that $M$ will indeed be an $K$-packing in $G$. After each deletion we continue to write $X_j$ for the vertices in $X_j$ which were not deleted, and $G$ for the $k$-graph which remains (that is, the restriction of $G$ to the undeleted vertices). We will also ensure that each deletion preserves the properties that $d$ divides $\sum_{j \in V(C)} |X_j|$ for any component $C$ of $\Sa$ and that $d'$ divides $|X_j|$ for any $j \in [t]$. 

The deletion step is as follows: suppose that there is some $u \in [t]$ such that $|X_u| \not \equiv 0 \mod d$, and let $x \in [d-1]$ satisfy $xd' \equiv |X_u| \mod d$ (this is possible since $d'$ is a factor of $d$ which divides $|X_u|$). Let $C$ be the component of $\Sa$ containing $u$; since $d$ divides $\sum_{j \in V(C)} |X_j|$ by (i) there must be some $v \in V(C)$ such that $v \neq u$ and $|X_v| \not \equiv 0 \mod d$. Also, since $C$ is a component of $\Sa$ we may choose a path $P$ from $u$ to $v$ in $\Sa$. Let $u = w_0, w_1, \dots, w_p = v$ be the vertices of $P$ (in order), so $p \leq t$.
Now, for each $\ell \in [p]$, $w_{\ell-1}w_\ell$ is an edge of $\Sa$, so by~(ii) we may choose $x$ copies of $K$ in $G$ such that the intersection of each copy of $K$ with the vertex class $X_j$ has size equal to $d'$ modulo $d$ if $j = w_{\ell-1}$, equal to $-d'$ modulo $d$ if $j= w_{\ell}$, and equal to $0$ modulo $d$ otherwise. We do this so that the chosen copies of $K$ are pairwise vertex-disjoint (we shall see shortly that we can simply choose copies of $K$ greedily to ensure this). 
Delete the vertices of each chosen copy of $K$ from $G$ and add these copies to $M$. The effect of these deletions is to reduce $|X_u|$ by $xd'$ modulo $d$, to increase $|X_v|$ by $xd'$ modulo $d$, and to leave the size of each other vertex class unchanged modulo $d$. So we now have $|X_u| \equiv 0 \mod d$, that is, the number of vertex classes $X_j$ with $|X_j| \equiv 0 \mod d$ has increased by at least one.

We repeat the deletion step until $|X_j| \equiv 0 \mod d$ for every $j \in [t]$; the previous observation shows that this must occur after at most $t$ steps. Since at each step we deleted $px < td$ copies of $K$, the $K$-packing $M$ obtained at termination has size less than $dt^2$, as required. The same argument shows that it is possible to choose copies of $K$ as claimed, since at any point the fewer than $dt^2$ previously-deleted copies of $K$ can intersect fewer than $bdt^2$ members of a family of pairwise vertex-disjoint copies of $K$. 
\end{proof}

If $\gcd(K) = 1$, then Lemma~\ref{gcd1balance} is in fact sufficient for our purposes (i.e. to delete a $K$-packing in $H$ so that the number of remaining vertices in each subcluster is divisible by $bk \gcd(K)$).  Indeed, in this case we first arbitrarily delete a small number of copies of $K$ so that $bk$ divides the total number of remaining vertices. We then choose $s$ large enough such that $\U_s(K)$ is defined (see Definition~\ref{defBULo}) and so that $k$ divides $s$, and apply the lemma with the adjacency graph $\mathrm{Adj}(\R')$, $\U_s(K)$, $bk$ and the subclusters $V^i_j$ in place of $\Sa, K, d$ and the sets $X_j$ respectively, and with $d' = 1$. The graph $\mathrm{Adj}(\R')$ is connected since $\R'$ is connected, and so has only one connected component, so condition (i) of Lemma~\ref{gcd1balance} holds by our initial deletions, and condition (ii) Lemma~\ref{gcd1balance} follows from Lemma~\ref{getrobuni} and the fact that $\U_s(K)$ has one vertex class of size $bs-1$ and one of size $bs+1$, whilst all other vertex classes have size $bs$. So we obtain a $\U_s(K)$-packing $M$ in $H$ whose deletion leaves all subclusters with size divisible by $bk$; since $\U_s(K)$ admits a perfect $K$-packing this gives a $K$-packing as required (this argument is given in more detail in the proof of Lemma~\ref{gcdbalance}).

However, if $\gcd(K) \geq 2$ then the situation is somewhat more complicated, and we in fact make two applications of Lemma~\ref{gcd1balance}; once with $\mathrm{Adj}(\R')$ in place of $\Sa$ as described above, and another with $\Sa$ being the graph $\Sa'$ described in the proof outline in Section~\ref{sec:outline}, whose edges indicate that the corresponding subclusters were taken from clusters which form ends of a $\Phi$-dense and $Z$-sparse triple. In the latter application, condition (ii) of Lemma~\ref{gcd1balance} follows as a consequence of Lemma~\ref{getphirobuni}. However, it is more problematic to ensure that condition (i) is satisfied, as $\Sa'$ may have multiple connected components. The key here is that, as outlined in Section~\ref{sec:outline}, $\Sa'$ must have fewer than $p$ components, where $p$ is the least prime factor of $\gcd(K)$. That is, every prime factor of $\gcd(K)$ is strictly greater than $r$, the number of components of~$\Sa'$. The next lemma shows that this fact allows us to choose edges of $\R'$ whose index vectors with respect to the partition of $V(\R')=V(\Sa')$ into components of $\Sa'$ sum to any chosen `target vector' $\vb$. In the proof of Lemma~\ref{gcdbalance} we use these edges of $\R'$ to chose copies of $K$ for deletion to ensure that condition (i) of Lemma~\ref{gcd1balance} is satisfied. Note that Lemma~\ref{edgevectors} would not hold if $d$ had some prime factor $p$ equal to $r$, as demonstrated by the $k$-graph constructed in Proposition~\ref{extrem3} for this value of $p$. So Lemma~\ref{edgevectors} is the point in the proof of Lemma~\ref{main1simple} at which the minimum codegree condition $\delta(H) \geq n/p + \alpha n$ is necessary (in the case $\gcd(K) > 1$).

For this lemma we use a slightly different definition of index vector. Let $\Part$ be a partition of a set $X$ into parts $X_1, \dots, X_r$; then for a given $d$ and any $S \subseteq X$ we now define the index vector $\ib^d_\Part(S)$ of $S$ with respect to $\Part$ to be the vector in $\Z_d^r$ whose $j$-th coordinate is $|S \cap X_j|$ modulo~$d$ (whereas our previous definition had $r$ in place of $d$). Again, we sometimes omit the subscript $\Part$ and write simply $\ib^d(S)$ if $\Part$ is clear from the context. Recall that $\ub_j$ denotes the $j$th unit vector of $\Z^r_d$, \emph{i.e.} the vector whose $j$th coordinate is equal to one with all other coordinates equal to zero.

\begin{lemma} \label{edgevectors}
Suppose that $k, d$ and $r$ are positive integers such that $k \geq 3$ and every prime factor of $d$ is strictly greater than $r$. Let $H$ be a $k$-graph on vertex set $X$, and let $\Part$ partition $X$ into parts $X_1, \dots, X_r$.
Also suppose that for any $j_1, \dots, j_{k-1} \in [r]$ there is an edge $\{u_1, \dots, u_k\} \in H$ with $u_i \in X_{j_i}$ for every $i \in [k-1]$. 
Then for any $\vb = (v_1, \dots, v_r) \in \Z^r_d$ such that $d$ divides $\sum_{i=1}^{r} v_i$ there exist a set $S$ of at most $(r+1)^2$ edges of $H$ and integers $a_e$ for $e \in S$ such that  $0 \leq a_e \leq d-1$ for each $e \in S$, $d$ divides $\sum_{e \in S} a_e$ and, working in $\Z^r_d$, we have $\sum_{e \in S} a_e \ib^d(e) = \vb$. 
\end{lemma}

Note that we do not assume that $H$ is $\Part$-partite. Also, throughout the proof of Lemma~\ref{edgevectors} we work within $\Z^r_d$ for all vector calculations (so all equalities of vectors should be interpreted in this context). \medskip

\begin{proof}
We fix $k$ and $d$, and proceed by induction on $r$; for this note that the fact that every prime factor of $d$ is strictly greater than $r$ implies that every prime factor of $d$ is strictly greater than $r'$ for any $r' \leq r$. For $r = 1$ the lemma is trivial since we must have $\vb = (0)$. So fix $r \geq 2$, and assume that the lemma holds with $r-1$ in place of~$r$.

We claim that for some distinct $i, j \in [r]$ the vector $\ub_i - \ub_j$ can be written as an integer combination of at most $r$ members of $\D := \{\ib^d(e) - \ib^d(e') : e, e' \in H\}$ (that is, there are integers $c_1, \dots, c_p$ and vectors $\xb_1, \dots, \xb_p \in \D$ such that $p \leq r$ and $\ub_i - \ub_j = \sum_{i \in [p]} c_i\xb_i$; we don't place any other restrictions on the integers $c_i$). To see that this is true, suppose for a contradiction that the claim is false, and fix any distinct $i, j \in [r]$. Since $k \geq 3$, our assumption on $H$ allows us to choose an edge of $e \in H$ which has at least two vertices in $X_i$. 
Similarly we may choose an edge $e' \in H$ such that $\ib^d(e') = \ib^d(e) - 2\ub_i + \ub_j + \ub_\ell$ for some $\ell \in [r]$.
Then $\ib^d(e) - \ib^d(e') = 2\ub_i - \ub_j - \ub_\ell \in \D$. 
If $\ell = i$, then this gives $\ub_i - \ub_j \in \D$, giving a contradiction (since $\ub_i - \ub_j$ can then be expressed as an integer combination of a single member of $\D$).
Similarly, if $\ell = j$, then we obtain $2\ub_i - 2\ub_j \in \D$. Since $r \geq 2$ we know that $d$ is odd, so $d' := (d+1)/2$ is an integer with $d'(2\ub_i - 2\ub_j) = \ub_i - \ub_j,$ and so $\ub_i - \ub_j$ is an integer combination of a single member of $\D$, again giving a contradiction. 
So we must have $\ell \neq i, j$; since $i$ and $j$ were arbitrary this implies that for any distinct $i, j \in [r]$ there is some $\ell = \ell(i, j)$ which is distinct from $i$ and $j$ such that $\xb_{i, j} := 2 \ub_i - \ub_j - \ub_\ell \in \D$. 
Fix any $i$ and write $f(j) := \ell(i, j)$ for each $j \neq i$. Then for any $j \neq i$ we can write
\begin{align*}
\xb_{i, j} - \xb_{i, f(j)} = (2 \ub_i - \ub_j - \ub_{f(j)}) - (2 \ub_i - \ub_{f(j)} - \ub_{f(f(j))}) = \ub_{f(f(j))}-\ub_j.
\end{align*}
This expresses $\ub_j - \ub_{f(f(j))}$ as an integer combination of $2 \leq r$ members of $\D$, giving another contradiction unless $f(f(j)) = j$ for any $j \neq i$. So we may assume that the family $\F_i := \{\{j, f(j)\} : j \in [r] \sm \{i\}\}$ is a partition of $[r] \sm \{i\}$ into $(r-1)/2$ pairs (note in particular this implies that $r$ is odd). Then write
$$ \yb_i := \sum_{\{j, \ell\} \in \F_i} \xb_{i, j} = \sum_{\{j, \ell\} \in \F_i} 2 \ub_i - \ub_j - \ub_\ell = (r-1) \ub_i - \sum_{j \in [r] \sm \{i\}} \ub_j.$$
So $\yb_i$ can be written as an integer combination of at most $(r-1)/2$ members of $\D$.
Since $r$ and $d$ are coprime, we may fix integers $\lambda, \mu$ such that $\lambda r + \mu d = 1$, whereupon $\lambda (\yb_1 -\yb_2) = \lambda r \ub_1 - \lambda r \ub_2 = \ub_1 - \ub_2$ can be written as an integer combination of at most $r-1$ members of $\D$, giving a final contradiction which completes the proof of the claim. 

We may therefore assume without loss of generality that $\ub_{r-1} - \ub_r$ can be written as an integer combination of at most $r$ elements of $\D$. That is, we may choose a set $S_1$ of at most $2r$ edges of $H$ and integers $m_e$ for $e \in S_1$ such that $\sum_{e \in S_1} m_e = 0$ and $\sum_{e \in S_1} m_e \ib^d_\Part(e) = \ub_{r-1} - \ub_r$.
Let $Y_j = X_j$ for each $j \in [r-2]$, and let $Y_{r-1} = X_{r-1} \cup X_r$. Then $H$ is a $k$-graph on vertex set $X = Y_1 \cup \dots \cup Y_{r-1}$ such that for any $j_1, \dots, j_{k-1} \in [r-1]$ there is an edge $\{u_1, \dots, u_k\} \in H$ with $v_i \in Y_{j_i}$ for every $i \in [k-1]$. Let $\Qart$ denote the partition of $X$ into the parts $Y_1, \dots, Y_{r-1}$; then by our induction hypothesis we may choose a set $S_2$ of at most $r^2$ edges of $H$ and integers $n_e$ for $e \in S$ such that $d$ divides $\sum_{e \in S_2} n_e$ and $\sum_{e \in S_2} n_e \ib^d_\Qart(e) = (v_1, \dots, v_{d-2}, v_{d-1}+v_d)$. The latter equation implies that $\sum_{e \in S_2} n_e \ib^d_\Part(e) = (v_1, \dots, v_{d-2}, y, z)$ for some $y$ and $z$ with $y+z = v_{d-1}+v_d$ modulo $d$, and so
$$(v_{d-1}-y) (\ub_{r-1} - \ub_r) + \sum_{e \in S_2} n_e \ib^d_\Part(e) = (v_1, \dots, v_{d-2}, v_{d-1}, v_d) = \vb.$$
Let $S := S_1 \cup S_2$ and let integers $0 \leq a_e \leq d-1$ satisfy $a_e \equiv (v_{d-1} - y)m_e + n_e \mod d$ for each $e \in S$ (we take $m_e  = 0$ for any $e \notin S_1$ and $n_e = 0$ for any $e \notin S_2$). Then $S$ is a set of at most $r^2 + 2r \leq (r+1)^2$ edges of $H$, and the equation above shows that $\sum_{e \in S} a_e \ib^d_\Part(e) = \vb$. Finally, 
$$\sum_{e \in S} a_e \equiv (v_{d-1} - y) \sum_{e \in S} m_e + \sum_{e \in S} n_e \equiv 0 + 0 \equiv 0 \mod d,$$
so $d$ divides $\sum_{e \in S} a_e$, as required.
\end{proof}

Finally, we can now give the full statement and proof of Lemma~\ref{gcdbalance}, showing that we can delete a $K$-packing in $H$ so that, following these deletions, all subclusters have size divisible by $bk\gcd(K)$. We achieve this by deleting five vertex-disjoint $K$-packings in succession. The first deletion is simple and ensures that the total number of vertices is divisible by $b\gcd(K)$, whilst the second uses Lemma~\ref{edgevectors} to ensure that $\gcd(K)$ divides the total number of vertices in subclusters within any component of $\Sa$. The third then uses Lemma~\ref{gcd1balance} to ensure that $\gcd(K)$ divides the size of each subcluster, and the fourth (again straightforward) maintains this property whilst also ensuring that $bk\gcd(K)$ divides the total number of vertices. Finally, our fifth deletion uses Lemma~\ref{gcd1balance} again to ensure that $bk\gcd(K)$ divides the number of vertices within any subcluster.
 
\begin{lemma} \label{gcdbalance} 
Suppose that $N, s, t, b$ and $k$ are integers such that $1/N \ll 1/t, 1/s \ll 1/b, 1/k$. Let $K$ be the complete $k$-partite $k$-graph with vertex class sizes $b_1, \dots, b_k$, where $b_1 + \dots + b_k = b$, and suppose that $\gcd(K)$ is defined, that $s$ is divisible by $k \gcd(K)$ and that $b_1$ and $\gcd(K)$ are coprime. Next let $G$ be a $t$-partite $k$-graph with vertex classes $Y_1, \dots, Y_t$, and suppose that $b$ divides $|Y|$, where $Y = \bigcup_{i \in [t]} Y_i$. Finally suppose that $\R$ is a connected $k$-graph on $[t]$, and $\Sa$ is a graph on $[t]$ with $r$ connected components $C_1, \dots, C_r$, such that the following properties hold.
\begin{enumerate}[(i)]
\item For any edge $e \in \R$ there are more than $N$ vertex-disjoint copies of $\B(b(s+1))$ in $G[\bigcup_{j \in e} Y_j]$,
\item For any edge $uv \in \Sa$ there is a set $T \in \binom{[t]\sm \{u, v\}}{k-1}$ such that $G[\bigcup_{j \in \{u, v\} \cup T} Y_j]$ contains more than $N$ vertex-disjoint copies of $\Phi(b(s+1))$ whose end vertex classes lie in $Y_u$ and $Y_v$ and whose central vertex classes lie in the sets $Y_j$ for $j \in T$. 
\item If $\gcd(K) > 1$, then $r$ (the number of components of $\Sa$) is smaller than the least prime factor of $\gcd(K)$, and for any $i_1, \dots, i_{k-1} \in [r]$ there is some edge $e = \{u_1, \dots, u_k\}$ of $\R$ such that $u_j \in V(C_{i_j})$ for each $j \in [k-1]$.
\end{enumerate} 
Then $G$ contains a $K$-packing $M$ of size at most $N/2b$ such that $bk\gcd(K)$ divides $|Y_j \sm V(M)|$ for every $j \in [t]$.
\end{lemma}
 
\begin{proof}
As in Lemma~\ref{gcd1balance}, we prove the lemma by repeatedly choosing some vertex-disjoint copies of $K$ in $G$ and deleting their vertices from $G$; as there, we continue to write $Y_j$, $Y$ and~$G$ for the sets and graph obtained following these deletions. We shall verify at the end of the proof that the $K$-packing $M$ formed by all the deleted copies of $K$ has size at most $N/2b$, so $M$ covers at most $N/2$ vertices. With this in mind, we can always assume that (i) and (ii) provide at least $N/2$ copies of $\B(b(s+1))$ and $\Phi(b(s+1))$ of the given forms.

Our first step is to delete at most $\gcd(K)$ pairwise vertex-disjoint copies of $K$ from $G$ so that, following these deletions, we have that $b\gcd(K)$ divides $|Y|$. Since each copy of $K$ has $b$ vertices, and $b$ divides $|Y|$, we can indeed achieve this by deleting at most $\gcd(K)$ pairwise vertex-disjoint copies of $K$ from $G$; by (i) these copies can be chosen from $G[\bigcup_{j \in e} Y_j]$ for an arbitrary edge $e \in \R$.
 
The next step is to delete at most $(r+1)^2\gcd(K)$ copies of $K$ from $G$ so that $\gcd(K)$ divides $\sum_{u \in V(C_i)} |Y_u|$ for each component $C_i$ of $\Sa$. If $\gcd(K) = 1$ then no deletions are necessary, whilst if $\gcd(K) > 1$ then we use Lemma~\ref{edgevectors}. For each $i \in [r]$, write $V_i := \bigcup_{u \in V(C_i)} Y_u$, and define $v_i \in \{0, \dots, \gcd(K)-1\}$ to be such that $v_i b_1 \equiv |V_i| \mod \gcd(K)$ (since $\gcd(K)$ and $b_1$ are coprime a unique such $v_i$ exists). Then since $\gcd(K)$ divides $|Y|$, we have $\sum_{i \in [r]} v_i b_1 \equiv |Y| \equiv 0$ modulo $\gcd(K)$, so $\sum_{i \in [r]} v_i$ is divisible by $\gcd(K)$. We may therefore apply Lemma~\ref{edgevectors} with $\R, \gcd(K)$ and the sets $V(C_i)$ in place of $H, d$ and the sets $X_i$ respectively to obtain at most $(r+1)^2$ edges $e_1, \dots, e_p \in \R$ and integers $a_1, \dots, a_p \in \{0, 1, \dots, \gcd(K)-1\}$ so that $\gcd(K)$ divides $\sum_{j \in [p]} a_j$ and (working in $\Z^r_{\gcd(K)}$) we have
$$\sum_{j \in [p]} a_j \ib^d_\Qart(e_j) = (v_1, \dots, v_r),$$
where $\Qart$ denotes the partition of $[t]$ into parts $V(C_i)$ for $i \in [r]$. For each $j \in [p]$ by (i) we may choose $a_j$ pairwise vertex-disjoint copies of $\B(b(s+1))$ in $G[\bigcup_{\ell \in e_j} Y_\ell]$, within which we can find $a_j$ pairwise vertex-disjoint copies of $K$. Delete all of these copies of $K$ from~$G$. By definition of $\gcd(K)$, each vertex class of $K$ has size $b_1$ modulo $\gcd(K)$. Furthermore, since $G$ is $t$-partite, each of the deleted copies of $K$ has one vertex class contained in $Y_\ell$ for each $\ell \in e_j$. So for any $i \in [r]$, the total number of vertices deleted from $V_i$ is 
$$b_1 \sum_{j \in [p]} a_j |e_j \cap V(C_i)| \equiv b_1 v_i \equiv |V_i| \mod \gcd(K).$$
So following these deletions we have that $\gcd(K)$ divides $|V_i| = \sum_{u \in V(C_i)} |Y_u|$ for each component $C_i$ of $\Sa$. Furthermore, since in total $\sum_{j \in [p]} a_j$ copies of $K$ were deleted, each with $b$ vertices, our assumption that $\gcd(K)$ divides $\sum_{j \in [p]} a_j$ implies that the total number of vertices deleted is divisible by $b\gcd(K)$. So we still have that $b\gcd(K)$ divides $|Y|$ after these deletions.

We now delete at most $N/3b$ further copies of $K$ from $G$ so that $\gcd(K)$ divides $|Y_i|$ for every $i \in [t]$. For this we use Lemma~\ref{gcd1balance} with $\B(b\gcd(K)), \gcd(K)$ and the sets $Y_i$ in place of $K, d$ and the sets $X_i$ respectively, with $G, \Sa$ and $t$ playing the same role here as there, and with $d'=1$. Then condition (i) of Lemma~\ref{gcd1balance} is satisfied as a consequence of our last round of deletions. Also, by (ii), for any edge $uv \in \Sa$ we can choose $N/2 \geq (bk\gcd(K))\gcd(K)t^2$ vertex-disjoint copies of $\Phi(b(s+1))$ in $G$ whose end vertex classes lie in $Y_u$ and $Y_v$, and whose central vertex classes are each a subset of some $Y_j$. Within each of these copies of $\Phi(b(s+1))$ we can find a copy of $\B(b\gcd(K))$ with one vertex in $Y_v$, $b\gcd(K)-1 \equiv -1 \mod \gcd(K)$ vertices in $Y_u$, and $b\gcd(K) \equiv 0 \mod \gcd(K)$ vertices in each of the other vertex classes intersected by this copy of $\Phi(b(s+1))$. So condition (ii) of Lemma~\ref{gcd1balance} is satisfied also, and so Lemma~\ref{gcd1balance} yields a $\B(b\gcd(K))$-packing $M'$ in $G$ of size at most $\gcd(K)t^2$ such that, deleting all vertices covered by $M'$ from $G$, we find that $\gcd(K)$ divides $|Y_i|$ for every $i \in [t]$. Then, since $\B(b\gcd(K))$ has $kb\gcd(K)$ vertices, it remains the case that $b \gcd(K)$ divides $|Y|$ following these deletions. Recall that $\B(K) = \B(b)$, so $\B(b\gcd(K))$ admits a perfect $\B(K)$-packing of size $\gcd(K)$, whilst $\B(K)$ admits a perfect $K$-packing of size $k$ by Proposition~\ref{hpackings}. So there is a $K$-packing in $G$ which covers the same vertices as $M'$, so we did indeed delete a $K$-packing of size at most $kt^2\gcd(K)^2 \leq N/5b$ in this step.

Next we delete at most $k\gcd(K)$ further copies of $K$ from $G$ so that, following these deletions, we have that $bk\gcd(K)$ divides $|Y|$, as well as preserving the property that $\gcd(K)$ divides $|Y_i|$ for every $i \in [t]$. Since $b\gcd(K)$ divides $|Y|$, we can achieve the latter property by deleting $z\gcd(K)$ copies of $K$ for some integer $0 \leq z \leq k-1$. So choose an arbitrary edge $e \in \R$, and use (i) to choose $z\gcd(K)$ vertex-disjoint copies of $K$ in $G[\bigcup_{i \in e} Y_e]$; then for any $j \in [t]$ the number of vertices deleted from $Y_j$ is equal to $z \gcd(K) b_1 \equiv 0$ modulo $\gcd(K)$, so $|Y_j|$ is still divisible by $\gcd(K)$ following these deletions.

Finally, we apply Lemma~\ref{gcd1balance} again to delete a final set of at most $N/5b$ copies of $K$ from $G$ so that, following these deletions, $bk\gcd(K)$ divides $|Y_i|$ for every $i \in [t]$, giving the desired $K$-packing $M$. We shall use the adjacency graph $\mathrm{Adj}(\R)$ in place of $\Sa$; since $\R$ is connected, $\mathrm{Adj}(\R)$ is connected also, and so condition~(i) of Lemma~\ref{gcd1balance} holds (with $bk\gcd(K)$ in place of $d$) as a consequence of our last round of deletions. Furthermore, for any edge $uv$ of $\mathrm{Adj}(\R)$ there is an edge $e \in \R$ containing $u$ and $v$. So by (i) there are at least $N/2 > (bsk)(bk\gcd(K))t^2$ vertex-disjoint copies of $\B(b(s+1))$ in $G[\bigcup_{\ell \in e} Y_\ell]$, each of which contains a copy of $\U_s(K)$ with $bs-\gcd(K)$ vertices in $Y_u$, $bs+\gcd(K)$ vertices in $Y_v$ and $bs$ vertices in $Y_w$ for any $w \in e \sm \{u, v\}$ 
(we can assume that $\U_s(K)$ is defined since we assumed that $1/s \ll 1/b, 1/k$). Since $bs \equiv 0$ modulo $bk\gcd(K)$, we may apply Lemma~\ref{gcd1balance} with $\U_s(K), \mathrm{Adj}(\R), \gcd(K), bk\gcd(K)$ and the sets $Y_j$ in place of $K, \Sa, d', d$ and the sets $X_j$ respectively, whereupon the requirement that $\gcd(K)$ divides $|Y_j|$ for every $j \in [t]$ is satisfied by our previous deletions. 
This gives a $\U_s(K)$-packing $M''$ in $G$ of size at most $bk\gcd(K)t^2$ in $G$ such that, deleting all members of $M''$ from $G$, we find that $bk\gcd(K)$ divides $|Y_i|$ for every $i \in [t]$. 
Since $\U_s(K)$ admits a perfect $K$-packing of size $ks$ by Proposition~\ref{hpackings}, we may treat $M''$ as being a $K$-packing in $G$ of size at most $k^2bs\gcd(K)t^2 \leq N/5b$, as required.

To complete the proof we must show that at most $N/2b$ copies of $K$ were deleted in total. Indeed, we deleted at most $\gcd(K)$ copies in the first step, at most $\sum_{j \in [p]} a_j \leq (r+1)^2\gcd(K)$ copies in the second step, at most $N/5b$ copies in the third step, at most $k\gcd(K)$ copies in the fourth step, and at most $N/5b$ copies in the final step, that is, fewer than $N/2b$ copies in total.
\end{proof} 
 
\subsection{Ensuring equality of subcluster sizes} \label{sec:deleteeq}

In Lemma~\ref{main2simple}, each vertex class of the $k$-partite $k$-graph $K$ has equal size $b_1$, so it is insufficient to delete a $K$-packing in $H$ so that every cluster satisfies certain divisibility conditions. Instead, we must ensure that the clusters in each of our robustly universal $k$-partite $k$-graphs have equal size. In this section we prove Lemma~\ref{samesizebalance}, which gives sufficient conditions for this to be possible. The stronger minimum codegree condition on $H$ will ensure that we can satisfy these conditions, and so apply Lemma~\ref{samesizebalance} in the proof of Lemma~\ref{main2simple} similarly as Lemma~\ref{gcdbalance} is used in the proof of Lemma~\ref{main1simple}. However, the results of this section are stated in a more general form which we can also use in the proof of Theorem~\ref{cyclepack}. 

In this section we proceed under the following setup, in which the $k$-graphs $G$ and $\R$, the directed graph $\Sa^+$, base graph $\Sa$ and clusters $U_i$ play the roles described in Section~\ref{sec:outline}. A \emph{directed graph} $D$ consists of a vertex set $V$ and a set of edges $E$, where each edge is an ordered pair $(u, v)$ with $u, v \in V$ and $u \neq v$. We write $u \to v$ to mean that $(u, v) \in E$. The \emph{outdegree} $\deg^+(u)$ of a vertex $u \in V$ is the number of vertices $v \in V$ for which $u \to v$, and the \emph{minimum outdegree} of $D$ is  $\delta^+(D) := \min_{u \in V} \deg^+(u)$. Finally, the \emph{base graph} of $D$ is the (undirected) graph $G$ on $V$ in which $uv$ is an edge of $G$ if either $u \to v$ or $v \to u$ in $D$. 

\begin{setup} \label{balancesetup}
Let $\R$ be a $k$-graph with vertex set $[m]$ which admits a perfect matching $M_\R$, and for any $i \in [m]$ let $e(i)$ denote the edge of $M_\R$ which contains $i$. Let $\Sa^+$ be a directed graph on $[m]$, and $\Sa$ be the base graph of $\Sa^+$. Also let $G$ be an $m$-partite $k$-graph with vertex classes $U_1, \dots, U_m$ each of size $n$, and let $K$ be the $k$-partite $k$-graph whose vertex classes each have size $b_1$. Suppose also that for any sets $V_\ell \subseteq U_\ell$ with $|V_\ell| \geq n/2$ for each $\ell \in [m]$ the following statements hold for $V := \bigcup_{\ell \in [m]} V_\ell$.
\begin{enumerate}[(i)]
\item For any edge $e \in \R$ there are at least $N$ vertex-disjoint copies of $K$ in $G[\bigcup_{\ell \in e} V_\ell]$.
\item For any edge $i \to j$ of $\Sa^+$, there are at least $N$ vertex-disjoint copies of $\Phi(b_1)$ in $G[V]$ whose end vertex classes lie in $V_i$ and $V_j$ and whose central vertex classes lie in the sets $V_\ell$ for $\ell \in e(i) \sm \{i\}$.
\end{enumerate}
\end{setup}

Note that (ii) implies that $e(i) \neq e(j)$ for any edge $ij \in \Sa$.
If $i \to j$ is an edge of $\Sa^+$, then condition (ii) allows us to choose a copy of $K$ in $G$ with one vertex in $V_j$, $b_1-1$ vertices in $V_i$, and $b_1$ vertices in $V_\ell$ for $\ell \in e(i) \sm \{i\}$. Deleting this copy reduces the size of $V_j$ by one relative to the sets $V_\ell$ for $\ell \in e(j) \sm \{j\}$, and increases the size of $V_i$ by one relative to the sets $V_\ell$ for $\ell \in e(i) \sm \{i\}$. The next lemma states that, if $\R$ is irreducible on $M_\R$, we can delete copies of $K$ to achieve the same effect if $j \to i$ is an edge of $\Sa^+$. This allows us to ignore the direction of edges of $\Sa^+$ and consider only the base graph $\Sa$ when proving Lemma~\ref{samesizebalance}, which significantly simplifies the argument.

\begin{lemma} \label{bidirection}
Adopt Setup~\ref{balancesetup}, and suppose additionally that $\R$ is $(C, L)$-irreducible on $M_\R$ for some $C$ and $L$ with $kb_1L + b_1C \leq N$, and that we now have fixed sets $V_\ell$ with $|V_\ell| \geq n/2$ for each $\ell \in [m]$. Then for any $i, j \in [m]$ with $ij \in \Sa$ there is a $K$-packing $M$ in $G[V]$ of size at most $C+L$ such that, if we write $M(\ell) := |V(M) \cap V_\ell|$ for $\ell \in [m]$, we have the following properties.
\begin{enumerate}[(i)]
\item $M(\ell)$ is divisible by $b_1$ for any $\ell \in [m] \sm \{i, j\}$.
\item $M(i) = M(\ell) - 1$ for any $\ell \in e(i) \sm \{i\}$.
\item $M(j) = M(\ell) + 1$ for any $\ell \in e(j) \sm \{j\}$.
\item $M(\ell) = M(\ell')$ for any $e \in M_\R \sm \{e_i, e_j\}$ and any $\ell, \ell' \in e$.
\end{enumerate}
\end{lemma}

\begin{proof}
Fix any $i, j \in [m]$ with $ij \in \Sa$. Then there is an edge between $i$ and $j$ in $\Sa^+$, directed either $i \to j$ or $j \to i$. Suppose first that $i \to j$; then there are at least $N \geq 1$ vertex-disjoint copies of $\Phi(b_1)$ in $G[V]$ whose end vertex classes lie in $V_i$ and $V_j$ and whose central vertex classes all lie in the sets $V_\ell$ for $\ell \in e(i) \sm \{i\}$. Inside any one of these copies we can find a copy of $K$ with $b_1-1$ vertices in $V_i$, with one vertex in $V_j$, and with $b_1$ vertices in $V_\ell$ for each $\ell \in e(i) \sm \{i\}$. We can then take $M$ to consist of this single copy of $K$.

So we may assume that $j \to i$, so there are at least $N$ vertex-disjoint copies of $\Phi(b_1)$ in $G[V]$ whose end vertex classes lie in $V_i$ and $V_j$ and whose central vertex classes all lie in the sets $V_\ell$ for $\ell \in e(j) \sm \{j\}$. Also, since $\R$ is $(C, L)$-irreducible on $M_\R$, we may choose $c \leq C$ and multisets $T$ and $T'$ of edges of $\R$ and $M_\R$ respectively such that $|T|, |T'| \leq L$ and such that $j$ appears in $c$ more edges of $T$ than of $T'$, $i$ appears in $c$ more edges of $T'$ than of $T$, and any $\ell \in [m] \sm \{i, j\}$ appears equally often in $T$ as in $T'$. For each edge $e \in T$, with multiplicity, choose a copy of $K$ in $G[\bigcup_{i \in e} V_i]$; since $N \geq kb_1L$ we may do this so that these copies are all vertex-disjoint. This gives a $K$-packing $M'$ in $G[V]$ of size at most $L$ such that $M'(\ell)$ is divisible by $b_1$ for any $\ell \in [m]$, $M'(i) = M'(\ell) - cb_1$ for any $\ell \in e(i) \sm \{i\}$, $M'(j) = M'(\ell) + cb_1$ for any $\ell \in e(j) \sm \{j\}$, and $M'(\ell) = M'(\ell')$ for any $e \in M_\R \sm \{e_i, e_j\}$ and any $\ell, \ell' \in e$. Now, since $M'$ covers at most $Lkb_1$ vertices of $G[V]$, and $N - Lkb_1 \geq cb_1-1$, we may choose $cb_1 -1 $ vertex-disjoint copies of $\Phi(b_1)$ in $G[V]$ which do not have any vertices in common with $M'$, whose end vertex classes lie in $V_i$ and $V_j$ and whose central vertex classes all lie in the sets $V_\ell$ for $\ell \in e(j) \sm \{j\}$. Each of these copies contains a copy of $K$ with one vertex in $V_i$, $b_1-1$ vertices in $V_j$, and $b_1$ vertices in $V_\ell$ for each $\ell \in e(j) \sm \{j\}$; adding these copies of $K$ to $M'$ gives the desired $K$-packing~$M$.
\end{proof}

We can now prove the main result of this section.

\begin{lemma} \label{samesizebalance}
Adopt Setup~\ref{balancesetup}, and assume that $1/n, 1/m, \beta \ll \alpha, 1/C, 1/L, 1/k, 1/b_1$, and also that $N \geq kb_1L + b_1C$ and $2n/3 \leq n' \leq n$. Let subsets $Y_j \subseteq U_j$ satisfy $(1-\beta) n' \leq |Y_j| \leq n'$ for each $j \in [m]$, and suppose that $b_1$ divides~$|Y|$, where $Y := \bigcup_{j \in [m]} Y_j$. Also let $X_1, \dots, X_s$ be sets which partition $[m]$ such that 
\begin{enumerate}[(i)]
\item For any $T \subseteq [m]$ with $|T| \leq \alpha m$, any $i \in [s]$ and any $x, y \in X_i \sm T$, there is a path from $x$ to $y$ in $\Sa[[m] \sm T]$ of length at most $L$. 
\item For any submatching $M_\R' \subseteq M_\R$ of size $|M_\R'| \geq (1-\alpha/2)|M_\R|$ the subgraph $\R[V(M_\R')]$ is $(C, L)$-irreducible on $M_\R'$.
\item The average size $Q_i := \sum_{j \in X_i} |Y_j|/|X_i|$ of sets $Y_j$ corresponding to vertices of $X_i$ is the same for every $i \in [s]$, and this common average size $Q := Q_1 = \dots = Q_s$ is an integer which is divisible by $b_1$.
\end{enumerate}
Then there is a $K$-packing $M$ in $G[Y]$ such that 
\begin{enumerate}[(a)]
\item for any edge $e \in M_\R$ the sets $Y_j \sm V(M)$ for $j \in e$ have equal size, and furthermore,
\item this common size is a multiple of $b_1$.
\end{enumerate}
\end{lemma}

\begin{proof}
For each $j \in [m]$ let $n_j := |Y_j| - Q$, that is, the difference between the size of $Y_j$ and the average size of the sets $Y_\ell$. So $n_j$ is an integer with $|n_j| \leq \beta n' \leq \beta n$ for any $j \in [m]$, and for any $i \in [s]$ we have $\sum_{j \in X_i} n_j = 0$ by (iii). Construct a multiset of pairs $\Gamma$ iteratively as follows. Initially take $\Gamma$ to be empty. If $n_j = 0$ for every $j \in [m]$, then terminate. Otherwise, since $\sum_{j \in X_i} n_j = 0$ for every $i \in [s]$, there must be $i \in [s]$ and $j, j' \in X_i$ such that $n_j < 0$ and $n_{j'} > 0$. Add $(j, j')$ to $\Gamma$, increment $n_j$ by one and decrement $n_{j'}$ by one, and repeat. Since $|n_j| \leq \beta n$ for each $j \in [m]$, we must terminate after at most $\beta nm$ steps. At this point, we have $|\Gamma| \leq \beta nm$, and, writing $s_j$ for the number of times $j$ appears as the first coordinate of a member of $\Gamma$, and $t_j$ for the number of times $j$ appears as the second coordinate of a member of $\Gamma$, we have 
\begin{equation} \label{gcdbalanceeq}
|Y_j| = Q - s_j + t_j 
\end{equation} 
for every $j \in [m]$. Furthermore, note that, returning to the original values of $n_j$, we have $s_j, t_j \leq |n_j| \leq \beta n$ for any $j \in [m]$.

Arbitrarily order the pairs of $\Gamma$, and take $M$ initially to be empty; we now add at most $L(C+L)$ copies of $K$ to $M$, and delete the vertices covered from $G$, for each member of $\Gamma$ in turn. So suppose that we are considering the $z$th pair in~$\Gamma$, say $(j, j')$. So $z \leq \beta nm$. Prior to this we have added at most $L(C+L)(z-1)$ copies of $K$ to $M$, so at most $kb_1 L(C+L)\beta nm$ vertices have been deleted. For each $\ell \in [m]$ let $Y'_\ell$ consist of the so far undeleted vertices of $Y_\ell$; we assume for now that $Y'_\ell \geq (1-7\alpha)n' \geq n/2$, and will justify this assumption later. Also, let $M'_\R$ be the submatching of $M_\R$ consisting of $e(i)$, $e(j)$, and every $e \in M_\R$ other than $e(j)$ and $e(j')$ for which at most $\alpha n$ vertices of $\bigcup_{z \in e} Y_z$ have previously been deleted, and let $T = [m] \sm V(M'_\R)$ and $Y' = \bigcup_{\ell \in V(M'_\R)} Y'_\ell$. Then $(\alpha n)(|T|/k) \leq kb_1L(C+L) \beta nm$, so we find that $|T| \leq \alpha m$. So $\R(V(M'_\R)$ is $(C, L)$-irreducible on $M_\R$ by (ii), whilst by (i) we may choose a path $j= v_0, \dots, v_p = j'$ from~$j$ to~$j'$ in $\Sa[[m] \sm T]$ of length $p \leq L$ (since $j, j' \in X_i$ for some $i$ by construction of $\Gamma$). Now, for each $x \in [p]$ in turn, apply Proposition~\ref{bidirection} to choose a $K$-packing $M'_x$ in $G[Y']$ of size at most $C+L$ which satisfies the conclusions of Proposition~\ref{bidirection} with $v_{x-1}$ and $v_x$ in place of~$i$ and~$j$, delete the vertices of $M'_x$ from the sets $Y_\ell$, and add the members of $M'_x$ to $M$. Having done this for every $x \in [p]$, the net effect is that there are integers $n_z(e) \leq p(C+L)$ for $e \in M_\R$ with the following property. For any $e \in M_\R$, precisely $b_1n_z(e)$ vertices were deleted from $Y'_\ell$ for every $\ell \in e$, except for two cases: $b_1n_z(e(j))-1$ vertices were deleted from~$Y'_{j}$, whilst $b_1n_z(e(j'))+1$ vertices were deleted from~$Y'_{j'}$. At this point we proceed to consider the $(z+1)$th pair in $\Gamma$, and continue in this manner.

After we have completed this process for every pair in $\Gamma$, we find that for any $j \in [m]$ the total number of vertices that were deleted from the set $Y_j$ is equal to  $t_j - s_j + b_1\sum_{z \leq |\Gamma|} n_z(e(j)).$
Combining this with~\eqref{gcdbalanceeq} we obtain 
$$|Y_j \sm V(M)| = Q - b_1\sum_{z \leq |\Gamma|} n_z(e(j)).$$
Properties (a) and (b) follow, since for any $e \in M_\R$ and $j \in e$ we have $e(j) = e$, and we assumed in (iii) that $Q$ is divisible by $b_1$. So it remains only to justify our assumption that at any point we had $|Y'_\ell| \geq (1-7\alpha)n'$ for any $\ell \in [m]$. For this, fix any $\ell \in [m]$, and consider the number of vertices deleted from $Y_\ell$ over the course of the procedure. If more than $\alpha n$ vertices of $Y_\ell$ were deleted in total, then for some $z$ the number of deleted vertices of $Y_\ell$ first exceeded $\alpha n$ when considering the $z$th pair of $\Gamma$. Whilst considering this pair we deleted at most $L(C+L)$ copies of $K$, and so at the end of this step the number of vertices deleted from $Y_\ell$ was at most $\alpha n + kb_1L(C+L) \leq 2\alpha n$. For all subsequent steps the edge $e(\ell)$ was excluded from $M_\R'$, and so vertices were only deleted from $Y_\ell$ when considering pairs $(j, j')$ for which $j \in e(\ell)$ or $j' \in e(\ell)$. The number of such pairs is at most $\sum_{j \in e(\ell)} s_j+t_j \leq 2k\beta n$, and so at most a further $(kb_1L(C+L))(2k\beta n) \leq \alpha n$ vertices were deleted from $Y_\ell$, giving a total of at most $3\alpha n \leq 6 \alpha n'$ vertices deleted from $Y_\ell$ over the entire course of the procedure. Since initially we had $(1-\beta)n' \leq |Y_\ell|$, this justifies our earlier assumption that $|Y'_\ell| \geq (1-7\alpha)n'$.
\end{proof}

\subsection{Packing complete $k$-partite $k$-graphs}~\label{sec:completegraphs}
The proof of Lemmas~\ref{main1simple} and~\ref{main2simple} will conclude by finding a perfect $K$-packing within the `top layer' $J_=$ of a $D$-universal $k$-partite $k$-complex $J$. Provided $D$ is sufficiently large, such a packing exists if the complete $k$-partite $k$-graph $G$ on the same vertex classes contains a perfect $K$-packing. In this section we give sufficient conditions to ensure that this is the case.

For Lemma~\ref{main2simple}, $K$ is a complete $k$-partite $k$-graph on vertex classes of equal size. In this case it is elementary to determine whether a complete $k$-partite $k$-graph $G$ contains a perfect $K$-packing. 

\begin{fact} \label{completepackingbalanced}
Let $K$ be the complete $k$-partite $k$-graph with vertex classes each of size $b_1$, and let $G$ be the complete $k$-partite $k$-graph with vertex classes $V_1, \dots, V_k$. Then $G$ contains a perfect $K$-packing if and only if $|V_1| = \dots = |V_k|$ and $b_1$ divides $|V_1|$.
\end{fact}

However, for Lemma~\ref{main1simple} things are more complicated. The next lemma gives sufficient conditions which ensure that a complete $k$-partite $k$-graph $G$ contains a perfect $K$-packing when $K$ is a complete $k$-partite $k$-graph as in Lemma~\ref{main1simple}. Recall for this the definitions of $\B(K)$ and $\U_s(K)$ (Definition~\ref{defBULo}).

\begin{lemma} \label{completepackingnearlybalanced}
Suppose that $1/n \ll \beta \ll 1/b, 1/k$. Let $K$ be the complete $k$-partite $k$-graph with vertex classes of size $b_1, \dots, b_k$, where $b_1 + \dots + b_k = b$ and the $b_i$ are not all equal. Also let $G$ be a complete $k$-partite $k$-graph with vertex classes $V_1, \dots, V_k$, where $V := V_1 \cup \dots \cup V_k$ has size $n$. Suppose that
\begin{enumerate}[(i)]
\item $bk\gcd(K)$ divides $|V_i|$ for each $i \in [k]$, and
\item $|V_j| \geq n/k - \beta n$ for every $j \in [k]$. 
\end{enumerate}
Then $G$ contains a perfect $K$-packing.
\end{lemma}

\begin{proof}
Introduce an integer $s$ with $\beta \ll 1/s \ll 1/b, 1/k$. Then we may assume that $\U_s(K)$ is defined and so contains a perfect $K$-packing by Proposition~\ref{hpackings}; the same is true of $\B(K)$. Also note that (i) implies that $bk\gcd(K)$ divides $n$.
For each $i \in [k-1]$ define 
\begin{equation} \label{completeeq1}
d_i := \frac{|V_i| - n/k}{\gcd(K)} + d_{i-1},
\end{equation}
with $d_0$ taken to be zero; then each $d_i$ must be an integer by (i). So by (ii) we have $|d_i| \leq k\beta n + |d_{i-1}|$, which implies that $\sum_{i \in [k-1]}|d_i| \leq k^3\beta n$. Now, for each $i \in [k-1]$, if $d_i$ is positive then delete $d_i$ copies of $\U_s(K)$ from $G$, each with $bs+\gcd(K)$ vertices in $V_i$, $bs-\gcd(K)$ vertices in $V_{i+1}$ and $bs$ vertices in each other vertex class. On the other hand, if $d_i$ is negative then delete $d_i$ copies of $\U_s(K)$ from $G$ each with $bs-\gcd(K)$ vertices in $V_i$, $bs+\gcd(K)$ vertices in $V_{i+1}$ and $bs$ vertices in each other vertex class. We also insist that all of these copies of $\U_s(K)$ are pairwise vertex-disjoint; this is not a problem since the total number of copies of $\U_s(K)$ deleted is $N := \sum_{i \in [k-1]} |d_i| \leq k^3\beta n$, and so the total number of vertices deleted is
$$kbsN \leq bsk^4 \beta n \leq n/k -  \beta n \leq \min_{i \in [k]} |V_i|.$$
For each $i \in [k]$ let $X_i$ consist of the undeleted vertices of $V_i$, and let $X := X_1 \cup \dots \cup X_k$. Then by~\eqref{completeeq1} we obtain
$$ |X_i| = |V_i| - bsN - (d_i - d_{i-1})\gcd(K) = n/k - bsN. $$
We conclude that $|X_1| = \dots = |X_k|$ and that $b$ divides $|X_1|$. So $G[X]$ contains a perfect $\B(K)$-packing by Fact~\ref{completepackingbalanced}; combined with the previously deleted copies of $\U_s(K)$ this gives a perfect $\{\B(K), \U_s(K)\}$-packing of $G$. Propositions~\ref{fpacktohpack} and~\ref{hpackings} then imply that $G$ contains a perfect $K$-packing.
\end{proof}

Unfortunately, Lemma~\ref{completepackingnearlybalanced} is not strong enough for our purposes. Indeed, it requires that the vertex classes of $G$ all have approximately equal size, whilst we wish to find a perfect $K$-packing within a `lopsided' complete $k$-partite $k$-graph $G$. For this we use the following corollary, which shows that we can indeed do this provided that $\sigma(G) > \sigma(K) + o(1)$, that is, if $G$ is `less lopsided' than $K$ (it is not hard to see that if $G$ and $K$ are complete $k$-partite $k$-graphs and $\sigma(G) < \sigma(K)$ then there can be no perfect $K$-packing in $G$).
Recall for this corollary the definition of $\Lo(K)$ (Definition~\ref{defBULo}); in particular, $\Lo(K)$ has $(k-1)!b$ vertices in total, with one vertex class of size $(k-1)!b\sigma(K)$ and $k-1$ vertex classes of size $(k-2)!b(1-\sigma(K))$. 

\begin{coro} \label{completepacking}
Suppose that $1/n \ll \beta \ll \alpha, 1/b, 1/k$. Let $K$ be the complete $k$-partite $k$-graph with vertex classes of size $b_1, \dots, b_k$, where $b_1 + \dots + b_k = b$ and the $b_i$ are not all equal. Also let $G$ be a complete $k$-partite $k$-graph with vertex classes $V_1, \dots, V_k$, where $|V_1| \leq |V_2|, \dots, |V_k|$ and $V := V_1 \cup \dots \cup V_k$ has size $n$. Suppose that 
\begin{enumerate}[(i)]
\item $\sigma(G) \geq \sigma(K) + \alpha$,
\item $||V_i| - |V_j|| \leq \beta n$ for every $i, j \in \{2, \dots, k\}$, and 
\item $bk\gcd(K)$ divides $|V_i|$ for any $i \in [k]$. 
\end{enumerate}
Then $G$ contains a perfect $K$-packing. 
\end{coro}

\begin{proof}
Let $d := \gcd(K)$, $\phi := \sigma(G)$ and $\sigma := \sigma(K)$, so $|V_1| = \phi n$ and by~(ii) we have $|V_j| \geq (1-\phi)n/(k-1) - \beta n$ for any $2 \leq j \leq k$. Without loss of generality we may assume that $b_1 \leq b_2, \dots, b_k$.
Define $x := \frac{\phi - \sigma}{1/k - \sigma}$, so $\alpha \leq x \leq 1$ by~(i). Then
\begin{equation}\label{completeeq2}
\phi = \frac{x}{k} + \sigma(1-x).
\end{equation}
Define $N := \lfloor (1-x)n/k!bd\rfloor$, and choose and delete a set of $kdN$ pairwise vertex-disjoint copies of $\Lo(K)$ in $G$, such that the smallest vertex class of each copy is contained in $V_1$. These copies thus cover $k!b\sigma Nd$ vertices of $V_1$ and $(k-2)!b(1-\sigma)kdN$ vertices of $V_i$ for each $2 \leq i \leq k$; in particular, the number of vertices covered in any vertex class $V_j$ is divisible by $kbd$. For each $i \in [k]$ let $X_i$ consist of the undeleted vertices of $V_i$, and let $X := X_1 \cup \dots \cup X_k$. So $|X_i|$ is divisible by $kbd$ for any $i \in [k]$ by (iii).
Also, by choice of $N$ we have 
\begin{align*} 
xn & \leq |X| = n - k!bdN \leq xn + k!bd,\\
|X_1| &= |V_1| - k!b\sigma dN \geq \phi n - (1-x)\sigma n = xn/k,\mbox{ and}\\
|X_j| &= |V_j| - (k-2)!b(1-\sigma) kdN \geq \frac{(1-\phi)n}{k-1} - \beta n - \frac{(1-x)(1-\sigma)n}{k-1} = \frac{xn}{k}-\beta n,
\end{align*}
for any $2 \leq j \leq k$, where the final equality in each of the last two lines holds by~\eqref{completeeq2}. Together these inequalities imply that $|X_j| \geq |X|/k - 2 \beta n$ for each $j \in [k]$.
Furthermore, together with $x \geq \alpha$ the first inequality shows that we may assume that $1/|X| \ll 2\beta \ll 1/b, 1/k$. So $G[X]$ meets the conditions of Lemma~\ref{completepackingnearlybalanced}, and so contains a perfect $K$-packing. Added to the deleted copies of $\Lo(K)$ this gives a perfect $\{K, \Lo(K)\}$-packing of $G$; by Propositions~\ref{fpacktohpack} and~\ref{hpackings} it follows that $G$ has a perfect $K$-packing.
\end{proof}

\section{Proofs} \label{sec:proofs}

We have now established all of the preliminary results and definitions we need for the proofs of Lemma~\ref{main1simple} and Lemma~\ref{main2simple}, for which we proceed as outlined in Section~\ref{sec:outline}.

\subsection{Proof of Lemma~\ref{main1simple}} \label{sec:proof1}

Recall the statement of Lemma~\ref{main1simple}.

\medskip \noindent {\bf Lemma~\ref{main1simple}.}
{\it Let $K$ be the complete $k$-partite $k$-graph whose vertex classes have sizes $b_1, \dots, b_k$, where these sizes are not all equal, and suppose that $\gcd(K)$ and $b_1$ are coprime. Then for any $\alpha > 0$ there exists $n_0 = n_0(K, \alpha)$ such that the following statement holds. Let $H$ be a $k$-graph on $n \geq n_0$ vertices such that
\begin{enumerate}
\item[(a)] $b := b_1 + \dots + b_k$ divides $n$,
\item[(b)] $\delta(H) \geq \sigma(K)n + \alpha n$, and
\item[(c)] if $\gcd(K) > 1$, then $\delta(H) \geq n/p^* + \alpha n$, where $p^*$ is the smallest prime factor of $\gcd(K)$.
\end{enumerate}
Then $H$ contains a perfect $K$-packing.}\medskip
 
\begin{proof}
If $\gcd(K) = 1$, define $p^* := b$; this is merely for notational convenience in handling the cases $\gcd(K) = 1$ and $\gcd(K) > 1$ simultaneously. Since $\sigma(K) \geq 1/b$ we may then assume that $\delta(H) \geq n/p^* + \alpha n$ in all cases. Furthermore, in any case we must have $p^* \leq b$.
Introduce new constants with
\begin{align*}
1/n &\ll 1/N \ll \eps  \ll d^* \ll 1/a \ll \xi, 1/r \ll \nu \ll \mu \\
&\ll c, \eta \ll \theta \ll \psi \ll \beta \ll 1/q \ll \alpha \ll 1/s, 1/D\ll 1/b, 1/k,
\end{align*}
and such that $s$ is divisible by $k \gcd(K)$, and choose an integer $p$ such that 
$$\sigma(K) +\alpha/3 \leq p/q \leq \sigma(K) + \alpha/2.$$ 
Note that our constant hierarchy assumes that $\alpha$ is sufficiently smaller that $1/s, 1/D, 1/b, 1/k$; this is not a problem since each property involving $\alpha$ in the statement of Lemma~\ref{main1simple} is monotone. Using this and the fact that the vertex class sizes $b_1, \dots, b_k$ are not all equal, we may assume that $\sigma(K) \leq 1/k - \alpha/2$, so $pk \leq q$. Finally, we may also assume that $a!r$ divides $n$. Indeed, by Theorem~\ref{turandensityzero} we may greedily delete a $K$-packing in $H$ of size up to $a!r$ so that the number of vertices remaining in $H$ is divisible by $a!r$, following which the subgraph induced by the remaining vertices of $H$ satisfies the conditions of the lemma (with weaker constants). So to prove the lemma it is sufficient to consider only the case where $a!r$ divides $n$. We assume this, so in fact no vertices were deleted.

\medskip \nib{Apply the Regular Approximation Lemma:} Let $U := V(H)$, and choose arbitrarily a partition $\Qart$ of $U$ into $r$ parts $T_1, \dots, T_r$ of equal size. Let $H'$ be the $k$-graph on $U$ consisting of all $\Qart$-partite edges of $H$. So $H'$ is a $\Qart$-partite $k$-graph on $U$ whose order is divisible by $a!r$, and so we may apply the Regular Approximation Lemma (Theorem~\ref{eq-partition}),
which yields
an $a$-bounded $\eps$-regular vertex-equitable partition $(k-1)$-complex $\Part$ on $U$, and 
a $\Qart$-partite $k$-graph $G$ on $U$ such that $G$ is $\xi$-close to $H'$ and the partition $k$-complex $G[\hat{\Part}]$ is $\eps$-regular. 
Let $Z = G \triangle H'$, so $Z$ is a $\Qart$-partite $k$-graph on $U$ with $|Z| \leq \xi n^k$, and we have $G \sm Z \subseteq H' \subseteq G \cup Z$. In particular, any edge of $G \sm Z$ is also an edge of $H$. Also let $U_1, \dots, U_m$ be the clusters of~$\Part$, and note that the partition $\Part^{(1)}$ of $U$ into clusters refines $\Qart$. Since $\Part$ is vertex-equitable, every cluster of $\Part$ must have the same size, so the number of clusters $m$ is divisible by $r$. Also, since $\Part$ is $a$-bounded we have $r \leq m \leq ar$. Define $n_1 := |U_i| = n/m$ to be the common cluster size. Then, since $a!r$ divides $n$ and $m \leq ar$ is divisible by $r$, we deduce that $n_1$ is divisible by all integers up to $a/2$; in particular, $n_1$ is certainly divisible by $(q-p)p(k-1)$.

Observe that the vertex set $U$, the partition $\Qart$ of $U$ into parts $T_1, \dots, T_r$, the partition $(k-1)$-complex $\Part$ with clusters $U_1, \dots, U_m$, and the $\Qart$-partite $k$-graphs $G$ and $Z$ therefore satisfy the conditions of Setup~\ref{redsetup}, with $U$ and the clusters $U_1, \dots, U_m$ in place of $X$ and $X_1, \dots, X_m$, and with constants $1/n \ll \eps \ll d^* \ll 1/a \ll \xi, 1/r \ll \nu \ll \mu \ll c, \eta \ll \theta \ll 1/D \ll 1/k$ playing identical roles there as here. Under this setup, let $\R$ be the reduced $k$-graph of $G$ and $Z$ as defined in Definition~\ref{redgraphdef}.
So $\R$ has vertex set $[m]$, where vertex $i$ corresponds to the cluster~$U_i$. Since $H'$ contains all $\Qart$-partite edges of $H$, and $H' \subseteq G \cup Z$, for any $\Qart$-partite $(k-1)$-tuple $e \in \binom{U}{k-1}$ we have 
$\deg_{G \cup Z}(e) \geq \deg_{H'}(e) \geq \delta(H) - (k-1)n/r \geq \sigma(K) n+2\alpha n/3$. 
So by Lemma~\ref{redgraphmindeg}, applied with $\sigma(K) + 2\alpha/3$ in place of $\gamma$,
\begin{enumerate}[(i)]
\newcounter{mem}
\setcounter{enumi}{\value{mem}} 
\item \label{itema} all but at most $\theta m^{k-1}$ many $(k-1)$-tuples $S \in \binom{[m]}{k-1}$ have 
$$\deg_\R(S) \geq \sigma(K) m + 2\alpha m/3 -\theta m \geq pm/q + \theta m.$$
\setcounter{mem}{\value{enumi}}
\end{enumerate}

\medskip \nib{Refine the regularity partition into `lopsided' groups:}
By (\ref{itema}) we may apply Lemma~\ref{akpqpacking} to obtain an $\Akpq$-packing $\F$ in $\R$ so that
\begin{enumerate}[(i)]
\setcounter{enumi}{\value{mem}} 
\item \label{pfcover} $\F$ covers at least $(1-\psi)m$ vertices of $\R$, and
\item \label{pfRconn} $\R[V(\F)]$ is connected,
\setcounter{mem}{\value{enumi}}
\end{enumerate}
where $V(\F)$ denotes the set of vertices of $\R$ covered by $\F$. Define a graph (\emph{i.e.} 2-graph) $\Sa$ with vertex set $V(\F)$, where $ij$ is an edge of $\Sa$ if there exists a $(k-1)$-tuple $S \in \binom{V(\F)}{k-1}$ for which the triple $(i, S, j)$ is $\Phi$-dense and $Z$-sparse (as defined in Definition~\ref{redgraphdef}). Then every $i \in V(\Sa)$ must lie in some edge $S$ of $\R$ (since $i$ is covered by $\F$). Since for any $\Qart$-partite $(k-1)$-tuple $e \in \binom{U}{k-1}$ we have 
$\deg_{G \cup Z}(e) \geq \deg_{H'}(e) \geq \delta(H) - kn/r \geq n/p^* + 2\alpha n/3$,
Lemma~\ref{redgraphmindeg} (applied with $1/p^*+2\alpha/3$ in place of $\gamma$) then implies that there are at least $m/p^* + \alpha m/2$ choices of $j \in [m] \sm S$ such that the triple $(i, S \sm \{i\}, j)$ is $\Phi$-dense and $Z$-sparse. At most $\psi m$ of these choices of $j$ do not lie in $V(\Sa)$, so we have $\deg_\Sa(i) \geq m/p^* + \alpha m/2 - \psi m > m/p^*$. So 
$$\delta(\Sa) > m/p^* \geq |V(\Sa)|/p^*,$$
 from which we conclude that each of the connected components $C_1, \dots, C_{s^*}$ of $\Sa$ contains more than $m/p^*$ vertices. In particular, for any $x_1, \dots, x_{k-1} \in [s^*]$, there are at least $\binom{m/p^*}{k-1} > \theta m^{k-1}$ many $(k-1)$-tuples $\{u_1, \dots, u_{k-1}\}$ with $u_j \in V(C_{x_j})$ for each $j \in [k-1]$, so by (\ref{itema}) at least one of these $(k-1)$-tuples must have degree at least $pm/q > \psi m$ in $\R$. This proves that 
\begin{enumerate}[(i)]
\setcounter{enumi}{\value{mem}} 
\item \label{itemd} $\Sa$ has fewer than $p^*$ connected components $C_1, \dots, C_{s^*}$, and for any $x_1, \dots, x_{k-1} \in [s^*]$ there is some edge $\{u_1, \dots, u_k\} \in \R[V(\Sa)]$ such that $u_j \in V(C_{x_j})$ for each $j \in [k-1]$.
\setcounter{mem}{\value{enumi}}
\end{enumerate}

Now, fix any $F \in \F$, and let $U_{F} := \bigcup_{j \in V(F)} U_j$. Then since each cluster $U_j$ has size $n_1$, which is divisible by $(q-p)p(k-1)$, by Lemma~\ref{akpqsplit} we may partition $U_F$ into disjoint sets $V_j^i$ with $j \in [k]$ and $i \in [(q-p)p(k-1)]$ such that
\begin{enumerate}[(i)]
\setcounter{enumi}{\value{mem}} 
\item \label{pfveq} $|V_1^i| = \frac{p}{q}\sum_{j \in [k]} |V_j^i|$ for each $i$, 
\item \label{pfvlb} $\frac{n_1}{q} \leq \frac{n_1}{q-p} = |V_1^i| \leq |V_2^i| = |V_3^i| = \dots = |V_k^i|$ for each $i$, and
\item \label{pfedgeR} for each $i$ and $j$ there exists $f(i,j) \in [m]$ for which $V_j^i \subseteq U_{f(i,j)}$ and such that for any fixed $i$ the set $\{f(i,j) : j \in [k]\}$ is an edge of $\R$.
\setcounter{mem}{\value{enumi}}
\end{enumerate}
Partition $U_F$ in this manner for every $F \in \F$ to obtain sets $V_j^i$ for $j \in [k]$ and $i \in [t]$, where $t := (q-p)p(k-1)|\F|$. We will refer to the sets $V_j^i$ as \emph{subclusters}. 

We naturally obtain from $\R$ a $k$-graph $\R'$ corresponding to our refined partition into subclusters. Indeed, this has vertex set $[t] \times [k]$, where the vertex $(i, j)$ corresponds to the subcluster $V^i_j$, and a set $\{(i_1, j_1), \dots, (i_k, j_k)\}$ is an edge of $\R'$ if and only if $\{f(i_1, j_1), \dots, f(i_k, j_k)\}$ is an edge of $\R$. That is, edges of $\R'$ correspond to $k$-tuples of subclusters which were taken from clusters of the same edge of the reduced $k$-graph. In the same way we define a graph $\Sa'$ on the vertex set $[t] \times [k]$, where $\{(i_1, j_1), (i_2, j_2)\}$ is an edge of $\Sa'$ if and only if $\{f(i_1, j_1), f(i_2, j_2)\}$ was an edge of $\Sa$. It follows from this definition that the components of $\Sa'$ correspond to the components of $\Sa$. That is, $\Sa'$ has components $C'_1, \dots, C'_{s^*}$, where for any $\ell \in [s^*]$ we have $(i, j) \in V(C'_\ell)$ if and only if $f(i, j) \in V(C_\ell)$. It then follows from (\ref{itemd}) that 
\begin{enumerate}[(i)]
\setcounter{enumi}{\value{mem}} 
\item \label{pfedges} $\Sa'$ has fewer than $p^*$ connected components $C'_1, \dots, C'_{s^*}$, and for any $x_1, \dots, x_{k-1} \in [s^*]$ there is some edge $e = \{(i_1, j_1), \dots, (i_k, j_k)\} \in \R'$ such that $(i_\ell, j_\ell) \in V(C'_{x_\ell})$ for each $\ell \in [k-1]$. 
Also, 
\item \label{pfconn} $\R'$ is connected. Indeed, if $f(i, j) = \ell$ and $f(i', j') = \ell'$, and $\ell$ and $\ell'$ were contained in a common edge of $\R$, then $(i, j)$ and $(i', j')$ are contained in a common edge of $\R'$ by definition of $\R'$. Since $\R[V(\F)]$ is connected by (\ref{pfRconn}), this implies that $\R'$ is connected.
\setcounter{mem}{\value{enumi}}
\end{enumerate}

\medskip \nib{Obtain robustly universal complexes:} 
Fix any $i \in [t]$. Then $e := \{f(i, j) : j \in [k]\}$ is an edge of $\R$ by (\ref{pfedgeR}), and $|V^i_j| \geq n_1/q \geq \eta n_1$ for each $j \in [k]$ by (\ref{pfvlb}). So we may apply Lemma~\ref{getrobuni} with the clusters $U_{f(i, j)}$ and subclusters $V^i_j$ in place of the clusters $X_i$ and subsets $Y_i$ respectively. This allows us to delete at most $\mu |V^i_j|$ vertices from each subcluster~$V^i_j$ to obtain subsets $W^i_j \subseteq V^i_j$ and a $k$-partite $k$-complex $J^i$ with vertex classes $W^i_1, \dots, W^i_k$ such that 
\begin{enumerate}[(i)]
\setcounter{enumi}{\value{mem}} 
\item \label{pfdens} $d(J^i_{[k]}) > d^*$ and $|J^i_=(v)| > d^*|J^i_=|/|W^i_j|$ for every $v \in W^i_j$, 
\item \label{pfrobuni} $J^i_= \subseteq G \sm Z$, and $J^i$ is $\eta$-robustly $D$-universal.
\setcounter{mem}{\value{enumi}}
\end{enumerate} 
Let $W_0$ be the set of all vertices of $H$ which do not lie in any set $W^i_j$. By (\ref{pfcover}) there are at most $\psi n$ vertices which lie in clusters $U_\ell$ for $\ell \in [m] \sm V(\F)$; the set $W_0$ contains all these, and also the at most $\mu n$ vertices deleted whilst forming the sets $W^i_j$. So $|W_0| \leq \psi n + \mu n \leq 2\psi n$. 
Next, for each $i$ and $j$ choose an integer $s^i_j$ such that $bk\gcd(K)$ divides $s^i_j$ and
$$\beta n_1 \leq s^i_j \leq 2\beta n_1,$$
and choose a subset $X^i_j \subseteq W^i_j$ of size precisely $s^i_j$ uniformly at random and independently of each other choice. Also let $Y^i_j = W^i_j \sm X^i_j$ for every $i$ and $j$, and for each $i \in [t]$ write $X^i := \bigcup_{j \in [k]} X^i_j$, $Y^i := \bigcup_{j \in [k]} Y^i_j$, $X = \bigcup_{i \in [t]} X^i$ and $Y = \bigcup_{i \in [t]} Y^i$. 
So the sets $X, Y$ and $W_0$ partition $V(H)$. Then we may fix an outcome of these random selections so that
\begin{enumerate}[(i)]
\setcounter{enumi}{\value{mem}}
\item \label{pfuni} for any $i \in [t]$ and any subset $Y''{}^i \subseteq Y^i$, the subcomplex $J^i[X^i \cup Y''{}^i]$ is $D$-universal, and
\item \label{pfdeg} $\delta(H[W_0 \cup Y]) \geq \delta(H) - |X| \geq \alpha n - \alpha n/2 = \alpha n/2$. 
\setcounter{mem}{\value{enumi}} 
\end{enumerate} 
Indeed, (\ref{pfdeg}) follows from the fact that $|X| \leq tk (2\beta n_1) \leq \alpha n/2$, whilst for any specific $i \in [t]$~(\ref{pfuni}) holds with probability $1-o(1)$ by Proposition~\ref{randomsplitkeepsmatching2}. Since the selections for distinct $i \in [t]$ are independent, with positive probability~(\ref{pfuni}) holds for every $i \in [t]$.

\medskip \nib{Delete a $K$-packing covering `bad' vertices:}
We will next greedily form a $K$-packing $M$ in $H[W_0 \cup Y]$ of size $|W_0|$ which covers every vertex of $W_0$ and which covers at most $\sqrt{\psi} n_1$ vertices from any set $Y^i$. To do this, suppose that we have already chosen fewer than $|W_0|$ members of $M$, and that $v \in W_0$ is not yet covered by $M$. Then $M$ covers fewer than $b|W_0| \leq 2b \psi n$ vertices of $H$, so there are at most $2b \psi n / (\sqrt{\psi} n_1/2) \leq 4b \sqrt{\psi} m$ sets $Y^i$ for which more than $\sqrt{\psi} n_1/2$ vertices of $Y^i$ are covered by $M$. Let $Y' \subseteq Y$ consist of all vertices not yet covered by $M$ which lie in sets $Y^i$ in which at most $\sqrt{\psi} n_1/2$ vertices of $Y^i$ are covered by~$M$. Then 
$$|Y'| \geq |Y| - (4b \sqrt{\psi} m) kn_1 - 2b \psi n \geq |Y| - \frac{\alpha n}{4},$$
so by (\ref{pfdeg}) we have 
$$\delta(H[\{v\} \cup Y']) \geq \frac{\alpha n}{2} - |(Y \cup W_0) \sm Y'| \geq \frac{\alpha n}{2} - \frac{\alpha n}{4} - 2 \psi n \geq \frac{\alpha n}{5}.$$
So by Lemma~\ref{incorporateexcep} there is a copy of $K$ in $H[\{v\} \cup Y']$ which contains $v$; choose such a copy and add it to $M$. Proceeding greedily in this manner, after $|W_0|$ steps we obtain a $K$-packing $M$ in $H[W_0 \cup Y]$ which covers every vertex of $|W_0|$, and which covers at most $\sqrt{\psi} n_1/2 + b \leq \sqrt{\psi} n_1$ vertices in any set $Y^i$. Write ${Y'}^i_j := Y^i_j \sm V(M)$ for each $i \in [t], j \in [k]$ and let $Y' := \bigcup_{i, j} Y'{}^i_j$.

\medskip \nib{Delete a $K$-packing to ensure divisibility of cluster sizes:} 
We now delete a further $K$-packing in $H[Y']$, so that after these deletions the number of vertices remaining in any set $Y^i_j$ is divisible by $bk \gcd(K)$. To do this we apply Lemma~\ref{gcdbalance} with $\R'$, $\Sa'$, $H'[Y']$ and $kt$ in place of $\R$, $\Sa$, $G$ and $t$ respectively, and the sets ${Y'}^i_j$ for $i \in [k]$ and $j \in [t]$ in place of the sets $Y_j$ of Lemma~\ref{gcdbalance}. (We use the subgraph $H'[Y']$ rather than simply $H[Y']$ in place of $G$ due to the requirement in  Lemma~\ref{gcdbalance} that $G$ should be $t$-partite.)  To see that $b$ divides $|Y'|$, note that $|X|$ is divisible by $b$ by choice of the integers $s^i_j$, that $|V(M)|$ is divisible by $b$ since $M$ is a $K$-packing, and $|V(H)|$ is divisible by $b$ by assumption; since $X$ and $V(M)$ are disjoint and $Y' = V(H) \sm (V(M) \cup X)$ we indeed have that $b$ divides $|Y'|$. Furthermore, $\R'$ is connected by (\ref{pfconn}), and condition (iii) of Lemma~\ref{gcdbalance} holds by (\ref{pfedges}) (since if $\gcd(K) > 1$ then $p^*$ is the smallest prime factor of $\gcd(K)$). So it remains to verify that conditions (i) and (ii) of Lemma~\ref{gcdbalance} are satisfied.

For condition (i), consider any edge $e \in \R'$, and let $(i_x, j_x)$ for $x \in [k]$ be the vertices of~$e$. Also let $v_x := f(i_x, j_x)$ for each $x \in [k]$; then $\{v_1, \dots, v_k\}$ is an edge of $\R[V_\F]$ by definition of $\R'$. We apply Lemma~\ref{getrobuni} with clusters $U_{v_x}$ and subclusters $Y'{}^{i_x}_{j_x}$ in place of the sets $X_\ell$ and $Y_\ell$ respectively, which is possible since $|Y'{}^{i_x}_{j_x}| \geq n_1/q - \sqrt{\psi}n_1 - 2\beta n_1 \geq \eta n_1$ for each pair $(i_x, j_x)$. Lemma~\ref{getrobuni} then yields a $(k+1)$-partite $k$-complex $J$ which covers at least $(1-\mu)|Y'{}^{i_x}_{j_x}| \geq n_1/2q$ vertices of each vertex class $Y'{}^{i_x}_{j_x}$ such that $J_= \subseteq G \sm Z \subseteq H'$ and such that $J$ is $D$-universal. By a very similar argument to the proof of Proposition~\ref{varykcomp}, the latter fact implies that $J_=$ contains at least $\lfloor n_1/2qb(s+1) \rfloor > N$ vertex-disjoint copies of $\B(b(s+1))$, and so $H'[\bigcup_{(i_x, j_x) \in e} Y'{}^{i_x}_{j_x}]$ does so also.

Now let $(i_1, j_1)(i_{k+1}, j_{k+1})$ be an edge of $\Sa'$, and write $v_1 = f(i_1, j_1)$ and $v_{k+1} = f(i_{k+1}, j_{k+1})$. Then $v_1v_{k+1}$ is an edge of $\Sa$ by definition of $\Sa'$, and so by definition of $\Sa$ there exists a $(k-1)$-tuple $S \in \binom{V(\Sa) \sm \{v_1, v_{k+1}\}}{k-1}$ such that $(u, S, v)$ is a $\Phi$-dense and $Z$-sparse triple. Let $v_2, \dots, v_k$ be the vertices of $S$, and for each $2 \leq x \leq k$ choose some $i_x$ and $j_x$ such that $f(i_x, j_x) = v_x$. We apply Lemma~\ref{getphirobuni} with $\{v_1, \dots, v_k\}$ and $\{v_2, \dots, v_{k+1}\}$ in place of $A$ and $B$ respectively, with the clusters $U_{v_x}$ and subclusters $Y'{}^{i_x}_{j_x}$ for $x \in [k+1]$ in place of the sets $X_\ell$ and $Y_\ell$ respectively. Similarly as before, this is possible since $|Y'{}^{i_x}_{j_x}| \geq \eta n_1$ for each pair $(i_x, j_x)$, and then Lemma~\ref{getphirobuni} then yields a $(k+1)$-partite $k$-complex $J$ which covers at least $n_1/2q$ vertices of each vertex class $Y'{}^{i_x}_{j_x}$ such that $J_= \subseteq G \sm Z \subseteq H'$ and such that $J$ is $D$-universal on $\Phi$ (where we consider vertex $x$ of $\Phi$ to correspond to the pair $(i_x, j_x)$). Then by Proposition~\ref{varykcomp}, $J_=$ contains $\lfloor (n_1/2q)/kbs \rfloor > N$ vertex-disjoint copies of $\Phi(b(s+1))$ whose end vertex classes lie in $Y'{}^{i_x}_{j_x}$ for $x = 1$ and $x = k+1$ and whose central vertex classes lie in $Y'{}^{i_x}_{j_x}$ for $2 \leq x \leq k$. Since $J_= \subseteq H'$, this establishes condition (ii) of Lemma~\ref{gcdbalance}.

So we may indeed apply Lemma~\ref{gcdbalance} as claimed. This yields a $K$-packing $M'$ in $H'[Y']$ (and therefore also in $H[Y']$) of size at most $N/b$ so that, taking $Y''{}^i_j := Y'{}^i_j \sm V(M')$ for each $i$ and $j$, we have that $bk \gcd(K)$ divides $|Y''{}^i_j|$ for every $i \in [t]$ and $j \in [k]$. 
 
\medskip \nib{Blow-up a perfect $K$-packing in the remaining $k$-graph:}
To finish the proof, let $L^i_j := X^i_j \cup Y''{}^i_j$ for every $i \in [t]$ and $j \in [k]$, and for each~$i \in [t]$ let $L^i := \bigcup_{j \in [k]} L^i_j$. So the sets $L^i$ partition $V(H) \sm V(M \cup M')$.
Fix any $i \in [t]$, and observe that the $k$-complex $J^i[L^i]$ is $D$-universal by (\ref{pfuni}). Since $J^i_= \subseteq G \sm Z \subseteq H$, this implies that $H[L^i]$ contains a perfect $K$-packing if $\K[L^i]$, the complete $k$-partite $k$-graph on vertex classes $L^i_1, \dots, L^i_k$, does also. For any $j \in [k]$, by choice of the sets $X^i_j$ and the $K$-packing $M'$ both $|X^i_j|$ and $|Y''{}^i_j|$ are divisible by $bk\gcd(K)$, so $bk\gcd(K)$ divides $|L^i_j|$ also. Furthermore, $L^i_j$ was formed from $V^i_j$ by first deleting at most $\mu |V^i_j|$ vertices to form $W^i_j$, and then deleting the at most $\sqrt{\psi} n_1 + N$ vertices covered by $M \cup M'$. Since $|V^i_j| \geq n_1/q$ by (\ref{pfvlb}), in total at most $\beta |V^i_j|$ vertices were deleted in forming $L^i_j$ from $V^i_j$, and in particular we have $|L^i_j| \geq |V^i_j|/2$. Writing $V^i := \bigcup_{j \in [k]} V^i_j$, by~(\ref{pfveq}),~(\ref{pfvlb}) and our choice of $p$, it follows that for any $j \in [k]$ we have 
$$|L^i_j| \geq (p/q)|V^i| - \beta |V^i_j| \geq (p/q)|L^i| - 2 \beta |L^i_j| \geq (\sigma(K) + \alpha/4) |L^i|,$$ 
so $\sigma(\K[L^i]) \geq \sigma(K) + \alpha/4$, and also that 
$||L^i_j| - |L^i_{j'}|| \leq \beta |V^i| \leq 2 \beta |L^i|$ for any $2 \leq j, j' \leq k$.
So by Corollary~\ref{completepacking} $\K[L^i]$ admits a perfect $K$-packing, and so $H[L^i]$  contains a perfect $K$-packing~$M^i$. Having chosen such a $K$-packing $M^i$ for every $i \in [t]$, the union $M \cup M' \cup \bigcup_{i \in [t]} M^i$ is a perfect $K$-packing in $H$. \end{proof}

\subsection{Proof of Lemma~\ref{main2simple}} \label{sec:proof2}
\medskip

Recall the statement of Lemma~\ref{main2simple}.

\medskip \noindent {\bf Lemma~\ref{main2simple}.}
{\it Let $K$ be the complete $k$-partite $k$-graph whose vertex classes each have size $b_1$. Then for any $\alpha > 0$ there exists $n_0 = n_0(K, \alpha)$ such that if $n \geq n_0$ is divisible by $b_1k$ and $H$ is a $k$-graph on $n$ vertices with $\delta(H) \geq n/2 + \alpha n$ then $H$ contains a perfect $K$-packing.} \medskip

In several places the proof of this lemma is similar to the proof of Lemma~\ref{main1simple}, in which case we refer to that proof.\medskip
 
\begin{proof}
Let $b := kb_1$, so $|V(K)| = b$, and introduce new constants with 
\begin{align*}
1/n & \ll \eps  \ll d^* \ll 1/a \ll \xi, 1/r \ll \nu \ll \mu \ll c, \eta \ll \theta  \ll \psi \ll 1/C, 1/L \ll \alpha, 1/D \ll 1/b, 1/k,
\end{align*}
As in the proof of Lemma~\ref{main1simple}, we have assumed without loss of generality that $\alpha$ is sufficiently smaller than $1/b$ and $1/k$. Furthermore, by the same argument as in the proof of Lemma~\ref{main1simple} we may assume that $n$ is divisible by $a!r$. 

We begin by following the exact same steps as in the section `Apply the Regular Approximation Lemma' of the previous proof, to obtain a partition $\Qart$ of $U := V(H)$ into parts $T_1, \dots, T_r$, a subgraph $H' \subseteq H$ consisting of all $\Qart$-partite edges of $H$, a partition $(k-1)$-complex $\Part$ on $U$ with clusters $U_1, \dots, U_m$, and $\Qart$-partite $k$-graphs $G$ and $Z = G \triangle H'$ which satisfy the conditions of Setup~\ref{redsetup} (with variables taking the same values there as here). Likewise, as before we let $\R$ be the reduced $k$-graph of $G$ and $Z$ as defined in Definition~\ref{redgraphdef}, so $\R$ has vertex set $[m]$. Similarly as before we find that any $\Qart$-partite $(k-1)$-tuple $e \in \binom{U}{k-1}$ has 
\begin{equation} \label{eq:GZdeg}
\deg_{G \cup Z}(e) \geq \deg_{H}(e) - (k-1)n/r \geq n/2 + 2\alpha n/3,\end{equation} 
and we apply Lemma~\ref{redgraphmindeg} to find that (with plenty of room to spare) all but at most $\theta m^{k-1}$ many $(k-1)$-tuples $S \in \binom{[m]}{k-1}$ have $\deg_\R(S) \geq m/k + \alpha m/2$. So we can apply Corollary~\ref{almostmatching} (with $\alpha/2$ in place of $\alpha$) to find a matching $M_\R$ in $\R$ which covers $m' \geq (1-\psi)m$ vertices of $\R$ such that $\R[V(M'_\R)]$ is $(C, L)$-irreducible on $M'_\R$ for any $M'_\R \subseteq M_\R$ with $|M'_\R| \geq (1-\alpha/4)|M_\R|$. Without loss of generality we assume that $V(M_\R) = [m']$.
 
Having chosen $M_\R$, we now define the graph $\Sa$ on $V(M_\R) = [m']$; this definition is different to that used in the proof of Lemma~\ref{main1simple}. Indeed, we first define a directed graph $\Sa^+$ on $[m']$, where $i \to j$ is an edge of $\Sa$ if the triple $(i, e_i \sm \{i\}, j)$ is $\Phi$-dense and $Z$-sparse, where $e_i$ is the edge of $M_\R$ which contains $i$. We then define $\Sa$ to be the base graph of $\Sa^+$. By~\eqref{eq:GZdeg} and Lemma~\ref{redgraphmindeg}, applied with $1/2 + 2\alpha/3$ in place of~$\gamma$, we find that for any $i \in [m']$ there are at least $m/2 + 2\alpha m/3 - \theta m$ choices of $j \in [m] \sm e_i$ such that the triple $(i, e_i \sm \{i\}, j)$ is $\Phi$-dense and $Z$-sparse. At most $\psi m$ of these choices of $j$ are not members of $[m']$, so we have $\delta(\Sa) \geq \delta^+(\Sa^+) > m/2 + \alpha m/2 \geq (1/2 + \alpha/2)m'$. 

\medskip \nib{Obtain robustly universal complexes:} 
We now obtain robustly universal complexes covering almost all of the vertices in clusters corresponding to edges of $\R$, similarly as in the proof of Lemma~\ref{main1simple} (but here we do not divide our clusters into subclusters, so there is no need to define $\R'$ and $\Sa'$). Fix any $e \in M_\R$. Then by Lemma~\ref{getrobuni} we can delete at most $\mu n_1$ vertices from each set $U_j$ with $j \in e$ to obtain subsets $W_j$ and a $k$-partite $k$-complex $J^e$ with vertex classes $W_1, \dots, W_k$ such that $J^e$ is $\eta$-robustly $D$-universal, $J^e_= \subseteq G \sm Z$, $d(J_e^e) > d^*$ and $|J^e_e(v)| > d^*|J^e_=|/|W_j|$ for every $v \in W_j$. 

Let $W_0$ be the set of all vertices of $H$ which do not lie in any set $W_j$. So $W_0$ contains the at most $\psi n$ vertices in clusters $U_\ell$ for $\ell \in [m] \sm [m']$, and the at most $\mu n$ vertices deleted from clusters $U_j$ in forming the sets $W_j$. So we have $|W_0| \leq \psi n + \mu n \leq 2\psi n$. Next, fix an integer $n_X$ with $\alpha n_1/3 \leq n_X \leq \alpha n_1/2$ such that $n_X$ is divisible by $b$. For each $j \in [m']$ choose a subset $X_j \subseteq W_j$ of size precisely $n_X$ uniformly at random, and take $Y_j := W_j \sm X_j$. Then just as in the proof of Lemma~\ref{main1simple}, we find that, as there, we can fix an outcome of these random selections so that 
\begin{enumerate}[(i)]
\item \label{pf2deg} $\delta(H[W_0 \cup Y]) \geq \alpha n/2$, and
\item \label{pf2uni} for any $e \in M_\R$ and any subset $Y''_e \subseteq Y_e$, the subcomplex $J^e[X_e \cup Y''_e]$ is $D$-universal,
\setcounter{mem}{\value{enumi}}
\end{enumerate}
where we write $Y_e := \bigcup_{j \in e} Y_j, X_e := \bigcup_{j \in e} X_j$, and $Y := \bigcup_{e \in M_\R} Y_e.$ Also let $X := \bigcup_{e \in M_\R} X_e$, and observe that the sets $W$, $X$ and $Y$ partition $V(H)$. So our assumption that $b$ divides $|V(H)|$, together with our choice of $n_X$, implies that $b$ divides both $|X|$ and $|W_0 \cup Y|$. 

\medskip \nib{Delete a $K$-packing covering `bad' vertices:}
Exactly as in the corresponding part of the proof of Lemma~\ref{main1simple}, (\ref{pf2deg}) allows us to repeatedly apply Lemma~\ref{incorporateexcep} to greedily form a $K$-packing $M^1$ in $H[W_0 \cup Y]$ of size $|W_0|$ which covers every vertex of $W_0$ and which covers at most $\sqrt{\psi} n_1$ vertices from $Y_e$ for any $e \in M_\R$. Following this, we apply Theorem~\ref{turandensityzero} up to $m'$ times to greedily choose a $K$-packing $M^2$ in $H[Y \sm V(M^1)]$ of size at most $m'$, so that $bm'$ divides $|Y \sm V(M^1 \cup M^2)|$ (this is possible since $Y \sm V(M^1)$ is identical to $(Y \cup W_0) \sm V(M^1)$, so has size divisible by $b$). Write $Y'_j := Y_j \sm V(M^1 \cup M^2)$ for every $j \in [m']$, and let $Y' := Y \sm V(M^1 \cup M^2)$. Note that $Y'_j$ was formed from $U_j$ by deleting at most $\mu n_1$ vertices to form $W_j$, then exactly $n_X$ vertices to form $Y_j$, and then at most $\sqrt{\psi} n_1 + bm'$ vertices which were covered by $M^1 \cup M^2$. So, writing $n'_1 := n_1 - n_X$, for any $j \in [m]$ we have 
$$(1-3\sqrt{\psi}) n'_1 \leq n'_1 - 2\sqrt{\psi} n_1 \leq |Y'_j| \leq n'_1.$$

\medskip \nib{Delete a $K$-packing to ensure equality of cluster sizes:}
We now use Lemma~\ref{samesizebalance} to delete a further $K$-packing $M^3$ in $H[Y']$ so that, following these deletions, the clusters $Y'_\ell$ for $\ell \in e$ have equal size for any edge $e \in M_\R$, and this common size is divisible by $b_1$. 
For this, recall that $\R' := \R[V(M_\R)]$ is a $k$-graph with vertex set $[m']$, and that $M_\R$ is a perfect matching in $\R'$, $\Sa^+$ is a directed graph on $[m']$, and $\Sa$ is the base graph of $\Sa^+$. We will show that $K, \R', M_\R, \Sa, \Sa^+$ and the clusters $U_\ell$ for $\ell \in [m']$ satisfy the conditions of Setup~\ref{balancesetup} with $H', m', n_1$ and $kb_1L+b_1C$ in place of $G, m, n$ and $N$ respectively. So fix any sets $V_\ell \subseteq U_\ell$ with $|V_\ell| \geq n_1/2$ for each $\ell \in [m']$, and let $V := \bigcup_{\ell \in [m']} V_\ell$. 

Let $e \in \R'$. We apply Lemma~\ref{getrobuni} with clusters $U_\ell$ and subclusters $V_\ell$ in place of the sets $X_\ell$ and $Y_\ell$ respectively. Lemma~\ref{getrobuni} then yields a $(k+1)$-partite $k$-complex $J$ which covers at least $(1-\mu)|V_\ell| \geq n_1/3$ vertices of each vertex class $V_\ell$ such that $J_= \subseteq G \sm Z \subseteq H'$ and such that $J$ is $D$-universal. So $J_=$ contains at least $\lfloor n_1/3b_1 \rfloor > N$ vertex-disjoint copies of $K$, and so $H'[\bigcup_{\ell \in e} V_\ell]$ does so also. This demonstrates that condition (i) of Setup~\ref{balancesetup} is satisfied.

Now let $i \to j$ be an edge of $\Sa^+$. By definition of $\Sa^+$ it follows that $(i, e_i \sm \{i\}, j)$ is a $\Phi$-dense and $Z$-sparse triple. Define $A := e_i$ and $B := \{j\} \cup e_i \sm \{i\}$, and apply Lemma~\ref{getphirobuni} with the sets $V_\ell$ and clusters $U_\ell$ for $\ell \in A \cup B$ in place of the sets $Y_\ell$ and $X_\ell$ respectively. Then Lemma~\ref{getphirobuni} yields a $(k+1)$-partite $k$-complex $J$ with $J_= \subseteq G \sm Z \subseteq H'$ which covers at least $(1-\mu)|V_\ell| \geq n_1/3$ vertices of each $V_\ell$ and is $D$-universal on $\Phi$. So by Proposition~\ref{varykcomp} $J_=$ contains at least $\lfloor n_1/3b_1 \rfloor \geq N$ vertex-disjoint copies of $\Phi(b_1)$ in $J_= \subseteq H'[V]$ whose end vertex classes lie in $V_i$ and $V_j$ and whose central vertex classes lie in $V_\ell$ for $\ell \in e_i \sm \{i\}$, so condition (ii) of Setup~\ref{balancesetup} is satisfied.

We will apply Lemma~\ref{samesizebalance} with the trivial partition of $[m']$ into one set $X_1 = [m']$ (so $s=1$). So condition (iii) of Lemma~\ref{samesizebalance} requires simply that $|Y'|/m'$ is an integer which is divisible by $b_1$; this holds by our choice of $M^2$ and the fact that $b_1$ divides $b$. Furthermore, condition (ii) of Lemma~\ref{samesizebalance} holds by our choice of $M_\R$. Finally, condition (i) of Lemma~\ref{samesizebalance} holds since $\delta(\Sa) > (1/2 + \alpha/2) m'$, so for any $x, y \in [m']$ there are more than $\alpha m'$ paths of length two from $x$ to $y$, and so these paths cannot all be be removed by the deletion of at most $\alpha m'$ vertices of $\Sa$ not including $x$ or $y$. We can therefore apply Lemma~\ref{samesizebalance}, with $3\sqrt{\psi}$ and $n_1'$ in place of $\beta$ and $n'$, and with $b_1, \alpha, k, C$ and $L$ playing the same role there as here. 
From this we obtain a $K$-packing $M^3$ in $H'[Y']$ such that, writing $Y''_j := Y'_j \sm V(M^3)$ for each $j \in [m']$, for each $e \in M_\R$ there is an integer $n_e$ such that $b_1$ divides $n_e$ and $|Y''_j| = n_e$ for any $j \in e$.

\medskip \nib{Blow-up a perfect $K$-packing in the remaining $k$-graph:}
To finish the proof, let $L_j := X_j \cup Y''_j$ for every $j \in [m']$, and for each~$e \in M_\R$ let $L_e := \bigcup_{j \in e} L_j$. So the sets $L_e$ for $e \in M_\R$ partition $V(H) \sm V(M^1 \cup M^2 \cup M^3)$. Now fix any $e \in M_\R$. By choice of $M^3$ we have $|L_j| = |Y''_j| + |X_j| = n_e + n_X$ for any $j \in e$, and so $|L_j|$ is divisible by $b_1$ since both $n_e$ and $n_X$ were. So the complete $k$-partite $k$-graph $\K[L_e]$ with vertex classes $L_j$ for $j \in e$ admits a perfect $K$-packing by Fact~\ref{completepackingbalanced}. Since the $k$-complex $J^e[L_e]$ is $D$-universal by~(\ref{pf2uni}) it follows that $J^e_=[L_e]$ admits a perfect $K$-packing also. Finally, since $J^e_= \subseteq G \sm Z \subseteq H$, this implies that $H[L_e]$ contains a perfect $K$-packing $M^e$. Choose $M^e$ in this way for each $e \in M_\R$; then $M^1 \cup M^2 \cup M^3 \cup \bigcup_{e \in M_\R} M^e$ is a perfect $K$-packing in $H$.
\end{proof}

\section{Packing loose cycles}\label{sec:cycles}

Recall that we write $C^k_s$ to denote the loose cycle $k$-graph on $s(k-1)$ vertices, which was defined for any $s > 1$ to have $s(k-1)$ vertices $\{1, \dots, s(k-1)\}$ and $s$ edges 
$\{\{j(k-1)+1, \dots, j(k-1)+k\} \textrm{ for }0 \leq j < s\}$, with addition taken modulo $s(k-1)$. Also recall that $\tau(K)$ denotes the proportion of vertices of $K$ in a smallest vertex cover of $K$, whilst $\sigma(K)$ denotes the proportion of vertices of $K$ in a smallest vertex class of a $k$-partite realisation of~$K$. In this section we prove Theorem~\ref{cyclepack}, giving the asymptotic value of $\delta(C^k_s, n)$ for any $k$ and $s$. To begin, we establish the values of $\gcd(C^k_s)$, $\tau(C^k_s)$ and $\sigma(C^k_s)$.

\begin{prop} \label{cycletype}
For any $k \geq 3$ and $s \geq 2$ we have 
$$\tau(C^k_s) = \sigma(C^k_s) = \frac{\lceil s/2 \rceil}{s(k-1)}.$$
Furthermore, we also have $\gcd(C^k_s) = 1$ except in the case $s=k=3$, whilst $\gcd(C^3_3)$ is undefined.
\end{prop}

\begin{proof}
Note that the vertices $j(k-1)+1$ for $0 \leq j < s$ are the vertices of $C_s^k$ which lie in two edges of $C^k_s$. Let $C$ be the (graph) cycle on these $s$ vertices (in order). Then any proper $k$-colouring of $C$ (as a graph) can be extended to a $k$-partite realisation of $C^k_s$ by colouring the $k-2$ uncoloured vertices of each edge of $C^k_s$ with the $k-2$ colours not used to colour the two coloured vertices. Furthermore, the size of each vertex class $V_i$ is then simply $s$ minus the number of vertices of $C$ with colour $i$. Now, if $s \geq 4$ then we can 3-colour $C$ with $\lfloor s/3 \rfloor$ red vertices, $\lfloor s/3 \rfloor + 1$ green vertices and all remaining vertices blue. Extending this colouring to a $k$-partite realisation of $C^k_s$ as described above, we find that the red and green vertex classes differ in size by one, from which we conclude that $\gcd(C^k_s) = 1$. If instead $s=3$ and $k \geq 4$, then we 3-colour $C$ with one red, one blue and one green vertex; extending this $3$-colouring of $C$ to a $k$-partite realisation of $C^k_s$ we find that the blue, green and red vertex classes each have one fewer vertex than each other vertex class, so again we have $\gcd(C^k_s) = 1$. Finally, if $s=2$, then whilst $C$ is no longer a simple graph, if we colour its two vertices red and blue and then extend this colouring to a $k$-partite realisation of $C^k_s$, we find that the red and blue vertex classes each have one fewer vertex than each other vertex class, again giving $\gcd(C^k_s) = 1$. So $\gcd(C^k_s) = 1$ in any case except for $k=s=3$.

Next observe that any vertex of $C^k_s$ lies in at most two edges of $C^k_s$, so any vertex cover of $C^k_s$ has size at least $\lceil s/2 \rceil$; taking vertices $j(k-1)+1$ for even $0 \leq j < s$ gives a vertex cover of this size. So $\tau(C^k_s) = \lceil s/2 \rceil / s(k-1)$ as claimed. We must have $\sigma(C^k_s) \geq \tau(C^k_s)$ since any vertex class of any $k$-partite realisation of $C^k_s$ is a vertex cover of $C^k_s$. So it remains only to show that  $\sigma(C^k_s) \leq \lceil s/2 \rceil / s(k-1)$, that is, that $C^k_s$ has a $k$-partite realisation in which some vertex class has size $\lceil s/2 \rceil$. For this, observe that we can 3-colour $C$ with $\lfloor s/2 \rfloor$ blue vertices, $\lfloor s/2 \rfloor$ red vertices, and either one or zero green vertices. Extending this colouring to a $k$-partite realisation of $C^k_s$ we find that the blue vertex class has size $s - \lfloor s/2 \rfloor = \lceil s/2 \rceil$, as required. 

Finally, observe that $C^3_3$ has only one $3$-partite realisation up to permutation of the vertex classes $\{1, 4\}, \{2, 5\}, \{3, 6\}$. Since each vertex class has equal size, $\gcd(C^3_3)$ is undefined.
\end{proof}

Except in the case $k=s=3$, Proposition~\ref{cycletype} shows that Theorem~\ref{main2} and Proposition~\ref{extrem2} give asymptotically matching upper and lower bounds on $\delta(C^k_s, n)$. However, since any $3$-partite realisation of $C^3_3$ has two vertices in each vertex class, $C^3_3$ has type $0$, and so our main theorems provide only the bound $\delta(C^3_3, n) \leq n/2 + o(n)$ which applies to all $k$-partite $k$-graphs. However, by modifying the proof of Lemma~\ref{main2simple} we can actually prove that $\delta(C^3_3, n) \leq n/3 + o(n)$, giving the correct asymptotic threshold in this case also. For this we need the following proposition, which allows us to find copies of $C^3_3$ with an odd number of vertices on each side of a partition of $V(H)$. Note that the $k$-graph constructed in Proposition~\ref{extrem1} demonstrates that this proposition does not hold if we replace $C^3_3$ by the $3$-partite $3$-graph $K$ with two vertices in each vertex class; this is the point at which the proof of Theorem~\ref{cyclepack} fails to hold for $K$.

\begin{prop} \label{oddc33}
Suppose that $1/n \ll \alpha$. Let $H$ be a $3$-graph on $n$ vertices with $\delta(H) \geq n/3 + \alpha n$, and suppose that sets $A$ and $B$ partition $V(H)$ and satisfy $|A|, |B| \geq n/3 + \alpha n$. Then there is a copy of $C^3_3$ in $H$ with an odd number of vertices in each of $A$ and $B$.
\end{prop}

\begin{proof}
Suppose for a contradiction that no such copy of $C^3_3$ exists.
There are $|A||B| \geq n^2/9$ pairs $(x, y)$ with $x \in A$ and $y \in B$. Each of these pairs lies in at least $\delta(H) \geq n/3$ edges of~$H$, so there are at least $n^3/81$ edges $e \in H$ which have at least one vertex in each of $A$ and~$B$. So without loss of generality we may assume that there are at least $n^3/200$ edges $e \in H$ with precisely two vertices in $A$. We now `colour' the edges of the complete graph $K[A]$ on $A$ with colours red and blue. Indeed, we colour $xy$ red if there are at least $3$ vertices $w \in B$ with $\{x, y, w\} \in H$, and we colour $xy$ blue if there are at least $6$ vertices $w \in A$ such that $\{x, y, w\} \in H$. So every edge $xy$ receives at least one colour; if both conditions are satisfied then we give $xy$ both colours, meaning that we can treat it as being either colour. Since any pair $xy$ lies in at most $n$ edges, we find that there are at least $(n^3/200 - 2n^2)/n \geq n^2/300$ red edges of $K[A]$.

Observe that there can be no triangle in $K[A]$ with three red edges. Indeed, if $xyz$ is such a triangle then we may choose distinct $w_1, w_2, w_3 \in B$ such that $(x, y, w_1), (x, z, w_2)$ and $(y, z, w_3)$ are each edges of $H$, thus forming a copy of $C^3_3$ with three vertices in $A$ and three in $B$. Similarly, there can be no triangle in $K[A]$ with two blue edges and one red edge, as then we can form a copy of $C^3_3$ with one vertex in $B$ and five in $A$. Now, choose any vertex $x \in A$ which lies in a red edge, and define $A_1 = \{y \in A \sm \{x\} : xy \mbox{ is red}\}$ and $A_2 := A \sm A_1$. So $A_1$ and $A_2$ partition $A$, and by our previous observations no edge of $K[A_1]$ or $K[A_2]$ is red. So all edges of $K[A_1]$ and $K[A_2]$ are blue and not red; it follows that every edge $yz$ with $y \in A_1$ and $z \in A_2$ is red and not blue (so in fact every edge of $K[A]$ has only one colour). So every edge of $K[A]$ must have received a single colour. Moreover the red edges of $K[A]$ form a complete bipartite subgraph of $K[A]$ with vertex classes $A_1$ and $A_2$. Since the number of red edges of $K[A]$ is at least $n^2/300$ it follows that $|A_1|, |A_2| \geq n/300$. Without loss of generality we may assume that $|A_1| \leq |A_2|$, so $|A_1| \leq (n-|B|)/2 \leq n/3 - \alpha n/2$.

Now let $y, z \in A_1$. Then there are at least $\delta(H) \geq n/3 + \alpha n$ vertices $w$ such that $\{w, y, z\} \in H$. At most $n/3 - \alpha n/2$ of these vertices $w$ lie in $A_1$, and since $yz$ is not red at most $2$ of these vertices $w$ lie in $B$. So there are at least $\alpha n$ vertices $w \in A_2$ such that $\{w, y, z\} \in H$; summing over all pairs $y, z \in A_1$ we find that there are at least $\binom{|A_1|}{2}\alpha n \geq 10^{-6} \alpha n^3$ edges of $H$ with two vertices in $A_1$ and one vertex in $A_2$. Since there are $|A_1| |A_2| \leq n^2$ pairs $yz$ with $y \in A_1$ and $z \in A_2$, we deduce that some such pair $yz$ lies in at least $10^{-6} \alpha n \geq 6$ such edges of $H$. But then $yz$ is blue, a contradiction.
\end{proof}

We can now give the proof of Theorem~\ref{cyclepack}. Note that, as mentioned above, the proof follows immediately from Proposition~\ref{cycletype} except in the case $k=s=3$.

\medskip \noindent {\bf Theorem~\ref{cyclepack}.}
{\it
For integers $k \geq 3$ and $s \geq 2$ we have
$$\delta(C^k_s, n) = 
\begin{cases}
\frac{n}{2(k-1)} + o(n) &\mbox{if $s$ is even, and}\\
\frac{s+1}{2s(k-1)}n + o(n) &\mbox{otherwise.} 
\end{cases}$$}

\begin{proof} 
Suppose first that we do not have $k = s = 3$. Then Proposition~\ref{cycletype} shows that $\gcd(C^k_s) = 1$ and $\sigma(C^k_s) = \lceil s/2 \rceil / s(k-1)$, so Theorem~\ref{main2} implies that $\delta(C^k_s, n) \leq \lceil s/2 \rceil n/ s(k-1) + o(n)$. On the other hand, Proposition~\ref{cycletype} also shows that $\tau(C^k_s) = \lceil s/2 \rceil / s(k-1)$, so Proposition~\ref{extrem2} implies that $\delta(C^k_s, n) \geq \lceil s/2 \rceil n/ s(k-1)$. This completes the proof except for the case $k=s=3$. 

Now suppose that $k=s=3$. By Proposition~\ref{cycletype} we have $\tau(C^3_3) = 1/3$, and it follows by Proposition~\ref{extrem2} that $\delta(C^3_3, n) \geq n/3$. So it suffices to prove that $\delta(C^3_3, n) \leq n/3 + o(n)$. That is, we must show that for any $\alpha > 0$ there exists $n_0$ such that for any $n \geq n_0$ which is divisible by 6, any $k$-graph $H$ on $n$ vertices with $\delta(H) \geq n/3 + \alpha n$ contains a perfect $C^3_3$-packing. To do this, let $K$ denote the complete $3$-partite $3$-graph with vertex classes each of size two, and note that any copy of $K$ contains a copy of $C^3_3$ on the same vertex set, so we can treat copies of $K$ in $H$ as being copies of $C^3_3$ for the purpose of finding a $C^3_3$-packing in $H$. We mimic the proof of Lemma~\ref{main2simple} as it would apply to $K$. Indeed, we use the same hierarchy of constants as in the proof of Lemma~\ref{main2simple}, except that we now have the fixed values $b_1=2, b=6$ and $k=3$, and we introduce two additional constants $\beta$ and $\beta'$ with $\psi \ll \beta \ll \beta' \ll 1/C, 1/L$. We proceed exactly as in the proof of Lemma~\ref{main2simple} for most steps of the proof. In the first section of the proof, the calculation of \eqref{eq:GZdeg} now only gives us $\deg_{G \cup Z}(e) \geq n/3 + 2\alpha n/3$, so Lemma~\ref{redgraphmindeg} now yields the weaker result that at most $\theta m^{2}$ pairs $S \in \binom{[m]}{2}$ have $\deg_{\R}(S) \geq m/3 + \alpha m/2$, but this is still sufficient to apply Corollary~\ref{almostmatching} as in the proof of Lemma~\ref{main2simple}. However, at the end of the first section of that proof our weaker minimum codegree condition now only implies that $\delta(\Sa) \geq \delta^+(\Sa^+) > (1/3 + \alpha/2)m'$. This condition is only used in the step `Delete a $K$-packing to ensure equality of cluster sizes', and indeed the rest of the proof proceeds exactly as before until that step (in particular we obtain $n_X \leq \alpha n_1/2$, $n'_1 = n_1 - n_X$, $m' \geq (1-\psi)m$, $K$-packings $M^1$ and $M^2$, and sets $X_j$ and $Y'_j$ as there). 

In this step we wish to find a $C^3_3$-packing $M^3$ in $H[Y']$ so that for any $e \in M_\R$ the sets $Y'_j$ for $j \in e$ have equal size, and this common size is even. We demonstrate below how this can be done, but first note that once we have found such a $C^3_3$-packing $M^3$ in $H[Y']$, the final step of the proof proceeds exactly as before to give a perfect $C^3_3$-packing in $H$.

Suppose first that for any set $T \subseteq V(\Sa)$ with $|T| \leq \alpha m'/6$ the graph $\Sa \sm T$ formed by deleting the vertices of $T$ from $\Sa$ is connected (as was the case in the proof of Lemma~\ref{main2simple} as a consequence of our stronger bound on $\delta(\Sa)$). Then $\delta(\Sa \sm T) \geq \delta(\Sa) - \alpha m'/6 > m'/3$, and for any $x, y \in V(\Sa) \sm T$ there is a path from $x$ to $y$ in $\Sa \sm T$. It follows that for any $x, y \in V(\Sa) \sm T$ the shortest path in $\Sa \sm T$ from $x$ to $y$ has length at most five, as in a path of length six or more two of the first, fourth and seventh vertices must share a common neighbour in $\Sa \sm T$, allowing us to construct a shorter path. We can therefore apply Lemma~\ref{samesizebalance} exactly as in the proof of Lemma~\ref{main2simple}, with the trivial partition of $V(\Sa)$ into one set $X_1 = V(\Sa)$, except that now we have $\alpha/6$ in place of $\alpha$ (condition (i) of Lemma~\ref{samesizebalance} then follows by our remarks above, and all other conditions hold exactly as before).

We may therefore assume that there exists a set $T \subseteq V(\Sa)$ with $|T| \leq \alpha m'/6$ for which the graph $\Sa' := \Sa \sm T$ is disconnected.
Observe that $\delta(\Sa') \geq \delta(\Sa) - \alpha m'/6 \geq m'/3 + \alpha m'/3$, so $\Sa'$ has precisely two connected components $C_1$ and $C_2$, each with at least $m'/3 + \alpha m'/3$ vertices and so at most $2m'/3 - \alpha m'/3$ vertices. In particular, any vertices $x, y \in V(C_1)$ have at least $\alpha m'$ neighbours in $V(C_1)$, and the analogous condition holds for $V_2$. Furthermore, every vertex $v \in T$ has at least $m'/3 + \alpha m'/2$ neighbours in $\Sa$, and so has either at least $m'/6$ neighbours in $V(C_1)$ or at least $m'/6$ neighbours in $V(C_2)$. Form a set $V_1$ by adding to $V(C_1)$ all those vertices of $T$ with at least $m'/6$ neighbours in $V(C_1)$, and let $V_2 = V(\Sa) \sm V_1$, so $V(C_2) \subseteq V_2$ and every vertex of $T \cap V_2$ has at least $m'/6$ neighbours in $V(C_2)$. Then $V_1$ and $V_2$ partition $[m']$, and $m'/3 + \alpha m'/3 \leq |V_1|, |V_2| \leq 2m'/3 - \alpha m'/3$. Furthermore, writing $\Sa_1 := \Sa[V_1]$, for any vertices $x, y \in V_1$ there is a path from $x$ to $y$ of length at most four in $\Sa_1$ from $x$ to $y$, even after the deletion of at most $\alpha m'/6$ vertices other than $x$ and $y$. Indeed, even following this deletion, $x$ must have a neighbour $w$ in $V(C_1)$, and $y$ must have a neighbour $z$ in $V(C_1)$, and then $z$ and $w$ have a common neighbour in $V(C_1)$ by our observation above that $z$ and $w$ have at least $\alpha m'$ common neighbours in $\Sa'$. A similar argument shows that if we delete at most $\alpha m'/6$ vertices of $\Sa_2 := \Sa[V_2]$ then there is a path of length at most four between any two vertices in the remaining subgraph.

Now recall from the proof of Lemma~\ref{main2simple} that every set $Y'_j$ has size $(1-3\sqrt{\psi})n_1' \leq |Y'_j| \leq n'_1$. In particular, it follows that 
\begin{equation} \label{eq:Y'size}
n-2\alpha n/3 \leq (1-3\sqrt{\psi})(1-\alpha/2)n_1(1-\psi)m \leq (1-3\sqrt{\psi}) n_1'm' \leq |Y'| \leq n_1'm' \leq n.
\end{equation}
Fix an integer $Q$ which is divisible by $6$ such that $(1-\beta)n'_1 \leq Q \leq (1-\beta + \psi) n'_1$.

\begin{claim} \label{cycleclaim}
There exists a $C^3_3$-packing $M^*$ in $H[Y']$ of size
$$N := \frac{|Y'| - m'Q}{6}$$ so that $|V(M^*) \cap Y'_j| \leq \beta' n'_1$ for any $j \in [m']$ and $|V(M^*) \cap \bigcup_{j \in V_1} Y_j'| = \sum_{j \in V_1} |Y'_j| - |V_1| Q$.
\end{claim}
 
To prove the claim, we first observe that $N$ is an integer, since $Y'$ consists of all vertices of $H$ except those in $X$ (which was chosen in the proof of Lemma~\ref{main2simple} to have size divisible by $b=6$) and those covered by the $K$-packing $M^1 \cup M^2$, so $|Y'|$ is divisible by 6, and we chose $Q$ to be divisible by 6. Also, we have $(\beta - \beta^2)m'n'_1 \leq 6N \leq \beta m'n'_1$ by \eqref{eq:Y'size} and our choice of~$Q$. Write $A := \bigcup_{j \in V_1} Y_j'$ and $B := \bigcup_{j \in V_2} Y_j'$. Since $|Y'| \geq n-2\alpha n/3$ by \eqref{eq:Y'size}, we have $\delta(H[Y']) \geq (1/3 + \alpha/3)|Y'|$, so, by three successive applications of Proposition~\ref{oddc33}, we may obtain a set $E_0$ of three vertex-disjoint copies of $C^3_3$ in $H[Y']$, each of which has an odd number of vertices in each of $A$ and $B$. 

Now recall that all but at most $\theta m^2$ pairs $S \in \binom{[m]}{2}$ had $\deg_\R(S) \geq m/3 + \alpha m/2 \geq m'/3 + \alpha m'/2$. Since $m' \geq (1-\psi)m$ and $|V_1|, |V_2| \geq m'/3$ we may greedily form a set of $m'/10$ disjoint pairs $(x, y)$ with $x \in V_1, y \in V_2$ and $\deg_\R(\{x, y\}) \geq m'/3$. So certainly either there are at least $m'/10$ edges of $\R$ with precisely two vertices in $V_1$ or there are at least $m'/10$ edges of $\R$ with precisely two vertices in $V_2$. Without loss of generality we assume the former, that there is a set $F_1$ of $m'/10$ vertex-disjoint edges of $\R$ each with precisely two vertices in $V_1$. The same argument for pairs $(x, y)$ with $x, y \in V_2$ shows that there must be a set $F_2$ of $m'/10$ vertex-disjoint edges of $\R$ each with at least two vertices in $V_2$. 
Using Lemma~\ref{getrobuni} exactly as Lemma~\ref{getphirobuni} was used in the proof of Lemma~\ref{main2simple}, we can obtain at least $\beta' n_1'/5$ vertex-disjoint copies of $K$ in $H'[\bigcup_{j \in e} Y'_e]$ for each edge $e \in F_1 \cup F_2$. So we can greedily form $K$-packings $E_1$ and $E_2$ in $H[Y']$ each of size $(\beta' n'_1/5)(m'/40) = \beta' n'_1m'/200 \geq N$ such that $V(E_1) \cap V(E_2) \cap V(E_3) = \emptyset$, every copy of $K$ in $E_1$ has precisely four vertices in $A$, every copy of $K$ in $E_2$ has either four or six vertices in $B$, and collectively $E_1$, $E_2$ and $E_3$ cover at most $4 \beta' n_1/5 + 24 \leq \beta' n_1'$ vertices in any set $Y'_j$. 

Recall that we want $M^*$ to cover $|A| - |V_1|Q$ vertices of $A$. Initially take $M^*$ to consist precisely $N$ edges taken from $E_0$ and $E_1$. By choosing an appropriate subset of the edges of $E_0$ to include, we can ensure that $|V(M^*) \cap A| - (|A|-|V_1|Q)$ is divisible by four. However, this initial selection of $M^*$ has 
$$|V(M^*) \cap A| \geq 4(N-3) \geq \frac{2(\beta-\beta^2) m'n'_1}{3} - 12 > \left(\frac{2}{3}-\frac{\alpha}{3}\right)\beta m'n_1' \geq |A| - |V_1|Q,$$ 
so too many vertices are taken from $A$ (the final inequality holds since $|A| \leq n_1'|V_1|$ and $|V_1| \leq (2/3 - \alpha/3)m'$). On the other hand, if we were to replace all edges of $E_1$ in $M^*$ by edges of $E_2$, then a similar calculation would show that $|V(M^*) \cap A| < |A| - |V_1|Q$, that is, that too few vertices are taken from $A$. Starting from our initial $M^*$, we repeatedly replace an edge $e \in E_1$ in $M^*$ by an edge $e' \in E_2$, beginning with those edges $e' \in E_2$ with $|e' \cap B| = 6$, and then using edges $e' \in E_2$ with $|e' \cap B| = 4$ if these run out. Then each replacement by an edge $e'$ with $|e' \cap B| = 6$ decreases $|V(M^*) \cap A|$ by four, and each replacement by an edge $e'$ with $|e' \cap B| = 4$ decreases $|V(M^*) \cap A|$ by two. Since our initial $M^*$ was chosen so that $|V(M^*) \cap A| - (|A| - |V_1|Q)$ was divisible by four, at some point in this process we must have $|V(M^*) \cap A| = |A| - |V_1|Q$. By our choice of $E_1$ and $E_2$, at this point we also have $|V(M^*) \cap Y'_j| \leq \beta' n_1'$ for any $j \in [m']$, so this $M^*$ is the desired $C^3_3$-packing. This completes the proof of Claim~\ref{cycleclaim}.

\medskip

Retuning to the proof of Theorem~\ref{cyclepack}, fix some $M^*$ as in Claim~\ref{cycleclaim}, and delete the vertices covered by $M^*$ from $H$. Having done so, the sets $Y^*_j := Y'_j \sm V(M^*)$ for each $j \in [m']$ of undeleted vertices satisfy 
\begin{equation} \label{eq:sizes}
(1-2\beta')n'_1 \leq |Y'_j| - \beta' n'_1 \leq |Y^*_j| \leq |Y'_j| \leq n'_1
\end{equation}
for any $j \in [m']$, and 
\begin{equation} \label{eq:sameaverage}
Q = \frac{|\bigcup_{i \in V_1} Y^*_i|}{|V_1|} = \frac{|\bigcup_{i \in V_2} Y^*_2|}{|V_2|}.
\end{equation}
where for the second equality we used the fact that $M^*$ covered $6N = |Y'| - m'Q$ vertices in total, so, since $V_1$ and $V_2$ partition $[m']$, we have
$$|V(M^*) \cap \bigcup_{j \in V_2} Y_j'| = |Y'| - m'Q - \sum_{j \in V_1} |Y'_j| + |V_1| Q= \sum_{j \in V_2} |Y'_j| - |V_2| Q.$$
We now apply Lemma~\ref{samesizebalance} similarly as in Lemma~\ref{main2simple}, but now with $V_1$ and $V_2$ in place of $X_1$ and $X_2$, giving a partition of $[m']$ into two parts. Again we have $n_1-n_X$ and $m'$ in place of $n$ and $m$ respectively and $M_\R$ and $\Sa$ play the same role there as here, but we now have the sets $Y^*_j, H'[Y^*]$, $\alpha/6$ and $2\beta'$ in place of $Y_j$, $G$, $\alpha$ and $\beta$ respectively, and we set $b_1=2$ and $k=3$. The sets $Y^*_j$ satisfy the size condition of Lemma~\ref{samesizebalance} by \eqref{eq:sizes}, and condition (iii) of Lemma~\ref{samesizebalance} then holds by \eqref{eq:sameaverage} (with $Q$ playing the same role here as there). The conditions of Setup~\ref{balancesetup} are satisfied exactly as in the proof of Lemma~\ref{main2simple}, since these do not depend on the sets $Y^*_j$. Likewise condition (ii) of Lemma~\ref{samesizebalance} holds by our choice of $M_\R$ exactly as in the proof of Lemma~\ref{main2simple}. Finally, condition~(i) of Lemma~\ref{samesizebalance} holds by our comments on $\Sa_1$ and $\Sa_2$ prior to the statement of Claim~\ref{cycleclaim}. 

So we may indeed apply Lemma~\ref{samesizebalance} as claimed, to obtain a $K$-packing $M^{**}$ in $H'[Y^*]$ such that for any $e \in M_\R$ the sets $Y''_j := Y^*_j \sm V(M^{**}) = Y'_j \sm V(M^* \cup M^{**})$ for $j \in e$ obtained by these deletions have equal size $n_e$, where $n_e$ is divisible by $b_1=2$. So we may take $M^3 = M^* \cup M^{**}$, following which, as stated earlier, the final step of the proof proceeds exactly as in the proof of Lemma~\ref{main2simple}.
\end{proof}

\section{Concluding remarks}\label{sec:conc}

\subsection{Non-complete $k$-partite $k$-graphs}
In Section~\ref{sec:extremal} we saw that Theorems~\ref{main1},~\ref{main2} and~\ref{main3} are asymptotically best possible for all complete $k$-partite $k$-graphs; we now consider those incomplete $k$-partite $k$-graphs for which the same is true. For brevity, throughout the following discussion the words `best possible' should be interpreted as meaning best possible up to an $o(n)$ error term.
 
Recall that Proposition~\ref{extrem2} showed that Theorem~\ref{main2} is best possible for all $k$-partite $k$-graphs $K$ of type $1$ with $\tau(K) = \sigma(K)$, and also for all $k$-partite $k$-graphs $K$ of type $d \geq 2$ with $\tau(K) = \sigma(K) \geq 1/p$, where $p$ is the smallest prime factor of $d$. In particular, any $k$-partite $k$-graph on $b$ vertices which contains a matching of size $\sigma(K) b$ (that is, a matching which covers an entire vertex class of some $k$-partite realisation) has $\sigma(K) = \tau(K)$.

However, there are $k$-partite $k$-graphs for which $\tau(K) \neq \sigma(K)$. For an example, choose disjoint sets $U_1, \dots, U_k$ each of size greater than $k$, and for each $j \in [k]$ choose a marked vertex $u_j \in U_j$. Let $K^*$ be the $k$-partite $k$-graph with vertex classes $U_1, \dots, U_k$ whose edges are all $k$-tuples of vertices including at least one of the marked vertices $u_j$.
Then it is not hard to see that $K^*$ has only one $k$-partite realisation up to permutations of the vertex classes $U_1, \dots, U_k$, so $\sigma(K^*) = \min_{j \in [k]} |U_j|/b > k/b$, where $b := |V(K^*)|$. However, $\{u_j : j \in [k]\}$ is a vertex cover of $K^*$, so we have $\tau(K^*) \leq k/b < \sigma(K^*)$. Having seen the construction of Proposition~\ref{extrem2}, it is natural to ask whether the best possible versions of Theorems~\ref{main2} and~\ref{main3} would have $\tau(K)$ in place of $\sigma(K)$, providing an upper bound to match the lower bound of Proposition~\ref{extrem2}. However, we can show this is not the case by considering the $k$-graph $K^*$ defined above and the following variation of the construction of Proposition~\ref{extrem2}. Take disjoint empty vertex sets $A$ and $B$ with sizes $|A| = \lceil(k+1)n/b\rceil - 1$ and $|B| = n - |A|$, and let $H$ be the $3$-graph on vertex set $A \cup B$ whose edges are all $3$-tuples $e \in \binom{A \cup B}{3}$ with $1 \leq |e \cap A| \leq k-1$. Then $\delta(H) = |A| = \lceil(k+1)n/b \rceil -1$, but $H$ does not have a perfect $K^*$-packing. This is because for any copy of $K^*$ in $H$, $A \cap V(K^*)$ is a vertex cover of $K^*$. However, we cannot have $u_j \in A$ for every $j \in [k]$ because $\{u_1, \dots, u_k\}$ is an edge of $K^*$. Any other vertex cover of $K^*$ has size at least $k+1$, so any copy of $K^*$ in $H$ must have at least $k+1$ vertices in $A$. It follows that any $K^*$-packing in $H$ has size at most $|A|/(k+1) < n/b$, so is not perfect. It is therefore not obvious what should be the general rule for the behaviour of $\delta(K, n)$ for $k$-partite $k$-graphs $K$ of type 1 which have $\tau(K) \neq \sigma(K)$.

We now consider the divisibility-based extremal constructions, for which we make the following definitions. Let $K$ be a $k$-partite $k$-graph; then we say that two vertices $u, v$ of $K$ are \emph{tightly-linked} if there is a set $S$ of $k-1$ vertices of $K$ such that $u \cup \{S\}$ and $v \cup \{S\}$ are both edges of $K$. Observe that $u$ and $v$ must then lie in the same vertex class of any $k$-partite realisation of $K$. Now let $U_1, \dots, U_k$ be the vertex classes of a $k$-partite realisation of $K$, and form a graph on $V(K)$ by joining any tightly-linked pair of vertices with an edge. So each connected component of this graph is a subset of a vertex class of $K$; we say that $K$ is \emph{tightly $k$-partite} if the connected components of this graph are precisely the vertex classes of $K$. So, for example, any complete $k$-partite $k$-graph is tightly $k$-partite. Note that if $K$ is tightly $k$-partite then it has only one $k$-partite realisation up to permutation of the vertex classes (but the converse does not hold: for example, $C^3_3$ has only one $k$-partite realisation but is not tightly $k$-partite).

Recall that Proposition~\ref{extrem1} showed that Theorem~\ref{main1} is best possible for any $k$-partite $k$-graph $K$ which satisfies property (P1) for some $p$, and that Theorem~\ref{main3} is best possible for any $k$-partite $k$-graph $H$ which satisfies property (P2) for some $p \geq 2$ with $1/p \geq \sigma(H)$. In Section~\ref{sec:extremal} we showed that property (P1) holds for any complete $k$-partite $k$-graph of type~0; the same argument shows that property (P1) holds for any tightly $k$-partite $k$-graph $K$ of type 0. Indeed, for any $i \in [k]$ and any $x, y \in U_i$, the fact that $H$ is tightly $k$-partite implies that there is a sequence $x, z_1, \dots, z_\ell, y$ such that each consecutive pair is tightly-linked in $K$, and so, given a set $A \subseteq V(K)$ such that $|e \cap A|$ is divisible by $p$ for any $e \in K$, the vertices $x, z_1, \dots, z_\ell, y$ are either all in $A$ or all not in $A$. Property (P1) then follows exactly as before. Similarly, in Section~\ref{sec:extremal} we showed that property (P2) holds for any complete $k$-partite $k$-graph of type $d \geq 2$ and any $p$ which divides $d$. Exactly as before we can adapt this argument to show that the same statement holds with `tightly $k$-partite' in place of `complete $k$-partite'. In conclusion, we find that Theorem~\ref{main1} is best possible for any tightly $k$-partite $k$-graph of type 0, and that Theorem~\ref{main3} is best possible for any tightly $k$-partite $k$-graph $K$ of type $d \geq 2$ such that the smallest prime factor $p$ of $d$ satisfies $1/p \geq \sigma(K)$. Again, it is not clear what the general rule should be for the behaviour of $\delta(K, n)$ for $k$-partite $k$-graphs $K$ of type $0$ or $d \geq 2$ which are not tightly $k$-partite. One example of the latter category is the cycle $C^3_3$, which has type $0$ and for which Theorem~\ref{cyclepack} showed that $\delta(C^3_3, n) = n/3 + o(n) = \sigma(C^3_3)n + o(n)$. The principal difference between the proofs of Theorem~\ref{cyclepack} and Lemma~\ref{main2simple} was the use of Proposition~\ref{oddc33} to find copies of $C^3_3$ with an odd number of vertices on each side of a partition of $V(H)$; it seems likely that the value of $\delta(K, n)$ for $k$-partite $k$-graphs $K$ of type $0$ or type $d \geq 2$ which are not tightly $k$-partite depends on the minimum codegree required to give an analogue of Proposition~\ref{oddc33}, in a manner reminiscent of the results of~\cite{KKM} for perfect matchings. 

\subsection{Almost-perfect packings} For many applications, it suffices to find an almost-perfect $K$-packing in a $k$-graph $H$, that is, a $K$-packing covering all but a small number of vertices. It is natural to consider the minimum codegree needed to guarantee such a packing, and this can be obtained by small changes to our methods. First, observe that if in Proposition~\ref{extrem2} we instead take the set $A$ to have size $\lceil \tau (n - C) \rceil - 1$, where $\tau = \tau(K)$, then we obtain a $k$-graph $G$ on $n$ vertices with $\delta(G) = |A| = \lceil \tau (n - C) \rceil - 1$ in which any $K$-packing has size at most $|A|/\tau b < (n-C)/b$ (where $b$ is the number of vertices of $K$). So any $K$-packing in $G$ leaves more than $C$ vertices uncovered, proving the following proposition.

\begin{prop} \label{almostpackex}
Let $K$ be a $k$-partite $k$-graph. Then for any $C > 0$ and any $n$ there exists a $k$-graph $G$ on $n$ vertices with $\delta(G) \geq \lceil\tau(K) (n - C)\rceil -1$ such that no $K$-packing in $G$ covers all but at most $C$ vertices of $G$.
\end{prop}

In the other direction, we now prove Theorem~\ref{almostpack}, which gives an upper bound similar to that of Theorem~\ref{main2} but which holds for all $k$-partite $k$-graphs $K$ (whereas Theorem~\ref{main2} applied only when $\gcd(K) = 1$). Combined with Proposition~\ref{almostpackex} this gives the asymptotically correct threshold for the almost-packing problem for any $k$-partite $k$-graph $K$ with $\tau(K) = \sigma(K)$; as we have seen, this includes all complete $k$-partite $k$-graphs.

\medskip \noindent {\bf Theorem~\ref{almostpack}.}
{\it
Let $K$ be a $k$-partite $k$-graph. Then there exists a constant $C = C(K)$ such that for any $\alpha > 0$ there exists $n_0 = n_0(K, \alpha)$ such that any $k$-graph $H$ on $n \geq n_0$ vertices with $\delta(H) \geq \sigma(K)n + \alpha n$ admits a $K$-packing covering all but at most $C$ vertices of $H$.} \medskip

\begin{proof}
Suppose first that $\sigma(K) < 1/k$, and fix $U_1, \dots, U_k$ to be the vertex classes of a $k$-partite realisation of $K$ in which $|U_1| = \sigma(K) b$, where $b := |V(K)|$. Note that since $\sigma(K) < 1/k$, the vertex classes $U_1, \dots, U_k$ cannot all be the same size. Let $K^*$ be the complete $k$-partite $k$-graph on these vertex classes; then it suffices to find a perfect $K^*$-packing in $H$. To do this, we follow the proof of Lemma~\ref{main1simple} exactly as in the case when $\gcd(K) = 1$ until the step `Delete a $K$-packing to ensure divisibility of cluster sizes', in which we applied Lemma~\ref{gcdbalance} to delete a $K^*$-packing in $H$ such that following this deletion all subclusters had size divisible by $bk\gcd(K^*)$. We cannot now apply Lemma~\ref{gcdbalance} as we may have $\gcd(K^*) \neq 1$; however, it is straightforward to modify (and hugely simplify) the argument of Lemma~\ref{gcdbalance} to show that we can delete a $K^*$-packing in $H$ such that, following this deletion, at most one subcluster in any connected component of $\Sa'$ does not have size divisible by $bk\gcd(K^*)$. Our condition $\delta(H) \geq \sigma(K)n + \alpha n \geq n/b + \alpha n$ implies that $\Sa'$ has fewer than $b$ connected components (exactly as $\delta(H) \geq n/p+\alpha n$ implied that $\Sa'$ had fewer than $p$ connected components). So at most $b$ subclusters do not have size divisible by $bk\gcd(K^*)$; by deleting up to $bk\gcd(K^*)$ vertices from each of these we can ensure that every subcluster has size divisible by $bk\gcd(K^*)$. Following the remainder of the proof exactly as before, we obtain a $K^*$-packing in $H$ which covers all vertices except for the at most $b^2k\gcd(K^*) \leq b^3k$ deleted vertices, so we may take $C = b^3k$.

The remaining possibility is that $\sigma(K) = 1/k$, in which case we have $\delta(H) \geq n/k + \alpha n$. Since $\sigma(K) = 1/k$ every vertex class of any $k$-partite realisation of $K$ must have equal size, so we may assume without loss of generality that $K$ is the complete $k$-partite $k$-graph with vertex classes each of size $b_1$, for some positive integer $b_1$. In this case we mimic the proof of Lemma~\ref{main2simple}, and as in the proof of Theorem~\ref{cyclepack}, all steps proceed as before except for the choice of the $K$-packing $M^3$. Using the notation of the proof of Lemma~\ref{main2simple}, we proceed as in the proof of Theorem~\ref{cyclepack} to find $M^3$, first obtaining a partition of $[m']$ into parts $V_1, \dots, V_s$ for some $s < k$ such that for any $j \in [s]$ we have $|V_j| \geq m'/k + \alpha m'/2k$ and the property that, even after the deletion of up to $\alpha m'/2k$ vertices of $\Sa[V_j]$, the remaining subgraph contains a short path between any two vertices. We cannot then obtain $M^*$ as in Claim~\ref{cycleclaim}, since we now have no analogue of Proposition~\ref{oddc33}, but instead we can simply delete up to $2kb$ vertices of $\bigcup_{i \in V_j} Y'_i$ for each $j \in [s]$ before proceeding similarly as the proof of Claim~\ref{cycleclaim} to obtain a $K$-packing $M^*$ such that the sets $Y^*_i$ of vertices of $Y'_i$ which were not deleted or covered by $M^*$ have the desired property, that is, that the average size of sets $Y^*_i$ with $i \in V_j$ is the same for each $j \in [s]$, and divisible by $b_1$. From this we obtain $M^3$ exactly as in the proof of Theorem~\ref{cyclepack}, and the final step of the proof then proceeds exactly as for Lemma~\ref{main2simple}, giving a $K$-packing in $H$
covering all vertices except for the at most $2ksb < k^2b$ deleted vertices, so we may take $C = k^2b$.
\end{proof} 

\subsection{Non $k$-partite $k$-graphs}
All of the results of this paper pertain to $k$-partite $k$-graphs, but we can also consider the value of $\delta(H, n)$ for $k$-graphs $H$ which are not $k$-partite. However, for most such $k$-graphs $H$ we do not even know the asymptotic value of the Tur\'an density, that is, the number of edges needed in a large $k$-graph $G$ to guarantee even a single copy of $H$ in $G$. By contrast, in this paper we used several times the fact that the Tur\'an density of a $k$-partite $k$-graph is zero. A lack of knowledge of the Tur\'an density of $H$ is not an essential obstacle to finding $\delta(H, n)$; indeed, Keevash and Mycroft determined the exact value of $\delta(K^3_4, n)$ for large $n$ even though finding the Tur\'an density of $K^3_4$ remains a significant open problem. It would be interesting to know whether $\delta(H, n)$ is determined by the parameters $\gcd(H)$ and $\sigma(H)$ (whose definitions extend naturally to the non $k$-partite case), in the same way as for graphs. However, we note that it is not sufficient to consider only the smallest $r$ for which a $k$-graph $H$ admits an $r$-partite realisation. Indeed $K^3_4 - e$ and $K^3_4$ both admit a $4$-partite realisation with one vertex in each vertex class, and no $3$-partite realisation, but as we saw in Section~\ref{intro}, results of Lo and Markstr\"om~\cite{LMp, LM2} and of Keevash and Mycroft~\cite{KM} show that $\delta(K^3_4 - e, n)$ and $\delta(K^3_4, n)$ differ by $n/4 + o(n)$.

A natural starting point to consider would be the complete $4$-partite $3$-graph $K$ with vertex classes of size $b_1$, $b_2$, $b_3$ and $b_4$ (so the edges are any triple whose vertices lie in three different vertex classes). The fact that $\delta(K^3_4, n) \leq 3n/4$ strongly suggests that we have $\delta(K, n) \leq 3n/4 + o(n)$ for any values of $b_1, b_2, b_3$ and $b_4$. On the other hand, if $\gcd(\{b_1, b_2, b_3, b_4\}) > 1$, then we can modify a construction of Pikhurko~\cite{P} to show that $\delta(K, n) \geq 3n/4 - 2$, and in fact the same construction gives the same threshold if $\gcd(K) = 2$, where $\gcd(K) := \gcd(\{b_i - b_j: i, j \in [4]\}) > 1$. So we expect that $\delta(K, n) = 3n/4 + o(1)$ in these cases. If $\delta(K, n)$ exhibits behaviour analogous to that of $k$-partite $k$-graph case, then we would also expect $\delta(K, n)$ to be lower in the case $\gcd(K) = 1$. Assume without loss of generality that $b_1 \leq b_2, b_3, b_4$ and define $\sigma = \sigma(K) := \frac{b_1}{b_1+b_2+b_3+b_4}$. Given disjoint sets $A$ and $B$ with $|A| = \lceil \sigma n \rceil - 1$ and $|B| = n - |A|$, a well-known random construction due to Czygrinow and Nagle~\cite{CN} gives a $k$-graph $G'$ on $|B|$ with $\delta(G') \geq |B|/2 - o(n)$ which does not contain any copy of $K^3_4$. Form a $k$-graph $G$ on $A \cup B$ whose edges are the edges of $G'$ together with any triple of vertices which intersects~$A$. Then $\delta(G) = \delta(G') + |A| = n/2 + |A|/2 - o(n) = (1/2+\sigma)n - o(n)$, and $G$ cannot contain a perfect $K$-packing, as any copy of $K$ in $G$ must have at least $b_1$ vertices in $A$. So we have $\delta(K, n) \geq (1/2+\sigma)n - o(n)$; it would be interesting to know whether this is the asymptotically correct threshold for $\gcd(K) = 1$. However, demonstrating that this is the correct threshold would imply that the minimum codegree which guarantees the existence of a copy of $K^3_4$ in a $k$-graph on $n$ vertices is asymptotically $n/2$, so would require the solution of a well-known open problem (see~\cite{CN} for further details). 

\section{Acknowledgements}

I would like to thank Katherine Staden for her helpful comments on an earlier version of this manuscript. I also thank an anonymous referee for many helpful remarks.

\medskip

\noindent
{\footnotesize
Richard Mycroft, \\ 
School of Mathematics, \\ 
University of Birmingham, \\
Birmingham B15 2TT, \\
United Kingdom. \\
{\tt r.mycroft@bham.ac.uk }}
\end{document}